\documentclass[11pt,a4paper,notitlepage]{article}
\usepackage[page]{appendix}
\usepackage{amsmath}
\usepackage[affil-it]{authblk}
\usepackage{geometry}
\usepackage{amsthm}
\usepackage{amssymb}
\usepackage{enumerate}
\usepackage{setspace}
\usepackage{tabularx}
\usepackage{multirow}
\usepackage{lipsum}

\newtheorem{Lemma}      {Lemma} [section]
\newtheorem{Theorem}    [Lemma] {Theorem}

\newtheorem*{Theorem2}    {Theorem 1.8}

\newtheorem{Corollary}  [Lemma] {Corollary}
\newtheorem{Proposition}[Lemma] {Proposition}

\theoremstyle{definition}
\newtheorem{Definition} [Lemma] {Definition}
\newtheorem*{Definition*}{Definition A.0.1}
\theoremstyle{definition}
\newtheorem{Remark} [Lemma] {Remark}

\newtheorem{Example}[Lemma] {Example}
\numberwithin{equation}{section}
\newcommand{\core}{\mathop{\mathrm{core}}}

\newcommand{\Stab}{\mathop{\mathrm{Stab}}}
\newcommand{\Rad}{\mathop{\mathrm{Rad}}}
\newcommand{\ab}{{a_{ab}}}
\newcommand{\anab}{{a_{nonab}}}
\newcommand{\nab}{{c_{nonab}}}
\newcommand{\Sym}{\mathop{\mathrm{Sym}}}

\newcommand{\Aut}{\mathop{\mathrm{Aut}}}
\newcommand{\Hom}{\mathop{\mathrm{Hom}}}
\newcommand{\Ker}{\mathop{\mathrm{Ker}}}
\newcommand{\Alt}{\mathop{\mathrm{Alt}}}
\newcommand{\bl}{\mathop{\mathrm{bl}}}
\newcommand{\HS}{\mathop{\mathrm{HS}}}
\newcommand{\ML}{\mathop{\mathrm{McL}}}
\newcommand{\Co}{\mathop{\mathrm{Co}}}
\newcommand{\Out}{\mathop{\mathrm{Out}}}

\newcommand{\Sub}{\mathop{\mathrm{Sub}}}
\newcommand{\Inn}{\mathop{\mathrm{Inn}}}
\newcommand{\lpp}{\mathop{\mathrm{lpp}}}

\newcommand{\as}{{a_{prim}}}

\newcommand{\ws}{{\widetilde{\omega}}}

 \DeclareMathOperator{\Sp}{Sp}

\DeclareMathOperator{\PSp}{PSp} \DeclareMathOperator{\Ps}{P}

 \newcommand{\ol}{\overline}

\numberwithin{equation}{section}
\begin{document}
\title{Minimal generation of transitive permutation groups}
\author{Gareth M. Tracey%
\thanks{Electronic address: \texttt{G.M.Tracey@warwick.ac.uk}}} 
\affil{Mathematics Institute, University of Warwick,\\Coventry CV4 7AL, United Kingdom}
\date{October 30, 2017}
\maketitle
\begin{abstract} This paper discusses upper bounds on the minimal number of elements $d(G)$ required to generate a transitive permutation group $G$, in terms of its degree $n$, and its order $|G|$. In particular, we reduce a conjecture of L. Pyber on the number of subgroups of the symmetric group $\Sym(n)$. We also prove that our bounds are best possible.\end{abstract}
\section{Introduction}
A well-developed branch of finite group theory studies properties of certain classes of permutation groups as functions of their degree. The purpose of this paper is to study the minimal generation of transitive permutation groups.

For a group $G$, let $d(G)$ denote the minimal number of elements required to generate $G$. In \cite{KovNew}, \cite{BKR}, \cite{Luc} and \cite{LucMenMor}, it is shown that $d(G)=O(n/\sqrt{\log{n}})$ whenever $G$ is a transitive permutation group of degree $n\ge 2$ (here, and throughout this paper, $``\log"$ means $\log$ to the base $2$). A beautifully constructed family of examples due to L. Kov\'{a}cs and M. Newman shows that this bound is `asymptotically best possible' (see Example \ref{TransExample}), thereby ending the hope that a bound of $d(G)=O(\log{n})$ could be proved. 

The constants involved in these theorems, however, were never estimated. We prove: 
\begin{Theorem}\label{TransTheorempre} Let $G$ be a transitive permutation group of degree $n\ge 2$. Then\begin{enumerate}[(1)]
\item $d(G)\le \left\lfloor \frac{cn}{\sqrt{\log{n}}}\right\rfloor,$where $c:=1512660\sqrt{\log{(2^{19}15)}}/(2^{19}15)=0.920581\hdots$, and;
\item $d(G)\le \left\lfloor \frac{c_{1}n}{\sqrt{\log{n}}}\right\rfloor,$
where $c_{1}:=\sqrt{3}/2=0.866025\hdots$, unless each of the following conditions hold:\begin{enumerate}[(i)]
\item $n=2^{k}v$, where $v=5$ and $17\le k\le 26$, or $v=15$ and $15\le k\le 35$, and; 
\item $G$ contains no soluble transitive subgroups.\end{enumerate}\end{enumerate}
\end{Theorem}
In fact, we prove a slightly stronger version of Theorem \ref{TransTheorempre}, which is given as Theorem \ref{TransTheorem}. The following corollary is immediate.
\begin{Corollary}\label{TransTheoremCor2} Let $G$ be a transitive permutation group of degree $n$, containing a soluble transitive subgroup. Then
$$d(G)\le \left\lfloor \frac{c_{1}n}{\sqrt{\log{n}}}\right\rfloor,$$
 where $c_{1}=\sqrt{3}/2$.\end{Corollary}

As shown in \cite{KovNew}, apart from the choice of constants, the bounds in our results are of the right order. Moreover, the infimum of the set of constants $\overline{c}$ satisfying $d(G)\le \overline{c}n/\sqrt{\log{n}}$, for all soluble transitive permutation groups $G$ of degree $n\ge 2$, is the constant $c_{1}$ in Theorem \ref{TransTheorem1}, since $d(G)=4$ when $n=8$ and $G\cong D_{8}\circ D_{8}$. We conjecture that the best `asymptotic' bound, that is, the best possible upper bound when one is permitted to exclude finitely many cases, is 
$$d(G)\le \left\lfloor \frac{\widetilde{c}n}{\sqrt{\log{n}}}\right\rfloor,$$ where $\widetilde{c}$ is some constant satisfying 
$$b/2\le \widetilde{c}<b:=\sqrt{2/\pi}$$ (see Example \ref{TransExample} for more details). 

The constant $b$, and the function $n/\sqrt{\log{n}}$, enter our work by means of the following combinatorial result. For a partially ordered set $P$, $w(P)$ denotes the \emph{width} of $P$, that is $w(P)$ denotes the size of the largest antichain in $P$.
\begin{Theorem}\label{usefulposet} Suppose that a partially ordered set $P$, of cardinality $s\ge 2$, is a cartesian product of the chains $P_{1}$, $P_{2}$, $\hdots$, $P_{t}$, where each $P_i$ has cardinality $k_i$. Let $K:=\sum_{i=1}^t k_i$. Then 
$$w(P)\le \left\lfloor\frac{s}{2^{K}}\binom{K}{\lfloor K/2\rfloor}\right\rfloor\le \left\lfloor\frac{bs}{\sqrt{K}}\right\rfloor\le\left\lfloor \frac{bs}{\sqrt{\log{s}}}\right\rfloor,$$where $b=\sqrt{2/\pi}$. Furthermore, if each chain has the same cardinality $p$, then $w(P)\le \lfloor bp^{t}/\sqrt{t(p-1)}\rfloor$.\end{Theorem}
We remark that an asymptotic version of this bound is proved in \cite[Theorem 1.4]{BKR}.

To state the key application of Theorem \ref{usefulposet}, we require two definitions.  
\begin{Definition}\label{KDef} For a positive integer $s$ with prime factorisation $s=p_{1}^{r_{1}}p_{2}^{r_{2}}\hdots p_{t}^{r_{t}}$, set\\ $\omega{(s)}:=\sum r_{i}$, $\omega_{1}{(s)}:=\sum r_{i}p_{i}$, $K(s):=\omega_{1}{(s)}-\omega(s)=\sum r_i(p_i-1)$ and 
$$\ws(s)=\frac{s}{2^{K(s)}}\binom{K(s)}{\left\lfloor\frac{K(s)}{2}\right\rfloor}.$$ \end{Definition} 
For a prime $p$, write $s_{p}$ for the $p$-part of $s$. 
\begin{Definition}\label{KDef2} Let $s$ be a positive integer, and let $p$ be prime. We define $s_p$ to be the $p$-part of $s$, $\lpp{(s)}=\max\{s_q\text{ : }q\text{ prime}\}$, and
$$E(s,p):=\min\left\{{\left\lfloor \frac{bs}{\sqrt{(p-1)\log_{p}{s_{p}}}}\right\rfloor,\frac{s}{\lpp{(s/s_{p})}}}\right\}\text{ and }E_{sol}(s,p):=\min\left\{\ws(s),s_{p}\right\}$$
where we take $\left\lfloor bs/\sqrt{(p-1)\log_{p}{s_{p}}}\right\rfloor$ to be $\infty$ if $s_{p}=1$. \end{Definition}

The mentioned application can now be given as follows.
\begin{Theorem}\label{pmodlemmaMainCorp} Let $G$ be a finite group, let $\mathbb{F}$ be a field of characteristic $p>0$, let $H$ be a subgroup of $G$, and let $V$ be an $\mathbb{F}[H]$-module, of dimension $a$. Let $S$ be the group induced by $G$ on the set of (right) cosets of $H$. Define $E'$ to be $E_{sol}$ if $S$ contains a soluble transitive subgroup, and $E':=E$ otherwise. Let $M$ be a submodule of the induced module $V\uparrow^G_H$. Then $d_G(M)\le aE'(s,p)$.\end{Theorem}

Here, $d_G(M)$ denotes the minimal number of elements required to generate $M$ as a $G$-module. We actually prove slightly stronger results than Theorem \ref{pmodlemmaMainCorp} - see Theorem \ref{BKGARTheorem} and Theorem \ref{pmodlemma}.

Our next main result is motivated by a conjecture of L. Pyber, which states that: \emph{The number of subgroups $|\Sub(\Sym(n))|$ of $\Sym(n)$ is precisely $2^{(\frac{1}{16}+o(1))n^{2}}$} \cite{PybAnnals}. For $m\in \mathbb{N}$, let $\Sub_{m}(\Sym(n))$ denote the set of subgroups $H$ of $\Sym(n)$ with the property that all $H$-orbits are of length at most $m$. J.C. Schlage-Puchta (private correspondence) has proven that if the quantity 
$$f(n):=\max\{d(G)\log{|G|}/n^{2}\text{ : }G\le \Sym(n)\text{ transitive}\}$$
approaches $0$ as $n$ tends to $\infty$, then there exists an absolute constant $\underline{c}$ such that the number of subgroups of $\Sym(n)$ is at most $2^{o(n^{2})}\Sub_{\underline{c}}(\Sym(n))$. This reduces Pyber's conjecture to counting the number of subgroups that have all orbit lengths bounded above by $\underline{c}$.

Motivated by this, we prove the following.
\begin{Theorem}\label{TransOrderTheorem} There exists an absolute constant $C$ such that
$$d(G)\le \left\lfloor \dfrac{Cn^2}{\log{|G|}\sqrt{\log{n}}}\right\rfloor$$
whenever $G$ is a transitive permutation group of degree $n\ge 2$.\end{Theorem}

In particular, the discussed reduction of Pyber's conjecture follows. We remark that the bound in Theorem \ref{TransOrderTheorem} is `asymptotically best possible'. See Example \ref{TransExample} for more details.

Finally, we also discuss minimally transitive permutation groups. A transitive permutation group $G$ is said to be \emph{minimally transitive} if every proper subgroup of $G$ is intransitive. Since every transitive group contains a minimally transitive subgroup, these groups arise naturally in reduction arguments.

Minimally transitive groups also have applications in Combinatorics (for counting vertex transitive graphs; for example, see \cite{BabaiSos}), and in the theory of BFC-groups (see \cite{NVL} and \cite{ShepWie}). In this paper, we use them to study minimal generator numbers in modules for permutation groups.  Thus, some information on their structure is desirable. Our main result is as follows.
\begin{Theorem}\label{MinTransTheorem1} Let $G$ be a minimally transitive permutation group of degree $n=2^m3$. Then one of the following holds:\begin{enumerate}[(i)]
\item $G$ is soluble, or:
\item $G$ has a unique nonabelian chief factor, which is a direct product of copies of $L_{2}(p)$, where $p$ is a Mersenne prime.
\end{enumerate}\end{Theorem}

A minimally transitive group of prime power degree is a $p$-group (see Lemma \ref{MinTransLemma1}), so is in particular soluble; the motivation behind Theorem \ref{MinTransTheorem1} is to study how far away from being soluble a minimally transitive group of degree $n:=2^{m}3$ is. It would be interesting to study the same question for minimally transitive groups of degree $n:=p^{m}q$, for arbitrary primes $p$ and $q$. For an analysis of the case $n=pq$, for distinct primes $p$ and $q$, see \cite{Suppq}, \cite{Koppq} and \cite{DallaVolta}. 

For information about minimal generator numbers in minimally transitive groups, see \cite{GT}.

The layout of the paper is as follows: In Section 2, we discuss preliminary results in Permutation Group Theory and Representation Theory. In Section 3 we discuss minimally transitive groups and prove Theorem \ref{MinTransTheorem1}. Section 4 is the critical step of the paper: there, we prove upper bounds on the minimal number of elements $d_G(M)$ required to generate a submodule $M$ of an induced module $V\uparrow^G_H$ for a finite group $G$, and a subgroup $H\le G$. These bounds are derived in terms of $\dim{V}$, $|G:H|$, and some additional data when either the field involved is finite, or when $G$ is insoluble. In particular, we prove Theorem \ref{pmodlemmaMainCorp}. We also prove Theorem \ref{usefulposet} in Section 4. In Section 5, we prove a stronger version of Theorem \ref{TransTheorem}, while in Section 6 we prove Theorem \ref{TransOrderTheorem}.

Our proofs are theoretical, although we do use MAGMA \cite{MAGMA} for computations of generator numbers and composition factors for some groups of small order. In particular, we compute the maximum values of $d(G)$ as $G$ runs over the transitive groups of degree $n$, for $2\le n\le 32$.

{\bf Notation:} The following is a table of constants which will be used throughout the paper.
\begin{center}\begin{tabular}{|c|c|}
  \hline
  $b$ & $\sqrt{2/\pi}=0.797885\hdots$ \\
  \hline
  $b_{1}$ & $\sqrt{2}b=1.12838\hdots$\\
  \hline
  $c_{1}$ & $\sqrt{3}/2=0.866025\hdots$\\
  \hline
  $c$ & $1512660\sqrt{\log{(2^{19}15)}}/(2^{19}15)=0.920581\hdots$\\
  \hline
  $c_{0}$ & $\log_{9}{48}+(1/3)\log_{9}{24}=2.24399\hdots$\\
  \hline
  $c'$ & $\ln{2}/1.25506=0.552282\hdots$\\
  \hline
\end{tabular}\end{center} 
We will adopt the $\mathbb{ATLAS}$ \cite{Atlas} notation for group names, although we will usually write $\Sym(n)$ and $\Alt(n)$ for the symmetric and alternating groups of degree $n$. Furthermore, these groups, and their subgroups, act naturally on the set $\{1,\hdots,n\}$; we will make no further mention of this.

The centre of a group $G$ will be written as $Z(G)$, the Frattini subgroup as $\Phi(G)$, and the Fitting subgroup as $F(G)$. The letters $G$, $H$, $K$ and $L$ will usually be used for groups, while $U$, $V$ and $W$ will usually be modules. The letter $M$ will usually denote a submodule. Finally, group homomorphisms will be written on the right.

We finish by recording a definition which will be used throughout the paper.
\begin{Definition} Let $G$ be a group.\begin{enumerate}[(a)]
\item Write $a(G)$ to denote the composition length of $G$.
\item Let $\ab(G)$ and $\anab(G)$ denote the number of abelian and non-abelian composition factors of $G$, respectively. 
\item Let $\nab(G)$ denote the number of non-abelian chief factors of $G$.
\end{enumerate}\end{Definition}
 
The author is hugely indebted to his supervisor Professor Derek Holt for many useful discussions and suggestions; without them, this paper would not be possible.  He would also like to thank both Dr. Tim Burness and the referee for many useful comments and suggestions. Finally, he would also like to thank the Engineering and Physical Sciences Research Council for their financial support.

\section{Preliminaries}\label{IntroChapter}
\subsection{Permutation groups}
We begin with some notation. Suppose that $G$ is a group acting on a set $\Omega$, via the homomorphism $\theta:G\rightarrow \Sym(\Omega)$. When there is no ambiguity, we will abbreviate $\omega^{g\theta}$ to $\omega^{g}$, for $g\in G$, $\omega\in \Omega$. We will also write 
$$G^{\Omega}:=G\theta\text{, and }{\Ker}_{G}(\Omega):=\Ker(\theta)$$ 
to denote the image and kernel of $\theta$, respectively. The orbit $\omega^{G\theta}$ of $\omega\in \Omega$ under the action of $G$ will be abbreviated to $\omega^{G}$, while the stabiliser will be written as $\Stab_{G}(\omega)$. If $\Omega$ is finite of cardinality $n$, we have
$$(\Sym(\Omega),\Omega)\cong (\Sym(n),\{1,\hdots,n\}).$$
Thus, in this case, we will usually write $\Sym{(\Omega)}=\Sym{(n)}=S_n$, and say that a subgroup $G\le \Sym(\Omega)$ is a \emph{permutation group of degree $n$}. If, for $1=1$, $2$, $G_i$ is a group acting on the set $\Omega_i$, we will write $(G_1,\Omega_1)\cong (G,\Omega_2)$ if $(G_1,\Omega_1)\cong (G,\Omega_2)$ are permutation isomorphic.

Let $\omega_{i}^{G}$,  $i\in I$, denote the orbits in $\Omega$ under the action of $G$ (the set $I$ is an index set). The groups $G^{\omega_{i}^{G}}$ are called the \emph{transitive constituents} of $G$ on $\Omega$. 

\begin{Definition} Let $G_i$, $i\in I$, be a set of groups. A subgroup $G$ of the direct product $\prod_i G_i$ is called a \emph{subdirect product} of the $G_i$ if $\pi_i|_{G}:G\rightarrow G_i$ is surjective for each projection map $\pi_i:\prod_i G_i\rightarrow G_i$.\end{Definition}

We note the following easily proved proposition, which will be used frequently.
\begin{Proposition}[{\bf \cite{CamPerm}, Theorem 1.1}]\label{SubdirectProp} Let the group $G$ act on the finite set $\Omega$. Then $G^{\Omega}$ is isomorphic to a subdirect product of its transitive constituents.\end{Proposition}

\subsection{Wreath products}\label{TransIntro}
Let $R$ be a finite group, let $S$ be a permutation group of degree $s$, and consider the wreath product $R\wr S$, as constructed in \cite{CamPerm}. Let $B$ be the base group of $R\wr S$, so that $B$ is isomorphic to the direct product of $s$ copies of $R$. Thus, for a subgroup $L$ of $R$, $B$ contains the direct product of $s$ copies of $L$: we will denote this direct product by $B_{L}$ (so that $B_{1}=1$ and $B_{R}=B$). 

Now, for each $1\le i\le s$, set 
$$R_{(i)}:=\{(g_1,\hdots,g_s)\in B\text{ : }g_j=1\text{ for all }j\neq i\}\unlhd B.$$ Then $R_{(i)}\cong R$, and $B=\prod_{1\le i\le s} R_{(i)}$. Furthermore, $N_{R\wr S}(R_{(i)})\cong R_{(i)}\times (R\wr \Stab_{S}(i))$. Hence, we may define the projection maps 
\begin{align}\rho_{i}:N_{R\wr S}(R_{(\gamma)})\rightarrow R_{(i)}.\end{align} We also define $\pi:R\wr S\rightarrow S$ to be the quotient map by $B$. This allows us to define a special class of subgroups of $R\wr S$.
\begin{Definition}\label{LargeDef} A subgroup $G$ of $R\wr S$ is called \emph{large} if \begin{enumerate}[(a)]
\item $G\rho_{i}=R_{(i)}$ for all $i$ in $1\le i\le s$, and;
\item $G\pi=S$.\end{enumerate}\end{Definition}

\begin{Remark}\label{WreathBlockRemark} Suppose, in addition, that $R$ is a permutation group of degree $r>1$. If $s>1$ and $G$ is a large subgroup of $R\wr S$, then $G$ is a transitive, and imprimitive, permutation group of degree $rs$, with a system of $s$ blocks, each of cardinality $r$. ($G$ acts on the cartesian product $\{1,\hdots,r\}\times \{1,\hdots,s\}$ in this case. \end{Remark}
In fact, it turns out that all imprimitive permutation groups arise as a large subgroup of a certain wreath product. 

\begin{Theorem}[{\bf \cite{Sup}, Theorem 3.3}]\label{SupPerm} Let $G$ be an imprimitive permutation group on a set $\Omega_1$, and let $\Delta$ be a block for $G$. Also, let $\Gamma:=\Delta^{G}$ be the set of $G$-translates of $\Delta$, and set $\Omega_2:=\Delta\times \Gamma$. Denote by $R$ and $S$ the permutation groups $\Stab_{G}(\Delta)^{\Delta}$, and $G^{\Delta^{G}}$, on $\Delta$ and $\Gamma$ respectively. Then\begin{enumerate}[(i)]
\item $G\cong G^{\Omega_{2}}$ is isomorphic to a large subgroup of $R\wr S$, and;
\item $(G,\Omega_{1})$ and $(G,\Omega_{2})$ are permutation isomorphic.\end{enumerate}\end{Theorem}

If $G$ is an imprimitive permutation group, and the block $\Delta$ as in Theorem \ref{SupPerm} is assumed to be a minimal block for $G$, then the group $R=\Stab_{G}(\Delta)^{\Delta}$ is primitive. When $\Omega$ is finite we can iterate this process, and deduce the following.
\begin{Corollary}\label{PrimCom} Let $G$ be a transitive permutation group on a finite set $\Omega$. Then there exist primitive permutation groups $R_{1}$, $R_{2}$, $\hdots$, $R_{t}$ such that $G$ is a subgroup of $R_{1}\wr R_{2}\wr \hdots\wr R_{t}$.\end{Corollary}

\begin{Remark}\label{WreathRemark} The wreath product construction is associative, in the sense that $R\wr (S\wr T)\cong (R\wr S)\wr T$, so the iterated wreath product in Corollary \ref{PrimCom} is well-defined.\end{Remark}

\begin{Definition} The tuple $(R_{1},R_{2},\hdots,R_{t})$, where the $R_i$ are as in Corollary \ref{PrimCom}, is called \emph{a tuple of primitive components} for $G$ on $\Omega$.\end{Definition}

We caution the reader that a tuple of primitive components for an imprimitive permutation group $G$ on a set $\Omega$ is not necessarily unique - see \cite[Page 13]{CamPerm} for an example.

We close this subsection with an easy lemma concerning the alternating group $\Alt(d)$.
\begin{Lemma}\label{AltOrbits} Let $D\cong \Alt(d)$ be the alternating group of degree $d\ge 5$, and let $p$ be prime. Then $D$ contains a soluble subgroup $E$ with at most two orbits, such that each orbit has $p'$-length.\end{Lemma}
\begin{proof} Assume first that $p=2$. Then since $n$ is either odd, or a sum of two odd numbers, we can take $E:=\langle x_{1}x_{2}\rangle$, where $x_{1}$ is a cycle of odd length, either $x_{2}=1$ or $x_{2}$ is a cycle of odd length, and $d$ is the sum of the orders (i.e. lengths) of $x_{1}$ and $x_{2}$.

So assume that $p>2$, and write $d=tp+k$, where $0\le k\le p-1$. If $k\neq p-1$, then take $E_{1}$ to be a soluble transitive subgroup of $\Alt(tp-1)$, and take $E_{2}$ to be a soluble transitive subgroup of $\Alt(k+1)$. If $k=p-1$, then take $E_{1}$ to be a soluble transitive subgroup of $\Alt(tp+1)$, and take $E_{2}$ to be a soluble transitive subgroup of $\Alt(k-1)$ (note that $k-1>0$ since $p>2$). Finally, taking $E:=E_{1}\times E_{2}\le D$ give us what we need, and proves the claim. \end{proof}

\subsection{Asymptotic results for permutation groups}
We will frequently use a result on composition length, due to Pyber. First, define the constant
\begin{align} c_0:=\log_{9}{48}+(1/3)\log_{9}{24}=2.24399\hdots \end{align}

The result of Pyber can now be given as follows. It is stated slightly different to how it is stated in \cite{Pyb}.
\begin{Theorem}[{\bf\cite{Pyb}, Theorem 2.10}]\label{pyber} Let $R$ be a primitive permutation group of degree $r\ge 2$. Then $\ab{(R)}\le (1+c_{0})\log{r}-(1/3)\log{24}$, and $\anab{(R)}\le \log{r}$.\end{Theorem}
We shall also require the following theorem of D. Holt and C. Roney-Dougal on generator numbers in primitive groups.
\begin{Theorem}[{\bf\cite{derek}, Theorem 1.1}]\label{derekthm} Let $H$ be a subnormal subgroup of a primitive permutation group of degree $r$. Then $d(H)\le\lfloor\log{r}\rfloor$, except that $d(H)=2$ when $m=3$ and $H\cong \Sym({3})$.\end{Theorem}
We deduce the following easy consequence.
\begin{Corollary}\label{derek} Let $G$ be an imprimitive permutation group of degree $n$, and suppose that $G$ has a minimal block $\Delta$ of cardinality $r\ge 4$. Let $S$ denote the induced action of $G$ on the set of distinct $G$-translates of $\Delta$. Then $d(G)\le s\lfloor\log{r}\rfloor+d(S)$, where $s:=n/r$.\end{Corollary}
\begin{proof} Let $R$ be the induced action of the block stabiliser $\Stab_G({\Delta})$ on $\Delta$, and let $K:=\Ker_G(\Omega)$ be the kernel of the action of $G$ on the set $\Omega$ of distinct $G$-translates of $\Delta$. Then $K^\Delta\unlhd R$, and hence, by Theorem \ref{derekthm}, each normal subgroup of $K^\Delta$ can be generated by $\lfloor\log{r}\rfloor$ elements.  

Since $K\unlhd G$, we have \begin{align}\label{mop}(K,\Delta)\cong (K,\Delta^g) \end{align}
for all $g\in G$. Also, since $R$ is primitive, $K^{\Delta}\unlhd R$ is either trivial or transitive. If $K^{\Delta}$ is trivial, then $K$ is trivial by \ref{mop}, and hence $d(G)=d(G/K)=d(S)$. So assume that $K^{\Delta}$ is transitive. Then $K$ is an iterated subdirect product of $s$ copies of $K^{\Delta}$, by Proposition \ref{SubdirectProp}. Hence, $d(K)\le s\lfloor\log{r}\rfloor$ by the previous paragraph. Since $G/K\cong S$, the claim follows.\end{proof}

\subsection{Some results from Representation Theory}\label{RepTheorySec}
We now record two lemmas which will be key in the proof of Proposition \ref{IrrProp}. The first has a stronger version which is stated in \cite[Lemma 2.13]{derek}, but we only require the following.
\begin{Lemma}[{\bf\cite{derek}, Lemma 2.13}]\label{213} Let $G\le GL_{n}(\mathbb{F})$ be finite, let $V=\mathbb{F}^n$ be the natural module, and assume that $G$ acts irreducibly on $V$. Suppose that\begin{enumerate}
\item $V\downarrow_L$ is homogeneous for each normal subgroup $L$ of $G$; and
\item $G$ has no non-trivial abelian quotients.\end{enumerate}
Then $G$ is isomorphic to a subgroup of $GL_{n/f}(\mathbb{K})$ for some divisor $f$ of $n$, and some extension field $\mathbb{K}$ of $\mathbb{F}$ of degree $f$. Furthermore, if $W$ denotes the natural module for $GL_{n/f}(\mathbb{K})$, then $G$ acts irreducibly on $W$ and\begin{enumerate}[(i)]
\item $W\downarrow_L$ is homogeneous for each normal subgroup $L$ of $G$;
\item $Z(G)$ is cyclic; and
\item Each abelian characteristic subgroup of $G$ is contained in $Z(GL_{n/f}(\mathbb{K}))$.
\end{enumerate}\end{Lemma}

\begin{Lemma}\label{FieldExt} Let $G\le GL_{n}(\mathbb{F})$ be finite, let $V$ be the natural module, and assume that $V$ is irreducible. Suppose that $1\neq E\unlhd L\unlhd G$, and that $V\downarrow_L$ is homogeneous. Suppose that $\mathbb{K}\supseteq \mathbb{F}$ is a splitting field for all subgroups of $L$, and assume that the resulting extension $\mathbb{K}/\mathbb{F}$ is normal. Then $V^\mathbb{K}\downarrow_E$ is a non-trivial completely reducible $\mathbb{K}[E]$-module.\end{Lemma}
\begin{proof} Since $L$ is homogeneous, $V\downarrow_L\cong eU$, for some irreducible $\mathbb{F}[L]$-module $U$ and some positive integer $e$. Since $G$ is faithful on $V$ and $L\neq 1$, $L$ is faithful on $U$. Moreover, $U^{\mathbb{K}}$ is completely reducible, and each of its irreducible constituents are algebraically conjugate, by \cite[Theorem 70.15]{CurtisReiner}. It follows that $L$ is faithful on $V^{\mathbb{K}}\downarrow_{L}$, and hence $V^{\mathbb{K}}\downarrow_{E}$ is non-trivial. Also, since $E\unlhd L$, and 
$$V^{\mathbb{K}}\downarrow_{E}\cong V^{\mathbb{K}}\downarrow_{L}\downarrow_{E},$$
it follows from Clifford's Theorem (see \cite[Theorem 49.7]{CurtisReiner}) that $V^{\mathbb{K}}\downarrow_E$ is completely reducible. This completes the proof.\end{proof} 

\begin{Remark}\label{FIELDEXT} Let $\mathbb{K}$ be a splitting field for the finite group $G$, containing the field $\mathbb{F}$. Then every field $\mathbb{E}$ containing $\mathbb{K}$ is also a splitting field for $G$ (for example, see \cite[Corollary 9.8]{Isaacs}). Thus, one can always find a splitting field $\mathbb{E}$ for $G$ such that $\mathbb{E}/\mathbb{F}$ is a normal extension (for instance, by taking $\mathbb{E}$ to be the normal closure of $\mathbb{K}/\mathbb{F}$).\end{Remark}    
 
\subsection{Number Theory: The prime counting function}\label{NTSection}
We close this section with a brief discussion of large prime power divisors of positive integers. 
\begin{Definition} For a positive integer $s$ and a prime $p$, write $s_{p}$ for the $p$-part of $n$. Also, define $\lpp{s}=\max_{p\text{ prime}} s_{p}$ to be the the largest prime power divisor of $s$.\end{Definition}

Fix $s\ge 2$, and let $k=\lpp{s}$. By writing the prime factorization of $s$ as $s=kp_{2}^{r_{2}}\hdots p_{t}^{r_{t}}$, one immediately sees that $s\le k^{\delta(k)}$, where $\delta(k)$ denotes the number of primes less than or equal to $k$. Hence, $\log{s}\le \delta(k)\log{k}$. Also, it is proved in \cite[Corollary 1]{Ros} that
$$\delta(k)< 1.25506k/\ln{k}$$
for $k\ge 2$. Define the constant $c'$ by
\begin{align} c':=\ln{2}/1.25506\end{align}
We deduce the following.
\begin{Lemma}\label{primecount} Let $s$ be a positive integer. Then
$$\lpp{s}\ge (\ln{2}/1.25506)\log{s}=c'\log{s}.$$\end{Lemma}
\section{Minimally transitive groups of degree $2^m3$}\label{MinTransChapter1}
We begin our work towards the proof of Theorem \ref{TransTheorem} with a discussion of minimally transitive permutation groups. As mentioned in Section 1, we use these groups to study minimal generator numbers in modules for permutation groups. Specifically, if $H\le G$ are finite groups, $V$ is a $G$-module, and $\widetilde{G}$ is a subgroup of $G$ acting transitively on the set $H\backslash G$ of right cosets of $H$ in $G$, then $V\uparrow^G_H\cong {V\uparrow^{\widetilde{G}}}_{\widetilde{G}\cap H}$, by Mackey's Theorem (see \cite[Proposition 6.20]{Mack}). Thus, when studying induced modules, one may often reduce to the case where $G$ acts minimally transitively on $H\backslash G$.

Note also that the bounds we obtain in Theorem \ref{pmodlemma} and its corollaries are strong enough to prove Theorem \ref{TransTheorem} in most cases. Due to the nature of the bounds however, this is not the case when $|G:H|$ has the form $2^{m}3$. Thus, we have to work harder, and try to obtain some information about the structure of the minimally transitive groups of degree $2^{m}3$. Recall from Section 1 that our main result is as follows.
\begin{Theorem2} Let $G$ be a minimally transitive permutation group of degree $n=2^m3$. Then one of the following holds:\begin{enumerate}[(i)]
\item $G$ is soluble; or
\item $G$ has a unique nonabelian chief factor, which is a direct product of copies of $L_{2}(p)$, where $p$ is a Mersenne prime.
\end{enumerate}\end{Theorem2}

We begin preparations towards the proof of Theorem \ref{MinTransTheorem1} with some easy observations on minimally transitive groups. 
\begin{Lemma}\label{MinTransLemma1}  Let $G$ be a transitive subgroup of $S_{n}$, let $A$ be a point stabiliser in $G$, let $1\neq L$ be a normal subgroup of $G$, and let $\Omega=\{\Delta_1,\hdots,\Delta_{\chi}\}$ be the set of $L$-orbits. Then\begin{enumerate}[(i)]
\item Either $L$ is transitive, or $\Omega$ forms a system of blocks for $G$. In particular, the size of an $L$-orbit divides $n$.
\item $(L,\Delta_1)$ is permutation isomorphic to $(L,\Delta_j)$, for all $j$.
\item $|\Omega|=|G:AL|$.
\item $G$ is minimally transitive if and only if the only subgroup $X\le G$ satisfying $AX=G$ is $X=G$.
\item If $G$ is minimally transitive, then $G^{\Omega}$ is minimally transitively.
\item If $n=p^{a}$ for a prime $p$ and $G$ is minimally transitive, then $G$ is a $p$-group.
\end{enumerate}\end{Lemma}
\begin{proof} Parts (i), (ii) and (iii) are clear. Also, a subgroup $X$ of $G$ is transitive if and only if $AX=G$. Hence, Part (iv) follows. 

Part (v) is proved in \cite[Theorem 2.4]{DallaVolta}. Finally, Part (vi) follows since a Sylow $p$-subgroup of a transitive group of degree $p^a$ acts transitively. \end{proof}

\subsection{Subgroups of index $2^{m}3$ in direct products of nonabelian simple groups}   
In \cite[Corollary 6]{LieMin}, information is given regarding the prime divisors of indices of subgroups of simple groups. We utilise this work in the following proposition.  
\begin{Proposition}\label{SubSimpleStep} Let $T$ be a nonabelian finite simple group, and suppose that $T$ has a proper subgroup $X$ of index $n=2^{i}3^{j}$, with $0\le j\le 1$. Then one of the following holds:\begin{enumerate}[(i)]
\item $T=M_{12}$ and $X$ is contained in one of the two $T$-conjugacy classes of copies of $M_{11}$ in $M_{12}$.
\item $T=M_{11}$ or $M_{24}$, and $X$ is $T$-conjugate to $L_2(11)$ or $M_{23}$, respectively. 
\item $T=A_r$, $r=2^{i}3^{j}$, and either $X$ is $T$-conjugate to $A_{r-1}$, or $r=6$ and $X$ is $T$-conjugate to $L_2(5)$. 
\item $T=L_{2}(p)$ where $p$ is a prime of the form $p=2^{f_{1}}3^{f_2}-1$ with $f_2\le 1$, and $X$ is a subgroup of index either $1$ or $3$ in a $T$-conjugate of the maximal subgroup $M=C_{p}\rtimes C_{(p-1)/2}< L_{2}(p)$.
\end{enumerate}\end{Proposition}
\begin{proof} For a finite set $F$, let $\pi(F)$ denote the set of prime divisors of $|F|$. Thus, we have $\pi(X)\subseteq\pi(T)$, since $X\le T$. We wish to reduce to the case $\pi(X)=\pi(T)$ and then use \cite[Corollary 6]{LieMin}. However, we first need to deal with some cases which are not covered by this approach. First, the classification of the maximal subgroups of the simple classical groups of dimension up to $12$ implies that $T$ is not $L_{2}(8)$, $L_{3}(3)$, $U_{3}(3)$, $\Sp_{4}(8)$, $U_4(2)$ or $U_5(2)$ (see \cite[Tables 8.1, 8.2, 8.3, 8.4, 8.5, 8.6, 8.10, 8.11, 8.14, 8.20 and 8.21]{DerekColvaJohn}). 

Assume next that $T\cong L_{2}(p)$, for some prime $p$ of the form $p=2^{f_{1}}3^{f_2}-1$, with $f_2\ge 0$. Also, let $M$ be a maximal subgroup of $T$ containing $X$. Then, since $|T:M|$ divides $|T:X|=2^{i}3^{j}$ with $j\le 1$, we must have $M=C_{p}\rtimes C_{(p-1)/2}$, and $f_2\le 1$ (see \cite[Table 8.1]{DerekColvaJohn}). Set $l:=1$ if $f_2=0$, and $l:=3$ if $f_2=1$. Since $(p+1)/l$ is the highest power of $2$ dividing $|T|$, and $|T:X|=2^{i}3^j$ with $j\le 1$, either $X=M$; or $f_2=0$ and $|M:X|=3$. This is the situation described in (iv).

Next, assume that $T$ is one of the Mathieu groups $M_{11}$ or $M_{12}$. Using the $\mathbb{ATLAS}$ \cite{Atlas}, we find that the only possibilities for $X$ are $T=M_{11}$ and $X$ is $T$-conjugate to $L_{2}(11)\le M_{11}$ (of index $12$); or $T=M_{12}$ and $X$ is a member of one of the two $T$-conjugacy classes of $M_{11}\le M_{12}$ (of index $12$).  

Finally, assume that $T$ is not one of the groups considered above, and let $\Pi$ be the set of primes for $T$ given in the statement of \cite[Corollary 6]{LieMin}. Then $\pi(|T:X|)\subseteq\{2,3\}$, and $q\ge 5$ for each $q\in \Pi$ (the cases where $\Pi$ contains $2$ or $3$ have been dealt with in the preceding paragraphs - see \cite[Corollary 6]{LieMin}). Thus, we must have $\Pi\subseteq \pi(X)$. Hence \cite[Corollary 6]{LieMin} gives $\pi(X)=\pi(T)$ and the possibilities for $T$ and $X$ are as follows (see \cite[Table 10.7]{LieMin}).\begin{enumerate}[(1)]
\item $T=A_{r}$, $A_{k}\unlhd X\le S_{k}\times S_{r-k}$, and $k$ is greater than or equal to the largest prime $p$ with $p\le r$ (in particular, $k\ge 5$, since $T$ is simple). Then $|A_{r}:A_{r}\cap (S_{k}\times S_{r-k})|=\binom{r}{k}$ divides $|T:X|=2^{i}3^{j}$. But a well-known theorem of Sylvester and Schur (see \cite{Syl}) states that either $\binom{r}{k}=1$ or $\binom{r}{k}$ has a prime divisor exceeding $\min\left\{k,r-k\right\}$. Thus, since $k\ge 5$ we must have $k=r-2$ or $k=r-1$. Since $r\ge 5$, $k=r-1$ is the only option and hence $X=A_{r-1}$, which gives us what we need.
\item $T=A_6$, $X= L_2(5)$. This, together with (1) above, gives precisely the situation described in (iii).
\item $T=\PSp_{2m}(q)$ ($m$, $q$ even) or $\Ps\Omega_{2m+1}(q)$ ($m$ even, $q$ odd), and $\Omega^{-}_{2m}(q) \unlhd X$. Then $X\le N_{T}(\Omega^{-}_{2m}(q))$, so $|T:N_{T}(\Omega^{-}_{2m}(q))|$ divides $|T:X|=2^{i}3^{j}$. But\\ $|N_{T}(\Omega^{-}_{2m}(q)):\Omega^{-}_{2m}(q)|= 2$, by \cite[Proposition 4.8.6]{KleidLie} for $T=\PSp_{2m}(q)$ and \cite[Proposition 4.1.6]{KleidLie} for $T=\Ps\Omega_{2m+1}(q)$. Hence, $|T:\Omega^{-}_{2m}(q)|$ divides $2^{i+1}3^{j}$. Also, for each of the two choices of $T$ we get $|T:\Omega^{-}_{2m}(q)|=q^{m}(q^{m}-1)$. But $q^{m}(q^{m}-1)$ cannot be of the form $2^{f}$ or $2^{f}3$, since $m>1$ and $(m,q)\neq (2,2)$ (as $T$ is simple). Therefore, we have a contradiction. 
\item $T=\Ps\Omega^{+}_{2m}(q)$ ($m$ even, $q$ odd) and $\Omega_{2m-1}(q) \unlhd X$. As above, $X\le N_{T}(\Omega_{2m-1}(q))$, and we use \cite[Proposition 4.1.6 Part (i)]{KleidLie} to conclude that $|N_{T}(\Omega_{2m-1}(q)):\Omega_{2m-1}(q)|= 2$. It follows that $\frac{1}{2}q^{m-1}(q^{m}-1)=|T:\Omega_{2m-1}(q)|$ divides $2^{i+1}3^{j}$. This again gives a contradiction, since $m\ge 4$.
\item $T=\PSp_{4}(q)$ and $\PSp_{2}(q^{2}) \unlhd X$. Then $X\le N_{T}(\PSp_{2}(q^2))$, and \cite[Proposition 4.3.10]{KleidLie} gives $|N_{T}(\PSp_{2}(q^{2})):\PSp_{2}(q^{2})|= 2$. It follows that $q^{2}(q^{2}-1)=|T:\PSp_{2}(q^{2})|$ divides $2^{i+1}3^{j}$. Again, this is impossible.
\item In each of the remaining cases (see \cite[Table 10.7]{KleidLie}), we are given a pair ($T$, $Y$), where $T$ is $L_{2}(8)$, $L_{3}(3)$, $L_{6}(2)$, $U_{3}(3)$, $U_{3}(5)$, $U_{4}(3)$, $U_{6}(2)$, $\PSp_{4}(7)$, $\PSp_{4}(8)$, $\PSp_{6}(2)$, $\Ps\Omega^{+}_{8}(2)$, $G_{2}(3)$, $^{2}F_{4}(2)'$, $M_{24}$, $\HS$, $\ML$, $\Co_{2}$ or $\Co_{3}$, and $Y$ is a subgroup of $T$ containing $X$. Apart from when $T=M_{24}$, we find that $|T:Y|$ does not divide $2^{i}3^{j}$, so we get a contradiction in each case. When $T=M_{24}$, the only possibility is when $X$ is $T$-conjugate to $M_{23}\le M_{24}$ (of index $24$).\end{enumerate}
This completes the proof.\end{proof}

Our main tool in proving Theorem \ref{MinTransTheorem1} is the Frattini argument. The result is well-known, but we couldn't find a reference so we include a proof here.
\begin{Lemma}\label{FrattiniArgument} Let $G$ be a group, and let $L$ be a normal subgroup of $G$. Suppose that $H$ is a subgroup of $L$ with the property that $H$ and $H^{\alpha}$ are $L$-conjugate for each $\alpha\in \Aut(L)$. Then $G=N_G(H)L$.\end{Lemma}
\begin{proof} Let $g\in G$. Then conjugation by $g$ induces an automorphism of $L$, so $H^g=H^l$ for some $l\in L$, by hypothesis. Hence, $gl^{-1}\in N_{G}(H)$, so $g\in N_{G}(H)L$, and this completes the proof.\end{proof}

With the Frattini argument in mind, the next corollary will be crucial.
\begin{Lemma}\label{HStep3} Let $T$ be a nonabelian finite simple group, and suppose that $T$ has a proper subgroup $X$ of index $r:=2^{i}3^{j}$, with $0\le j\le 1$. Assume also that if $T\cong L_{2}(p)$, with $p$ a Mersenne prime, then $j=0$. Denote by $\Gamma$ the set of right cosets of $X$ in $T$. Then there exists a proper subgroup $H$ of $T$ with the following properties:\begin{enumerate}[(i)]
\item $H$ and $H^{\alpha}$ are conjugate in $T$ for each automorphism $\alpha\in \Aut(T)$; and
\item $N_T(H)^{\Gamma}$ is transitive.
\end{enumerate}\end{Lemma}
\begin{proof} By Proposition \ref{SubSimpleStep}, the possibilities for the pair $(T,X)$ (up to conjugation in $T$) are as follows:\begin{enumerate}
\item $(T,X)=(A_{r},A_{r-1})$, with $r=2^{i}3^{j}$ for some $j\le 1$, or $(T,X)=(A_6,L_2(5))$. Since $T$ is nonabelian simple, $r\ge 6$, so $r$ is even. If $r$ is a power of $2$, let $H$ be a Sylow $2$-subgroup of $T$. Then $H^{\Gamma}$ itself is transitive, and properties (i) and (ii) are clearly satisfied.

Otherwise, let $H=\langle (1,2,3),(4,5,6),\hdots,(r-1,r-2,r)\rangle$. Then $N_{T}(H)^{\Gamma}$ is transitive. Thus, (ii) is satisfied. Property (i) is also easily seen to be satisfied (this includes the case $r=6$, when $\Out{(A_{6})}$ has order $4$). 
\item $(T,X)=(M_{11},L_{2}(11))$: Let $H$ be a Sylow $3$-subgroup of $T$. Then $N_{T}(H)\cong M_{9}:2$ (see page 18 of the $\mathbb{ATLAS}$ of finite groups \cite{Atlas}) acts transitively on the cosets of $X$. Since $\Aut(M_{11})=\Inn(M_{11})$, (i) and (ii) are satisfied.
\item $T=M_{12}$ and $X$ is $T$-conjugate to one of the two copies of $M_{11}$ in $M_{12}$; or $T=M_{24}$ and $X$ is $T$-conjugate $M_{23}$: In each case, let $H$ be a subgroup of $T$ generated by a fixed point free element of order $3$. When $T=M_{12}$, $N_{T}(H)\cong A_{4}\times S_{3}$ (see \cite[page 18]{Atlas}) is a maximal subgroup of $T$, and acts transitively on the cosets of $X$ (for each copy of $M_{11}$). Also, the unique non-identity outer automorphism of $M_{12}$ fixes the set of $T$-conjugates of $H$, so both (i) and (ii) are satisfied.

When $T=M_{24}$, $N_{T}(H)$ has order $1008$, and acts transitively on the cosets of $X$ (using MAGMA \cite{MAGMA}, for example). Also, $\Out{(T)}$ is trivial. Thus, (i) and (ii) are again satisfied.
\item $T=L_{2}(p)$, with $p=2^{f_{1}}3^{f_2}-1\ge 7$, $f_2\le 1$ and $X=C_{p}\rtimes C_{(p-1)/2}$. Then $|T:X|=p+1=2^{f_1}3^{f_2}$. Assume first that $p\ge 7$, and let $H$ be a dihedral group of order $p+1$ contained in $T$. Since $T$ has a unique conjugacy class of maximal subgroups of dihedral groups of order $p+1$, (i) follows. Furthermore, $|T:H|$ and $|T:X|$ are coprime, so (ii) is also satisfied. 

This just leaves the case $p=5$, but in this case $T=A_5$ and $X$ is $T$-conjugate to $D_{10}$ so taking $H=A_4$ gives us what we need.  
\end{enumerate}\end{proof}

\begin{Lemma}\label{pplus1Lemma} Let $p\ge 7$ be a Mersenne prime, and let ${L}={T_{1}}\times {T_{2}}\times\hdots\times {T_{e}}$, where each $T_{i}\cong L_{2}(p)$. Also, let ${A}$ be a subgroup of ${L}$ such that $|{L}:{A}|=2^{a}3$, for some $a$, and $|T_{i}:T_{i}\cap {A}|\in\{p+1,3(p+1)\}$ for all $i$, with $|T_i:T_i\cap A|=3(p+1)$ for at least one $i$. Then
\begin{enumerate}[(i)]\item $|{L}:{A}|= 3(p+1)^{e}$.
\item Let $P$ be a Sylow $p$-subgroup of $L$. Then $N_{{L}}(P)$ is soluble, and has precisely $2^{e}$ orbits on the set $\Delta$ of (right) cosets of ${A}$ in ${L}$, with $\binom{e}{k}$ orbits of size $3p^{k}$, for each $k$, $0\le k\le e$.\end{enumerate}
\end{Lemma}
\begin{proof} We first prove Part (i) by induction on $e$, with the case $e=1$ being trivial. So assume that $e>1$, and fix $k$ in the range $1\le k\le e$ with $|T_k:T_k\cap A|=3(p+1)$. Also, fix $i\neq k$, and set 
$\hat{T_{i}}:=T_{1}\times\hdots\times T_{i-1}\times T_{i+1}\times\hdots\times T_{e}$
and $\hat{A_{i}}={A}\cap \hat{T_{i}}$. Then 
$$|{T_{j}}:{T_{j}}\cap \hat{A_{i}}|=|{T_{j}}:T_{j}\cap \hat{T_{i}}\cap {A}|=|T_{j}:T_{j}\cap A|\in\{3(p+1),p+1\}$$ for each $j\neq i$. In particular, $|{T_{k}}:{T_{k}}\cap \hat{A_{i}}|=3(p+1)$. Also, $|\hat{T_{i}}:\hat{A_{i}}|=|\hat{T_{i}}A:A|$ divides $|{L}:{A}|$, and is divisible by $|{T_{k}}:T_{k}\cap \hat{A_{i}}|=|T_{k}\hat{A_{i}}:\hat{A_{i}}|=3(p+1)$, so $|\hat{T_{i}}:\hat{A_{i}}|=2^{b_{i}}3$, for some $b_{i}\le a$. Hence, the inductive hypothesis implies that $|\hat{T_{i}}:\hat{A_{i}}|=3(p+1)^{e-1}$.

Assume that the claim in Part (i) does not hold. Then since $(p+1)^{e}$ is the highest power of $2$ dividing $|{L}|$, we must have $|{L}:\hat{T_i}{A}|=|L:A|/|\hat{T_i}:\hat{A_i}|<p+1$. Hence, if $\rho_{i}:{L}\rightarrow T_{i}$ denotes projection onto $T_{i}$, then $|T_{i}:\rho_{i}({A})|=|\rho_{i}({L}):\rho_{i}(\hat{T_{i}}{A})|=|{L}:\hat{T_{i}}{A}|<p+1$. But, as can be readily checked using \cite[Tables 8.1 and 8.2]{DerekColvaJohn}, no maximal subgroup of $L_{2}(p)$ can have index a power of $2$ and strictly less than $p+1$. Thus, we must have $\hat{T_{i}}{A}={L}$, so ${A}$ projects onto $T_{i}$. But then ${A}\cap T_{i}$ is a normal subgroup of $T_{i}$, so ${A}\cap T_{i}=1$ or $T_{i}$. This contradicts $|T_{i}:{A}\cap T_{i}|\in\{p+1,3(p+1)\}$, and Part (i) follows.

Finally, we prove (ii). Let $N:=N_{{L}}(P)$. By Proposition \ref{SubSimpleStep} Part (iii), each $T_{j}\cap {A}$ is contained in a maximal subgroup $M_{j}:=C_{p}\rtimes C_{(p-1)/2}$ of $T_{j}$, and $|T_{j}:T_{j}\cap {A}|\in\{p+1,3(p+1)\}$. Thus, $T_{j}\cap {A}$ has a normal Sylow $p$-subgroup $P_{j}\cong C_{p}$. Let $\widetilde{P}:=P_{1}\times \hdots\times P_{e}$, so that $\widetilde{P}$ is a Sylow $p$-subgroup of $L$. Since $P$ and $\widetilde{P}$ are conjugate in $L$, we may assume, for the purposes of proving Part (ii), that $\widetilde{P}=P$. Since $M_{j}=N_{T_{j}}(P_{j})$ is soluble, $N=M_{1}\times \hdots\times M_{e}$ is soluble. Also, $P\unlhd {A}$ since $P$ is a characteristic subgroup of $(T_1\cap A)\times\hdots\times (T_e\cap A)\unlhd A$, so ${A}\le N$. 

Suppose first that $e=1$. Then $|{L}:{A}|=3(p+1)$, so ${A}$ has index $3$ in ${N}$, since $|{L} :N|=|{L}:M_{1}|=p+1$. Let $x\in {L}\backslash N$, and let $\Gamma\subset \Delta$ be the $N$-orbit corresponding to ${A}x$. Then $|\Gamma|=|N:N\cap {A}^{x}|=\frac{|{L}:N\cap {A}^{x}|}{|{L}:N|}$. Since $|{L}:N|=p+1$ is a power of $2$ and $|{L}:N\cap {A}^{x}|$ is divisible by $|{L}:{A}^{x}|=3(p+1)$, it follows that $3$ divides $|\Gamma|$. Also, as mentioned above, $A^{x}$ and $N$ have unique Sylow $p$-subgroups $P^{x}$ and $P$, respectively. Since $x$ does not normalise $P$, we have $P^{x}\neq P$, so $p$, and hence $3p$, divides $|N:N\cap {A}^{x}|=|\Gamma|$. Since $|N:{A}|=3$ and $|{L}:{A}|=3(p+1)$, it follows that $|\Gamma|=3p$, which proves the claim in the case $e=1$.

We now consider the general case. Fix $1\le i\le e$, and $x_{i}\in T_{i}\backslash M_{i}$. Suppose first that $|T_i:T_i\cap A|=3(p+1)$. From the previous paragraph, we see that $M_{i}$ has precisely two orbits on the cosets of $T_i\cap {A}$ in $T_{i}$, of size $3$ and $3p$, represented by ${A}$ and ${A}x_{i}$ respectively. Next, assume that $|T_i:T_i\cap A|=p+1$. Then $M_i=T_i\cap A$. Moreover, arguing as in the previous paragraph, $p$ divides $|M_i:M_i\cap {A}^{x_i}|$, from which it follows that $M_i$ again has two orbits on the cosets of ${A}\cap T_{i}$ in $T_{i}$, of size $1$ and $p$, represented by ${A}$ and ${A}x_{i}$ respectively.

Let $B:=(T_1\cap {A})\times\hdots\times (T_e\cap {A})\unlhd {A}$. It is clear, from the previous paragraph, that $N=M_{1}\times\hdots\times M_{e}$ has $2^{e}$ orbits on the cosets of $B$ in $L$, represented by $Bt_{1}t_{2}\hdots t_{e}$, where $t_{i}\in \left\{1,x_{i}\right\}$, for $1\le i\le e$. Also, the orbit represented by the coset $Bt_{1}t_{2}\hdots t_{e}$ has cardinality $3^{d}p^{k}$, where $k$ is the number of subscripts $i$ with $t_{i}\neq 1$, and $d$ is the number of subscripts $i$ with \begin{align}|T_i:T_i\cap A|=3(p+1).\end{align}

Since $B\le {A}$, $N$ has at most $2^{e}$ orbits in $\Delta$. Suppose there exist $t_{i}$, $\widetilde{t}_{i}\in \left\{1,x_{i}\right\}$ for $1\le i\le e$, and $n=n_{1}n_{2}\hdots n_{e}\in N$ (with $n_{i}\in M_{i}$), such that ${A}t_{1}t_{2}\hdots t_{e}={A}(\widetilde{t}_{1}\widetilde{t}_{2}\hdots \widetilde{t}_{e})(n_{1}n_{2}\hdots n_{e})$. Then $t_{i}=a_{i}\widetilde{t}_{i}n_{i}$, where $a_{1}a_{2}\hdots a_{e}\in {A}$. Since ${A}\le N$, it follows that $t_{i}=1$ if and only if $\widetilde{t}_{i}=1$. Hence, $t_{1}t_{2}\hdots t_{e}=\widetilde{t}_{1}\widetilde{t}_{2}\hdots\widetilde{t}_{e}$. Thus, $N$ has precisely $2^{e}$ orbits in $\Delta$, represented by ${A}t_{1}\hdots t_{e}$, where $t_{i}\in \left\{1,x_{i}\right\}$. Since the size of the $N$-orbit corresponding to ${A}{t_{1}t_{2}\hdots t_{e}}$ is 
$$|N:N\cap {A}^{t_{1}t_{2}\hdots t_{e}}|=\frac{|N:N\cap B^{t_{1}t_{2}\hdots t_{e}}|}{|N\cap {A}^{t_{1}t_{2}\hdots t_{e}}:N\cap B^{t_{1}t_{2}\hdots t_{e}}|}\ge \frac{|N:N\cap {B}^{t_{1}t_{2}\hdots t_{e}}|}{|{A}^{t_{1}t_{2}\hdots t_{e}}: B^{t_{1}t_{2}\hdots t_{e}}|},$$
and $|{A}^{t_{1}t_{2}\hdots t_{e}}:B^{t_{1}t_{2}\hdots t_{e}}|=|{A}:B|=|N:B|/|N:{A}|=3^{d-1}$, it now follows from (4.2.1) that 
$$|N:N\cap {A}^{t_{1}t_{2}\hdots t_{e}}|=\frac{|N:N\cap B^{t_{1}t_{2}\hdots t_{e}}|}{3^{d-1}}=3p^{k}$$
where $k$ is the number of subscripts $i$ such that $t_{i}\neq 1$. This proves (ii).
\end{proof}     

\subsection{The proof of Theorem \ref{MinTransTheorem1}}
First, we fix some notation which will be retained for the remainder of this section: Let $G$ be a minimally transitive permutation group of degree $2^{m}3$; let $A$ be the stabiliser in $G$ of a point $\delta$; let $L$ be a minimal normal subgroup of $G$; let $\Omega$ be the set of $L$-orbits; let $K:=\Ker (G^{\Omega})$ be the kernel of the action of $G$ on $\Omega$; and finally, let $\Delta$ be the $L$-orbit containing $\delta$.

\begin{Remark}\label{Rem} $G^{\Omega}$ acts minimally transitively on $\Omega$, by Lemma \ref{MinTransLemma1} Part (v). Note also that, if $|G:AL|$ is a power of $2$, then $G^{\Omega}$ is a $2$-group by Lemma \ref{MinTransLemma1} Part (vi).\end{Remark}

We require the following easy proposition.
\begin{Proposition}\label{MInsolStep} There exists a subgroup $E$ of $G$ such that $G=EL$ and $E\cap K$ is soluble.\end{Proposition}
\begin{proof} Consider the (set-wise) stabiliser $\Stab_G({\Delta})$ of $\Delta$ in $G$. Since $L$ acts transitively on $\Delta$, we have $LA=\Stab_G({\Delta})$. Let $E$ be a subgroup of $G$ minimal with the property that $EK=G$. Then $E\cap K$ is contained in the Frattini subgroup of $E$, and hence is soluble. Finally, $G=EK\le E\Stab_G({\Delta})=ELA$, so $G=ELA$. Thus, $EL=G$ by minimal transitivity, as needed.\end{proof}

\begin{Corollary}\label{LastCor} If $L$ is abelian, then the set of nonabelian chief factors of $G$ equals the set of nonabelian chief factors of $G^{\Omega}$. If $L$ is nonabelian and $|\Omega|=|G:LA|$ is a power of $2$, then $L$ is the unique nonabelian chief factor of $G$.\end{Corollary}
\begin{proof} Let $E$ be as in Proposition \ref{MInsolStep}, and assume that either $L$ is abelian or $L$ is nonabelian and $|\Omega|=|G:LA|$ is a power of $2$. For a finite group $X$ write $\mathrm{NCF}(X)$ for the set of nonabelian chief factors of $X$. We need to prove that $\mathrm{NCF}(G)=\mathrm{NCF}(G^{\Omega})$ if $L$ is abelian, and $\mathrm{NCF}(G)=\{L\}$ otherwise. Note that if $|\Omega|$ is a power of $2$ then $G^{\Omega}$ is soluble, by Remark \ref{Rem}. 

Since $E^{\Omega}$ is transitive, the minimal transitivity of $G^{\Omega}$ implies that $G^{\Omega}=E^{\Omega}\cong E/E\cap K$. Since $E\cap K$ is soluble, it follows that $\mathrm{NCF}(G^{\Omega})=\mathrm{NCF}(E)$. By hypothesis, either $L$ is abelian, or $L$ is nonabelian and $E^{\Omega}$, and hence $E$, is soluble. Since $G=EL$, the claim follows, in either case.\end{proof}

\begin{Proposition}\label{preProp} Suppose that $L=T_1\times \hdots\times T_f$, where each $T_i$ is isomorphic to a nonabelian simple group $T$. Without loss of generality, assume that $\Ker_L(\Delta)=T_{e+1}\times\hdots\times T_f$, so that $L^{\Delta}=T_1^{\Delta}\times\hdots\times T_e^{\Delta}$. Then \begin{enumerate}[(i)]\item $T\cong L_2(p)$ for some Mersenne prime $p$, 
\item $|T_i:T_i\cap A|\in\{p+1,3(p+1)\}$ for each $1\le i\le e$, and;
\item There exists at least one $i$ in the range $1\le i\le e$ such that $|T_i:T_i\cap A|=3(p+1)$.\end{enumerate}\end{Proposition}
\begin{proof} Suppose that the proposition is false, and set $X_i:=T_i\cap A$. Note that $|T_i:X_i|$ divides $2^{m}3$ for each $i$, by Lemma \ref{MinTransLemma1} Part (i). Hence, Proposition \ref{SubSimpleStep} implies that one of the following must hold:\begin{enumerate}[(a)]
\item $T\not\cong L_{2}(p)$, for any Mersenne prime $p$. Then by Proposition \ref{SubSimpleStep}, either $T_i\cong M_{12}$ and each $X_i$ is contained in one of the two conjugacy classes of $M_{11}$ in $M_{12}$; or $(T_i,X_i)=(A_r,A_{r-1})$, $(A_6,L_2(5))$, $(M_{11},L_2(11))$, $(M_{24},M_{23})$, or $(L_2(p),C_p\rtimes C_{\frac{p-1}{2}})$ where $p$ is a prime of the form $p=2^{f_1}3-1$. Here, the group $X_i$ is given up to conjugacy in $T_i$. 
\item $T\cong L_{2}(p)$ for some Mersenne prime $p$. In this case, Proposition \ref{SubSimpleStep} implies that $|T_i:X_i|=p+1$ for all $i$. In particular, $X_i$ is $T_i$-conjugate to the maximal subgroup $M_i:=C_p\rtimes C_{\frac{p-1}{2}}$ of $T_i$. (We remark that it is here where we use the assumption that the proposition is false. Specifically, since $|T_i:X_i|$ divides $2^m3$ for each $i$, Proposition \ref{SubSimpleStep} implies that $X_i$ is $T_i$-conjugate to either $M_i$, or an index $3$ subgroup of $M_i$. Hence $|T_i:X_i|\in \{p+1,3(p+1)\}$ for each $i$. Thus, Part (iii) of the proposition must fail, forcing $|T_i:X_i|$ to be $p+1$, and hence for $X_i$ to be $T_i$-conjugate to $M_i$, for each $i$.) \end{enumerate} 

Fix $1\le i\le e$, and write $T=T_i$. Note that $T^{\Delta}$ is isomorphic to $T$. Set $\Gamma:=\delta^T\subset\Delta$, and set $X:=T\cap A$. Then the pair $(T,X)$ satisfies the hypothesis of Lemma \ref{HStep3}. Thus, we conclude that $T$ contains a proper subgroup $H$ such that \begin{enumerate}[(i)]
\item $H$ and $H^{\alpha}$ are conjugate in $T$ for each automorphism $\alpha\in \Aut(T)$; and
\item $N_{T}(H)^\Gamma$ is transitive.\end{enumerate}

Fix a $T$-orbit $\Gamma'$ in $\Delta$. We claim that $N_{T}(H)^{\Gamma'}$ is transitive. By Lemma \ref{MinTransLemma1} Part (ii), $T^{\Gamma'}$ is permutation isomorphic to $T^{\Gamma}$. Hence, by (ii) above, there exists an automorphism $\alpha$ of $T$ such that $N_{T}(H)^{\alpha}=N_{T}(H^{\alpha})$ acts transitively on $\Gamma'$. Since $H$ is $T$-conjugate to $H^\alpha$, it follows that $N_T(H)$ is $T$-conjugate to $N_{T}(H)^{\alpha}$. Thus, $N_T(H)$ acts transitively on $\Gamma'$, as claimed.

Since $T_{i}\cong T_{j}$ for all $i$, $j$, we can choose the subgroup $H_j<T_j$ corresponding to $H$, and the subgroup $N_j<T_j$ corresponding to $N_T(H)$, for each $1\le j\le f$. Furthermore, each group $X_i$ is determined up to conjugacy in $T_i$ by (a) and (b) above. Hence, by the previous paragraph
\begin{align}N_j\text{ acts transitively on each }T_j\text{-orbit in }\Delta\text{ whenever }1\le j\le e.\end{align}

Set $\widetilde{H}=H_{1}\times H_{2}\times\hdots\times H_{f}< L\text{, and }N:=N_{1}\times N_{2}\times \hdots\times N_{f}$. Now, note that $N\le N_{L}(\widetilde{H})$. Thus, $N_{1}^{\Delta}\times N_{2}^{\Delta}\times\hdots\times N_{e}^{\Delta}=N^{\Delta}\le N_{L}(\widetilde{H})^{\Delta}$. 

We will now prove that $N^{\Delta}$ is transitive. Indeed, let $\epsilon\in\Delta$, and let $x\in L$ such that $\delta^{x}=\epsilon$. Write $x=t_{1}t_{2}\hdots t_{e}$, with $t_{j}\in T_{j}$. By (ii) above, $N_1$ acts transitively on $\delta^{T_1}$. Hence, there exists $n_{1}\in N_{1}$ such that $\delta^{t_{1}}=\delta^{n_{1}}$. We now inductively define the permutations $n_{2}$, $\hdots$, $n_{e}$ by choosing $n_{j}\in N_{j}$ such that $(\delta^{n_{1}\cdots n_{j-1}})^{n_{j}}=\delta^{n_{1}\cdots n_{j-1}t_{j}}$ (this is possible since $N_{j}$ acts transitively on $(\delta^{n_{1}\hdots n_{j-1}})^{T_j}$, by (4.3.1)). Then \begin{align*} 
\epsilon &=\delta^{t_{1}t_{2}\cdots t_{e}} =(\delta^{t_{1}})^{t_{2}\cdots t_{e}}
=\delta^{n_{1}t_{2}\cdots t_{e}} =(\delta^{n_{1}t_{2}})^{t_{3}\cdots t_{e}} \\
&=\delta^{n_{1}n_{2}t_{3}\cdots t_{e}} =(\delta^{n_{1}n_{2}t_{3}})^{t_{4}\cdots t_{e}} 
=\cdots 
=\delta^{n_{1}n_{2}\cdots n_{e}}\end{align*}
Thus
\begin{align}N^{\Delta}\text{ is transitive, as claimed.}\end{align}
 
Finally, let $\alpha\in\Aut{(L)}\cong \Aut{(T)}\wr \Sym({f})$. Then there exists $\tau \in \Sym{(f)}$ and $\alpha_{i}\in \Aut{(T)}$ such that 
\begin{align*}\widetilde{H}^{\alpha} &=H_{1^{\tau}}^{\alpha_{1}}\times H_{2^{\tau}}^{\alpha_{2}}\times\hdots\times H_{f^{\tau}}^{\alpha_{f}}\\
&= H_{1}^{\alpha_{1^{\tau^{-1}}}}\times H_{2}^{\alpha_{2^{\tau^{-1}}}}\times\hdots\times H_{f}^{\alpha_{f^{\tau^{-1}}}}\end{align*}
By (i) above, there exists, for each $1\le i\le f$, an element $t_i\in T_i$ such that $H_{i}^{\alpha_{i^{\tau^{-1}}}}=H_i^{t_i}$. Hence
\begin{align*}\widetilde{H}^{\alpha} &= H_{1}^{t_1}\times H_{2}^{t_2}\times\hdots\times H_{f}^{t_f}=\widetilde{H}^{t_{1}t_{2}\hdots t_{f}}.\end{align*}

Thus, $\widetilde{H}$ and $\widetilde{H}^{\alpha}$ are conjugate in $L$ for all $\alpha\in \Aut{(L)}$. Lemma \ref{FrattiniArgument} then implies that $G=N_{G}(\widetilde{H})L$. Thus, $N_{G}(\widetilde{H})$ acts transitively on the set $\Omega$ of $L$-orbits. But $N_{G}(\widetilde{H})$ also acts transitively on the fixed $L$-orbit $\Delta$, by (4.3.2). Hence, $N_{G}(\widetilde{H})$ is a transitive subgroup of $G$. By minimal transitivity of $G$, it follows that $N_{G}(\widetilde{H})=G$, so $\widetilde{H}$ is normal in $G$. But this is a contradiction, since $1<\widetilde{H}<L$ and $L$ is a minimal normal subgroup of $G$. The proof is complete.\end{proof}

Property (iii) of Proposition \ref{preProp} immediately implies the following.
\begin{Corollary}\label{MinTransTheoremCor2pre} Suppose that $L$ is isomorphic to a direct product of copies of $L_{2}(p)$, where $p$ is a Mersenne prime. Then $|\Delta|$ is divisible by $3$.\end{Corollary}

Finally, we are ready to prove Theorem \ref{MinTransTheorem1}.
\begin{proof}[Proof of Theorem \ref{MinTransTheorem1}] Assume that $G$ is a counterexample to the theorem of minimal degree. Note that $|\Omega|=|G:LA|$ divides $|G:A|=2^{m}3$, and is less than $2^{m}3$. Furthermore, a minimally transitive group of $2$-power degree is soluble by Remark \ref{Rem}. Hence, the minimality of $G$ as a counterexample implies that $G^{\Omega}=G/K$ satisfies either (i) or (ii) in the statement of the theorem. 

If $L$ is abelian, then Corollary \ref{LastCor} implies that the set of nonabelian chief factors of $G$ equals the set of nonabelian chief factors of $G^{\Omega}$. Thus, the result follow from the inductive hypothesis in this case. So we may assume that $L=T_{1}\times T_{2}\times \hdots\times T_{f}$, where each $T_{i}$ is isomorphic to a nonabelian finite simple group $T$. Furthermore, Proposition \ref{preProp} then implies that $T\cong L_2(p)$, where $p$ is a Mersenne prime. Also, $3$ divides $|\Delta|$ by Corollary \ref{MinTransTheoremCor2pre}. But then $|\Omega|=|G:LA|$ is a power of $2$, so $L$ is the unique nonabelian chief factor of $G$ by Corollary \ref{LastCor}. This contradiction completes the proof.\end{proof}

We also deduce two corollaries which will be vital in our application of Theorem \ref{pmodlemma} (see Section \ref{pmodlemmaSection}).
\begin{Corollary}\label{MinTransTheoremCor2} Assume that $G$ is insoluble, and let $p:=2^{a}-1$ be a Mersenne prime such that $G$ has a unique nonabelian chief factor isomorphic to a direct product of $f$ copies of $L_{2}(p)$. Then there exists a triple of integers $(e,t_1,t)$, with $e\ge 1$, and $t\ge t_{1}\ge 0$, such that\begin{enumerate}[(i)]
\item $m=ea+t$, and;
\item For some soluble subgroup $N$ of $G$, $N$ has $2^{e+t_{1}}$ orbits, with $\binom{e}{k}2^{t_{1}}$ of them of length $3p^{k}\times 2^{t-t_{1}}$, for each $k$, $0\le k\le e$.\end{enumerate}\end{Corollary}
\begin{proof} Let $E$ be as in Proposition \ref{MInsolStep}, so that $G=EL$, and $E\cap K$ is soluble. We prove the claim by induction on $m$. Suppose first that $L$ is abelian. Then since $EL=G$ and $E\cap K$ is soluble, $G^{\Omega}=E^{\Omega}$ is insoluble. Hence $|\Omega|=2^{\widetilde{m}}3$ and $|\Delta|=2^{m-\widetilde{m}}$, for some $\widetilde{m}$ with $1\le \widetilde{m}<m$, by Lemma \ref{MinTransLemma1} Parts (i) and (vi). The inductive hypothesis then implies that there exists a triple $(\widetilde{e},\widetilde{t_{1}},\widetilde{t})$ such that\begin{enumerate}
\item $\widetilde{m}=\widetilde{e}a+\widetilde{t}$, and;
\item For some soluble subgroup $\widetilde{N}$ of $E^{\Omega}$, $\widetilde{N}$ has $2^{\widetilde{e}+\widetilde{t_{1}}}$ orbits, with $\binom{\widetilde{e}}{k}2^{\widetilde{t_{1}}}$ of them of length $3p^{k}\times 2^{\widetilde{t}-\widetilde{t_{1}}}$, for each $k$, $0\le k\le \widetilde{e}$.\end{enumerate}
Set $e:=\widetilde{e}$, $t:=m-\widetilde{m}+\widetilde{t}$, and $t_{1}:=\widetilde{t_{1}}$, so that $m=ea+t$, which is what we need for (i). Also, let $Y\le E$ such that $Y^{\Omega}=\widetilde{N}$, and set $N:=LY$. Then $N$ is soluble, since the groups $Y^{\Omega}$, $Y\cap K$ and $L$ are soluble. Moreover, $N$ acts transitively on each $L$-orbit, since $L\le N$. Since each $L$-orbit has size $2^{m-\widetilde{m}}$, it follows that $N$ has $2^{e+t_{1}}$ orbits, with $\binom{e}{k}\times 2^{t_{1}}$ of them of length $3p^{k}2^{\widetilde{t}-\widetilde{t}_{1}+m-\widetilde{m}}=3p^{k}2^{t-t_{1}}$. This gives us what we need.

So assume that $L=T_{1}\times T_{2}\times\hdots\times T_{f}$, where each $T_{i}\cong L_{2}(p)$. By Proposition \ref{SubSimpleStep} Part (iii), $T_i\cap A$ is contained in the maximal subgroup $M_i\cong C_{p}\rtimes C_{(p-1)/2}$ of $T_i$, and $|T_i:T_i\cap A|\in\{p+1,3(p+1)\}$ for all $i$. Furthermore, Proposition \ref{preProp} implies that there exists at least one subscript $i$ such that $|T_{i}:T_i\cap A|=3(p+1)$. Lemma \ref{pplus1Lemma} now implies that $|\Delta|=|L:L\cap A|=3(p+1)^{e}=2^{ea}3$, where $e$ is the number of direct factors of $L$ acting non-trivially on $\Delta$. It also follows that $|\Omega|=2^{m-ea}$.    

By relabeling the $T_{i}$ if necessary, we may write $L^{\Delta}=T_{1}^{\Delta}\times T_{2}^{\Delta}\times\hdots\times T_{e}^{\Delta}$. Let $P$ be a Sylow $p$-subgroup of $L$, and let $N:=N_{L}(P)$. By Lemma \ref{pplus1Lemma} Part (ii), $N$ is soluble, and $N_{L}(P)^{\Delta}=N_{L^{\Delta}}(P^{\Delta})$ has $2^{e}$ orbits on $\Delta$, with $\binom{e}{k}$ of size $3p^{k}$, for each $0\le k\le e$. Since the action of $L$ on each $L$-orbit is permutation isomorphic to the action of $L$ on $\Delta$, it follows that $N:=N_{L}(P)$ has $2^{e}$ orbits on each $L$-orbit, with $\binom{e}{k}$ of size $3p^{k}$, for each $0\le k\le e$. Also, $N$ acts trivially on the set $\Omega$ of $L$-orbits, so $N$ has $2^{e+m-ea}$ orbits in total, with $2^{m-ea}\binom{e}{k}$ of them of size $3p^{k}$, for each $0\le k\le e$. Setting $t:=m-ea$ and $t_{1}:=t$ now gives us what we need, and completes the proof.\end{proof} 

\begin{Corollary}\label{MinTransTheoremCor3} Let $S$ be a transitive permutation group of degree $s:=2^{m}3$, and assume that $S$ contains no soluble transitive subgroups. Then there exists a Mersenne prime $p:=2^{a}-1$ and a triple of integers $(e,t_1,t)$, with $e\ge 1$, and $t\ge t_{1}\ge 0$, such that\begin{enumerate}[(i)]
\item $m=ea+t$, and;
\item For some soluble subgroup $N$ of $S$, $N$ has $2^{e+t_{1}}$ orbits, with $\binom{e}{k}2^{t_{1}}$ of them of length $3p^{k}\times 2^{t-t_{1}}$, for each $k$, $0\le k\le e$.\end{enumerate}\end{Corollary}
\begin{proof} Let $G$ be a minimally transitive subgroup of $S$. Then $G$ is insoluble, so Corollary \ref{MinTransTheoremCor2} applies, and the result follows.\end{proof}

\section{Generating submodules of induced modules for finite groups}\label{ModuleChapter}
The purpose of this paper is to study upper bounds for the function $d$ on the class of finite transitive permutation groups. As can be seen from Section 1, this essentially amounts to deriving upper bounds on $d(G)$ for subgroups $G$ of wreath products $R\wr S$. Our main strategy for doing this will be to reduce modulo the base group $B$ of $R\wr S$ and use induction to bound $d(G/G\cap B)$. In this way, all that remains is to investigate the contribution of $G\cap B$ to $d(G)$: The purpose of this section is to carry out such an investigation. 

As we will show in Lemma \ref{prechief}, the group $G\cap B$ is built, as a normal subgroup of $G$, from submodules of induced modules for $G$, and nonabelian chief factors of $G$. Thus, the main aim of the section will be to derive upper bounds for generator numbers in submodules of induced modules. The strategy to do this will be to first view soluble groups as certain partially ordered sets: We prove some properties of these partially ordered sets in Section \ref{PosetSection}. Our main results are Theorem \ref{BKGARTheorem} and Theorem \ref{pmodlemma}, which are proved in Sections \ref{BKGARSection} and \ref{pmodlemmaSection} respectively. We remark that Theorem \ref{BKGARTheorem} improves \cite[Theorem 1.5]{BKR}, while Theorem \ref{pmodlemma} improves \cite[Lemma 4]{LucMenMor}. 
   
\subsection{Partially ordered sets}\label{PosetSection}
Let $P=(P,\preccurlyeq)$ be a finite partially ordered set, and let $w(P)$ denote the \emph{width} of $P$. That is, $w(P)$ is the maximum cardinality of an antichain in $P$. Suppose now that, with respect to $\preccurlyeq$, $P$ is a cartesian product of chains, and write $P=P_{1}\times P_{2}\times \hdots\times P_{t}$, where each $P_{i}$ is a chain of cardinality $k_{i}$. Then $P$ is poset-isomorphic to the set of divisors of the positive integer $m=p_{1}^{k_{1}-1}p_{2}^{k_{2}-1}\hdots p_{t}^{k_{t}-1}$, where $p_{1}$, $p_{2}$, $\hdots$, $p_{t}$ are distinct primes. We make this identification without further comment. 

Next, recall that each divisor $d$ of $m$ can be written uniquely in the form $d=p_{1}^{r_{1}}p_{2}^{r_{2}}\hdots p_{t}^{r_{t}}$, where $0\le r_{i}\le k_{i}-1$, for each $i$, $1\le i\le t$. In this case, the \emph{rank} of $d$ is defined as $r(d)=\sum_{i=1}^{t} r_{i}$. For $0\le k\le K:=\sum_{i=1}^{t} (k_{i}-1)$, let $R_{k}$ denote the set of elements of $P$ of rank $k$; clearly $R_{k}$ is an antichain in $P$. In fact, it is proved in \cite{dBr} that $w(P)=\max |R_{k}|$. This maximal rank set occurs at $k=\lfloor K/2 \rfloor$, and hence, by \cite[Theorem 2]{And}, we have
$$w(P)\le \left\lfloor\frac{s}{2^{K}}\binom{K}{\lfloor K/2\rfloor}\right\rfloor$$ where $s:=|P|=\prod_{i=1}^{t} k_{i}$ (note that equality holds when $t$ is even and each $k_{i}$ is $2$, so this upper bound is best possible). Stated more concisely, we have   

\begin{Lemma}\label{poset1} Suppose that a partially ordered set $P$, of cardinality $s\ge 2$, is a cartesian product of the chains $P_{1}$, $P_{2}$, $\hdots$, $P_{t}$, where each $P_{i}$ has cardinality $k_{i}$. Then $$w(P)\le\left\lfloor\frac{s}{2^{K}}\binom{K}{\lfloor K/2\rfloor}\right\rfloor,$$ where $K:=\sum_{i=1}^{t} (k_{i}-1)$.\end{Lemma}
We now define a constant $b$,
\begin{align*} b:=\sqrt\frac{2}{\pi}.\end{align*}

\begin{Proposition}\label{usefulpreposet} Let $K$ be a positive integer. Then \begin{align} \binom{K}{\lfloor K/2\rfloor}\le \frac{b2^{K}}{\sqrt{K}}.\end{align}\label{KBound} 
\end{Proposition}
\begin{proof}\footnote{The idea for this bound arose from a discussion at the url\\ http://math.stackexchange.com/questions/58560/elementary-central-binomial-coefficient-estimates.} First consider the case where $K=2t$ ($t\in\mathbb{N}$), and note that
$$2t\left[\binom{2t}{t}\frac{1}{4^{t}}\right]^{2}=\frac{1}{2}\left(\frac{3}{2}\frac{3}{4}\right)\left(\frac{5}{4}\frac{5}{6}\right)\hdots\left(\frac{2t-1}{2t-2}\frac{2t-1}{2t}\right)=\frac{1}{2}\prod_{j=2}^{t}\left(1+\frac{1}{4j(j-1)}\right)$$ By Wallis' Formula, the expression in the middle converges to $2/\pi$. Hence, since the expression on the right is increasing, we have $2t\left[\binom{2t}{t}\frac{1}{4^{t}}\right]^{2}\le 2/\pi$, that is, $\binom{2t}{t}\le b4^{t}/\sqrt{2t}$, as claimed. If $K$ is odd, we have $\binom{K}{\lfloor K/2\rfloor}=\frac{1}{2}\binom{K+1}{\lfloor (K+1)/2\rfloor}$, and the bound in (\ref{KBound}) follows from the even case above.\end{proof}

\begin{proof}[Proof of Theorem \ref{usefulposet}] By Lemma \ref{poset1} and Proposition \ref{usefulpreposet}, we have
$$w(P)\le \frac{s}{2^{K}}\binom{K}{\lfloor K/2\rfloor}\le \frac{s}{2^{K}}\left(\frac{b2^{K}}{\sqrt{K}}\right)=\frac{bs}{\sqrt{K}}$$
If each $k_{i}=p$, then $K=t(p-1)$, and the second part of the claim follows. Since $K=\sum_{i=1}^{t} (k_{i}-1)\ge \sum_{i=1}^{t} \log{k_{i}}=\log{s}$, the first part also follows, and the proof is complete.\end{proof}

\subsection{Preliminary results on induced modules for finite groups}
\subsubsection{Composition factors in induced modules}
Let $\mathbb{F}$ be a field, let $G$ be a finite group, and let $V$ be a module for $G$ over $\mathbb{F}$. Let 
$$0=N_{0}<N_{1}<\hdots<N_{a}=V$$
be a $G$-composition series for $V$, and say that a factor $N_{i}/N_{i-1}$ is \emph{complemented} if there exists a submodule $S_{i}$ of $V$ containing $N_{i-1}$ such that $V/N_{i-1}=N_{i}/N_{i-1} \oplus S_{i}/N_{i-1}$. Also, for an irreducible $\mathbb{F}[G]$-module $W$, write $t_{W}(V)$ for the number of complemented composition factors of $V$ isomorphic to $W$. 

Now, fix an irreducible $\mathbb{F}[G]$-module $W$ with $t_{W}(V)\ge 1$. Then there exists a submodule $M$ of $V$ with the property that $V/M$ is $G$-isomorphic to $W$: Define $R_{W}(V)$ to be the intersection of all such $M$. In particular, $R_W(V)$ contains the radical $\Rad(V)$ of $V$.

\begin{Lemma} $V/R_{W}(V)\cong W^{\oplus t_{W}(V)}$.\end{Lemma}
\begin{proof} Let $t:=t_{W}(V)$, and write $R:=R_{W}(V)=M_{1}\cap M_{2}\cap \hdots\cap M_{e}$, where $V/M_{i}$ is isomorphic to $W$. Then 
$$V/R\le (V/M_{1})\oplus (V/M_{2})\oplus \hdots\oplus (V/M_{e})$$
and hence $V/R$ is a direct sum of $k$ copies of $W$, where $k\le e$. Since $t_{W}(V)=t_{W}(V/R)$, we have $t=k$, and this completes the proof.\end{proof}

\begin{Lemma}\label{BigLemma} Suppose that $V=U\uparrow^{G}_H$, for a subgroup $H$ of $G$ and an $H$-module $U$, and suppose that $W$ is a $1$-dimensional $\mathbb{F}[G]$-module. Then $t_{W}(V)\le \dim{U}$.\end{Lemma}
\begin{proof} Let $R=R_{W}(V)$ and $t=t_{W}(V)$. Writing bars to denote reduction modulo $R$, we have
$${\ol{V}}=\ol{N_{1}}\oplus \ol{N_{2}}\oplus\hdots\oplus \ol{N_{t}}$$ 
where each $\ol{N_{i}}$ is isomorphic to $W$. In particular, if we write
$$V/\Rad(V)=\sum_{X\text{ an irreducible }\mathbb{F}[G]\text{-module}} X^{f_X(V)},$$
then we have $t\le f_W(V)$. Moreover, since $\dim{W}=1$, we have
$$f_W(V)=\dim{{\Hom}_{\mathbb{F}[G]}(V,W)}=\dim{{\Hom}_{\mathbb{F}[H]}(U,W\downarrow_{H})}=f_{W\downarrow_{H}}(U)\le \dim{U}$$
where the second equality above follows from Frobenius Reciprosity (see \cite[Proposition 3.3.1]{Benson}). This completes the proof.\end{proof}

We will need an easy consequence of Lemma \ref{BigLemma}. To state it, we first require two definitions and a remark.
\begin{Definition} Let $G$ be a non-trivial finite group, and $\mathbb{F}$ a field. A \emph{projective representation} of $G$ of dimension $m$ over $\mathbb{F}$ is a homomorphism $\rho:G\rightarrow PGL_{m}(\mathbb{F})$. Define 
\begin{align*}{R}_{\mathbb{F}}(G) &:=\min\left\{m\text{ : }G\text{ has a non-trivial representation of dimension }m \text{ over }\mathbb{F}\right\};\text{ and}\\
\ol{R}_{\mathbb{F}}(G) &:=\min\left\{m\text{ : }G\text{ has a non-trivial projective representation of dimension }m \text{ over }\mathbb{F}\right\}.\end{align*}
Also define \begin{align*} \ol{R}(G) &:=\min\left\{\ol{R}_{\mathbb{F}}(G)\text{ : }\mathbb{F}\text{ a field}\right\}\end{align*}
\end{Definition}

\begin{Definition}\label{AAA} Let $G$ be a finite group, let $\mathbb{F}$ be a field, and let $V$ be an $\mathbb{F}[G]$-module. Define $d_G(V)$ to be the minimal number of elements required to generate $V$ as an $\mathbb{F}[G]$-module.\end{Definition}

\begin{Remark}\label{BBB} Let $G$, $\mathbb{F}$ and $V$ be as in Definition \ref{AAA}, and let $t$ be the number of complemented $G$-composition factors of $V$. We claim that $d_G(V)\le t$. Note first that $t$ is precisely the number of irreducible constituents of $V/\Rad(V)$. In particular, it follows that $d_G(V/\Rad(V))\le t$: let $v_1$, $\hdots$, $v_t\in V$ such that $V/\Rad(V)$ is generated, as a $G$-module, by $\{\Rad(V)+v_1,\hdots,\Rad(V)+v_t\}$. Let $M$ be the $G$-submodule of $V$ generated by $\{v_1,\hdots,v_t\}$. Then $V=M+\Rad(V)$. Since $\Rad(V)$ is contained in every maximal submodule of $V$, it follows that $V=M$, and hence $d_G(V)\le t$, as claimed.\end{Remark}  

The corollary of Lemma \ref{BigLemma} can now be stated as follows.
\begin{Corollary}\label{BigCor} Let $G$ be a finite group, let $H$ be a subgroup of $G$, and let $U$ be an $H$-module, over a field $\mathbb{F}$.  Let $V:=U\uparrow^G_H$. Then 
$$d_{G}(V)\le \frac{\dim{U}|G:H|-\dim{U}}{R_{\mathbb{F}}(G)}+\dim{U}.$$ \end{Corollary}
\begin{proof} Write $t$ for the number of complemented $G$-composition factors of $V$ which are not isomorphic to the trivial $G$-module ${1}_G$. By Remark \ref{BBB}, we have
$$d_G(V)\le t_{{1}_G}(V)+t.$$
Since $\dim{V}=\dim{U}|G:H|$, we have
$$t\le \frac{\dim{U}|G:H|-\dim{U}}{R_{\mathbb{F}}(G)}.$$
The result now follows immediately from Lemma \ref{BigLemma}.\end{proof}

\subsubsection{Induced modules for Frattini extensions of nonabelian simple groups}
In this subsection, we make some observations on modules for Frattini extensions of nonabelian simple groups. That is, modules for groups $G$ with $G/\Phi(G)$ a non-abelian simple group. 

The main result of this section reads as follows.
\begin{Proposition}\label{IrrProp} Let $G$ be a finite group with a normal subgroup $N\le \Phi(G)$ such that $G/N\cong T$, where $T$ is a non-abelian finite simple group. Also, let $W$ be a nontrivial irreducible $G$-module, over an arbitrary field $\mathbb{F}$. Then\begin{enumerate}[(i)]
\item Each proper normal subgroup of $G$ is contained in $N$. In particular, $N=\Phi(G)$.
\item $\Ker_{G}(W)$, the kernel of the action of $G$ on $W$, is contained in $N$.
\item $n:=\dim{W}\ge \ol{R}(T)$. \end{enumerate}\end{Proposition}
\begin{proof} Part (i) follows since $N\le \Phi(G)$ and $G/N$ is simple. Part (ii) now follows from Part (i) since $W$ is non-trivial.

We will now prove (iii). By (ii), we may assume that $G$ is faithful on $W$. In particular, we may view $G$ as a subgroup of $GL_n(\mathbb{F})$. Let $L$ be a normal subgroup of $G$, and assume that $W\downarrow_L$ is non-homogeneous. If $K$ is the kernel of the action of $G$ on the homogeneous components of $W\downarrow_{L}$, then $K$ is a proper normal subgroup of $G$, so $K\le N$ by Part (i). Thus, $HN< G$ for some stabiliser $H$ of a homogeneous component. Hence, $|G:H|\ge |G:HN|=|G/N:HN/N|\ge \ol{R}_{\mathbb{F}}(T)$, since any proper subgroup $E$ of $T$ gives rise to a nontrivial permutation representation for $T$ of dimension $|T:E|$ over $\mathbb{F}$ (a non-trivial projective representation of dimension $|T:E|$ is then achieved by reducing modulo scalars). Thus, the number of homogeneous components is at least $\ol{R}_{\mathbb{F}}(T)$, and the result follows.

So we may assume that $W\downarrow_L$ is homogeneous for each normal subgroup $L$ of $G$. Hence, by Lemma \ref{213}, we may assume that $Z(G)$ is cyclic and
that each abelian characteristic subgroup of $G$ is contained in $Z(GL_n(\mathbb{F}))$.

Let $L$ be the generalised Fitting subgroup of $G$, and extend the field $\mathbb{F}$ so that $\mathbb{F}$ is a splitting field for each subgroup of $L$, and so that the resulting field extension is normal (see Remark \ref{FIELDEXT}).

We distinguish two cases.\begin{enumerate}
\item $L$ is soluble. In this case, since $L>Z(G)$, $O_{r}(G)$ must be non-central, for some prime $r$, and $O_{r}(G)C_{G}(O_{r}(G))\ge L$. Also, since $O_{r}(G)$ is non-central, we have $O_{r}(G)$, $C_{G}(O_{r}(G))\le N$ by Part (i). Thus, since $N\le \Phi(G)\le L$, it follows that $N=L=O_{r}(G)C_{G}(O_{r}(G))$. Hence, by \cite[Lemma 1.7]{LucComp}, there exists a positive integer $m$ such that\begin{enumerate}[(1)]
\item $O_{r}(G)$ is a central product of its intersection with $Z:=Z(G)$ and an extraspecial group $E$ of order $r^{1+2m}$; 
\item $Z(E)$ coincides with the subgroup of $Z$ of order $r$ (recall that $Z$ is cyclic);
\item $EZ/Z$ is a completely reducible $\mathbb{F}_r[G]$-module under conjugation; and 
\item $C_{G/Z}(EZ/Z)=O_r(G)C_G(O_r(G))/Z$.\end{enumerate}
It follows from (4) that $T\cong G/N=G/O_{r}(G)C_{G}(O_{r}(G))$ is a non-trivial completely reducible subgroup of $GL_{2m}(r)$. It then follows that \begin{align}\label{Steph} \ol{R}_{\mathbb{F}_r}(T)\le 2m.\end{align} 

Next, by Lemma \ref{FieldExt}, $W\downarrow_E$ is completely reducible and its irreducible constituents are non-trivial. Let $U$ be such a constituent. Since $\mathbb{F}$ is a splitting field for $E$, $U$ is absolutely irreducible. Hence, $\dim{U}\ge r^{m}$, by \cite[Theorem 5.5]{Gorenstein}. Thus, by (\ref{Steph}), we have 
$$\ol{R}(T)\le \ol{R}_{\mathbb{F}_r}(T)\le 2m\le r^{m}\le \dim{U}\le \dim{W},$$ which gives us what we need.

\item $L$ is insoluble. By \cite[Lemma 2.14]{derek}, $L$ contains a normal subgroup $X$ of $G$ of the form $X=S_{1}\circ\hdots\circ S_{t}$, where each $S_{i}$ is isomorphic to a quasisimple group $S$. But since $N\le \Phi(G)$, $N$ is nilpotent. Also, $G/N$ is simple, so we must have $G=X$ and $G$ is quasisimple. In particular, $N=Z\le Z(GL_n(\mathbb{F}))$. Hence, $T\cong G/Z\le PGL_n(\mathbb{F})$ and $\dim{W}\ge \ol{R}_{\mathbb{F}}(T)\ge \ol{R}(T)$, as required.\end{enumerate}
This completes the proof.\end{proof}

\subsection{Induced modules for finite groups}\label{MainModuleSection}
We begin with some terminology.
\begin{Definition} Let $M$ be a group, acted on by another group $G$. A \emph{$G$-subgroup} of $M$ is a subgroup of $M$ which is stabilised by $G$. We say that $M$ is \emph{generated as a $G$-group} by $X\subset M$, and write $M=\langle X\rangle_{G}$, if no proper $G$-subgroup of $M$ contains $X$. We will write $d_{G}(M)$ for the cardinality of the smallest subset $X$ of $M$ satisfying $\langle X\rangle_G=M$. Finally, write $M^{\ast}:=M\backslash\{1\}$.\end{Definition}

Note that the definition of $d_G(M)$ is consistent with the notation introduced in Definition \ref{AAA} in the case where $M$ is a $G$-module.

\begin{Definition} Let $G$ be a group, acting on a set $\Omega$. Write $\chi(G,\Omega)$ for the number of orbits of $G$ on $\Omega$.\end{Definition}

The purpose of this section is to derive upper bounds for $d_{G}(M)$ when $M$ is a submodule of an induced module for $G$. To this end, we introduce some notation which will be retained for the remainder of the section: \begin{itemize}
\item Let $G$ be a finite group.
\item Fix a subgroup $H$ of $G$ of index $s\ge 2$.
\item Fix a subgroup $H_{1}$ of $H$ of index $d\ge 1$.
\item Let $U$ be a module for $H_{1}$ of dimension $a$, over a field $\mathbb{F}$.
\item Let $K:=\core_{G}(H)$, and fix a subgroup $K'$ of $K$.
\item Set $V:=U\uparrow^{H}_{H_{1}}$ and $W:=V\uparrow^G_H$ to be the induced modules. Note also that $V\uparrow^G_H\cong U\uparrow^G_{H_{1}}$. 
\item Denote the set of right cosets of $H$ in $G$ [respectively $H_{1}$ in $H$] by $\Omega$ [resp. $\Omega_{1}$].
\item Define $$m:=m(K')=\min\{\chi({Q^{\Omega_1}},{\Omega_{1}})\text{ : }Q\le K'\text{ and }Q^V\text{ is semisimple}\}.$$\end{itemize}
We do not exclude the case $d=1$, that is, $H=H_1$.

\subsubsection{Induced modules: The soluble case}\label{BKGARSection}
This section is essentially an analogue of \cite[Section 5]{BKR}. We first recall the constant $b$,
\begin{align*} b:=\sqrt\frac{2}{\pi}.\end{align*}
We also recall, from Section 1, the following definition.
\begin{Definition} For a positive integer $s$ with prime factorisation $s=p_{1}^{r_{1}}p_{2}^{r_{2}}\hdots p_{t}^{r_{t}}$, set $\omega{(s)}:=\sum r_{i}$, $\omega_{1}{(s)}:=\sum r_{i}p_{i}$, $K(s):=\omega_{1}{(s)}-\omega(s)=\sum r_i(p_i-1)$ and 
$$\ws(s)=\frac{s}{2^{K(s)}}\binom{K(s)}{\left\lfloor\frac{K(s)}{2}\right\rfloor}.$$ \end{Definition}

The main result of this section reads as follows.
\begin{Theorem}\label{BKGARTheorem} Suppose that $G^\Omega$ contains a soluble transitive subgroup, and let $M$ be a submodule of $W$. Also, denote by $\chi=\chi{(K,V^\ast)}$ the number of orbits of $K$ on the non-zero elements of $V$. Then
$$d_{G}(M)\le \min\left\{\frac{ad-am}{R_{\mathbb{F}}(K')}+am,\chi\right\}\ws(s)\le \min\left\{\frac{ad-am}{R_{\mathbb{F}}(K')}+am,\chi\right\}\left\lfloor\frac{bs}{\sqrt{\log{s}}}\right\rfloor$$ where $b:=\sqrt{2/\pi}$. Furthermore, if $s=p^{t}$, with $p$ prime, then $$d_{G}(M)\le \min\left\{\frac{ad-am}{R_{\mathbb{F}}(K')}+am,\chi\right\}\left\lfloor \frac{bp^{t}}{\sqrt{t(p-1)}}\right\rfloor.$$
\end{Theorem}

\begin{Remark} If $K$ has infinitely many orbits on the non-zero elements of $V$, then we assume, in Theorem \ref{BKGARTheorem}, and whenever it is used in the remainder of the paper, that $$\min\left\{\frac{ad-am}{R_{\mathbb{F}}(K')}+am,\chi\right\}=\frac{ad-am}{R_{\mathbb{F}}(K')}+am.$$\end{Remark}

We begin our work towards the proof of Theorem \ref{BKGARTheorem} by first collecting a series of lemmas from \cite[Section 5]{BKR}.
\begin{Lemma} [{\bf \cite{BKR}, Lemma 5.1}]\label{TLemma} Suppose that $G^{\Omega}$ contains a soluble transitive subgroup. Then there is a right transversal $\mathcal{T}$ to $H$ in $G$, with a partial order $\preccurlyeq$ and a full order $\leqslant$, satisfying the following properties:
\begin{enumerate}[(i)]\item Whenever $t_{1}$, $t_{2}$, $t_{3}\in \mathcal{T}$ with $t_{1}<t_{2}\preccurlyeq t_{3}$, we have $t_{4}<t_{3}$, where $t_{4}$ is the unique element of $\mathcal{T}$ such that $t_{1}t_{2}^{-1}t_{3}\in Ht_{4}$.
\item With respect to this partial order, $\mathcal{T}$ is a cartesian product of $k$ chains, of length $p_{1}$, $p_{2}$, $\hdots$, $p_{k}$, where $k=\omega(s)$, and $p_{1}$, $p_{2}$, $\hdots$, $p_{k}$ denote the (not necessarily distinct) prime divisors of $s$.\end{enumerate}\end{Lemma}
\begin{proof} Let $F$ be a subgroup of $G$ such that $F^{\Omega}$ is soluble and transitive. By \cite[Lemma 5.1]{BKR}, there exists a right transversal $\mathcal{T}$ for $F\cap H$ in $F$ such that the image ${\mathcal{T}}^{\Omega}$ has a partial order ${\preccurlyeq}'$ and a full order ${\leqslant}'$ satisfying \begin{enumerate}[(a)]\item Whenever $t_{1}$, $t_{2}$, $t_{3}\in \mathcal{T}$ with ${t_{1}}^{\Omega}{<'}{t_{2}}^{\Omega}{\preccurlyeq'} {t_{3}}^{\Omega}$, we have ${t_{4}}^{\Omega}{<'}{t_{3}}^{\Omega}$, where ${t_{4}}$ is the unique element of $\mathcal{T}$ such that $({t_{1}t_{2}^{-1}t_{3}})^{\Omega}\in ({F\cap H})^{\Omega}{t_{4}}^{\Omega}$.
\item With respect to this partial order, ${\mathcal{T}}^{\Omega}$ is a cartesian product of $k$ chains, of length $p_{1}$, $p_{2}$, $\hdots$, $p_{k}$, where $k=\omega(|F:F\cap H|)=\omega(|G:H|)=\omega(s)$, and $p_{1}$, $p_{2}$, $\hdots$, $p_{k}$ denote the (not necessarily distinct) prime divisors of $s$.\end{enumerate}

For $t_{1}$, $t_{2}\in \mathcal{T}$, say now that $t_{1}\preccurlyeq t_{2}$ if ${t_{1}}^{\Omega}{\preccurlyeq'} {t_{2}}^{\Omega}$, and $t_{1}\leqslant t_{2}$ if ${t_{1}}^{\Omega}{\leqslant'} {t_{2}}^{\Omega}$. Since ${F}^{\Omega}$ acts transitively on the set of cosets of $H$ in $G$, $\mathcal{T}$ is a right transversal for $H$ in $G$. By definition, (a) and (b) above imply that (i) and (ii) hold for this choice of $\preccurlyeq$ and $\leqslant$. This gives us what we need.\end{proof}

For the remainder of Section \ref{MainModuleSection} assume that $G^{\Omega}$ contains a soluble transitive subgroup, and fix $\mathcal{T}$ to be a right transversal for $H$ in $G$ as exhibited in Lemma \ref{TLemma}. Then we may write the induced module $W=V\uparrow^G_H$ as $W=\bigoplus_{t\in \mathcal{T}}V\otimes t$, where the action of $G$ is given by
$$(v\otimes t)^{ht'}=v^{h_1}\otimes t_1,$$ where $tht'=h_1t_1$, $h$, $h_1\in H$, $t$, $t'$, $t_1\in \mathcal{T}$. Thus, each element $w$ in $W$ may be written as $w=\sum_{t\in \mathcal{T}} v(w,t)\otimes t$, with uniquely determined coefficients $v(w,t)$ in $V$. 

\begin{Definition}[{\bf \cite{BKR}, Section 5}]\label{HeightDef} Let $w\in W$ be non-zero. The \emph{height} of $w$, written $\tau(w)$, is the largest element of the set $\left\{{t\in \mathcal{T}\text{ : }v(w,t)\neq 0}\right\}$, with respect to the full order $\leqslant$. Also, we define $\mu (w):=v(w, \tau(w))$. Thus, $\mu(w)$ is non-zero, and $v(w,t)=0$ whenever $t>\tau(w)$. The element $\mu(w)\otimes \tau(w)$ is called the \emph{leading summand} of $w$.\end{Definition}

\begin{Remark}\label{REMGT}In the language of Definition \ref{HeightDef}, Lemma \ref{TLemma} Part (i) states that if the height of $w$ is $t_{2}$, and if $t_{2}\preccurlyeq t_{3}$, then the height of $w^{t_{2}^{-1}t_{3}}$ is $t_{3}$. Further, the leading summand of $w^{t_{2}^{-1}t_{3}}$ is $\mu(w)\otimes t_{3}$.\end{Remark}
The formulation in Remark \ref{REMGT} leads to an important technical point.
\begin{Proposition}\label{a1Lemma} Let $M$ be a submodule of $W$. Then $M$ has a generating set $X$ with the following property: No subset $Y$ of $X$, whose image $\tau(Y)$ in $\mathcal{T}$ is a chain with respect to the partial order $\preccurlyeq$, can have more than 
$$\min\left\{\frac{ad-am}{R_{\mathbb{F}}(K')}+am,\chi\right\}$$ elements, where $\chi=\chi(K,V^\ast)$ denotes the number of orbits of $K$ on the nonzero elements of $V$.\end{Proposition}

Before proving Proposition \ref{a1Lemma}, we need a preliminary lemma.
\begin{Lemma}\label{FrobProp} A $K'$-composition series for $V$ contains at most $am$ factors isomorphic to the trivial module.\end{Lemma}
\begin{proof} Let $Q\le K'$ such that $Q^V$ is semisimple and $\chi(Q^{\Omega_1},\Omega_1)=m$. By Mackey's Theorem,
\begin{align}V\downarrow_{Q}=\left(U\uparrow^H_{H_1}\right)\downarrow_Q\cong \bigoplus_{i=1}^{m} U_{x_i},\end{align}\label{RE}
where $U_{x_i}:=(U\otimes x_i)\uparrow^Q_{Q\cap H_1^{x_i}}$, $\dim{U_{x_i}}=\dim{U}=a$, for each $i$, and $\sum_{j}|Q:Q\cap H_1^{x_i}|=|H:H_1|=d$. Since $Q^V$ is semisimple, the number of $Q$-composition factors of $U_{x_i}=(U\otimes x_i)\uparrow^Q_{Q\cap H_1^{x_i}}$ isomorphic to the trivial module $1_{Q}$ is precisely
$$\dim{{\Hom}_{\mathbb{F}[Q]}((U\otimes x_i)\uparrow^Q_{Q\cap H_1^{x_i}},1_{Q})}=\dim{{\Hom}_{\mathbb{F}[Q\cap H_1^{x_i}]}((U\otimes x_i),1_{Q\cap H_1^{x_i}})},$$ 
applying Frobenius Reciprosity. This is at most $\dim(U\otimes x_i)=\dim{U}=a$. The result now follows immediately from (\ref{RE}).\end{proof}

\begin{proof}[Proof of Proposition \ref{a1Lemma}] Set $e:=\frac{ad-am}{R_{\mathbb{F}}(K')}+am$, and let $X$ be a finite generating set for $M$, consisting of non-zero elements. Suppose that $Y:=\left\{w_0, w_1, \hdots, w_{e}\right\}$ is a subset of $X$ whose image under $\tau$ forms a chain in $\mathcal{T}$: Say $\tau(w_0)\preccurlyeq \tau(w_1)\preccurlyeq\hdots\preccurlyeq \tau(w_{e})$. 
 
Consider now the vectors $\mu(w_{0})$, $\mu(w_1)$, $\hdots$, $\mu(w_{e})$: For $1\le i\le e+1$ let $W_{i}$ denote the $K'$-module generated by $\mu(w_0)$, $\hdots$, $\mu(w_{i-1})$, and consider the series of $K'$-modules
\begin{align}\label{NNN} 0=:W_{0}\le W_{1}\le \hdots\le W_{e+1}\text{ }\end{align}
Suppose that $W_{i}<W_{i+1}$ for all $i$. Then the series (\ref{NNN}) can be extended to give a $K'$-composition series for $V$. Thus, Lemma \ref{FrobProp} implies that at most $am$ of the factors $W_{i+1}/W_{i}$ are trivial. Furthermore, the rest have dimension at least $R_{\mathbb{F}}(K')$. It follows that $\dim{W_{e+1}}= \sum_{i=1}^{e+1}\dim{W_{i}/W_{i-1}}\ge am+(e+1-am)R_{\mathbb{F}}(K')>ad$, which is a contradiction, since $\dim{V}=ad$.

Thus, we must have $\mu(w_{i})\in W_{i}$ for some $i$. In this case, 
$$\mu(w_{i})=\sum_{j=0}^{i-1}\sum_{k\in K'}\lambda_{j,k}\mu(w_{j})^{k},$$ for some scalars $\lambda_{j,k}$. Moreover, the element
$$x:=\sum_{j=0}^{i-1}\sum_{k\in K'}\lambda_{j,k}w_{j}^{k^{\tau(w_{j})}\tau(w_j)^{-1}\tau(w_i)}$$
of $M$ has the same leading summand as $w_{i}$, by Lemma \ref{TLemma} Part (i) (see also Remark \ref{REMGT}). Hence, either $x=w_{i}$ and $w_{i}$ may be removed from $X$, or $w_{i}$ may be replaced in $X$ by the element $w_{i}-x$, which has height strictly preceding $w_{i}$ in the full order $\leqslant$. In this way, the resulting (modified) set $X$ still generates $M$. This procedure can only be carried out a finite number of times, and when it can no longer be repeated, the (modified) generating set can have no more than $e$ elements. 

If $\chi\ge e$, then we are done, so assume that $\chi<e$. Let $v$ and $w$ be elements of $X$ whose images $\tau(v)$ and $\tau(w)$ are comparable (with respect to $\preccurlyeq$) in $\mathcal{T}$: Say $\tau(v)\preccurlyeq \tau(w)$. Suppose that $\mu(w)$ and $\mu(v)$ lie in the same $K$-orbit of $V$, and let $g\in K$ such that $\mu(w)^{g}=\mu(v)$. Since $K$ is normal in $G$, the leading summand of $w^{g}$ is $\mu(v)\otimes \tau(w)$. Thus, by replacing $w$ with $w^{g}$, we may assume that $\mu(v)=\mu(w)$. Then, using Lemma \ref{TLemma} Part (i) again, we see that $v^{\tau(v)^{-1}\tau(w)}$ has the same leading summand as $w$. Write $v^{\tau(v)^{-1}\tau(w)}=x+\mu(v)\otimes\tau(w)$, and $w=y+\mu(v)\otimes\tau(w)$, for $x$, $y\in V$, and let $u=y-x$. Then, we see that, as in the proof of \cite[Lemma 5.2]{BKR}, either $u=0$, and $w=v^{\tau(v)^{-1}\tau(w)}$ may be omitted from $X$, or $u\neq 0$, and $w=u+v^{\tau(v)^{-1}\tau(w)}$ may be replaced in $X$ by the element $u$, which has height strictly preceding $\tau(w)$ in the full order $\leqslant$. This way, the resulting set obtained from $X$ still generates $M$. The procedure outlined above can only be carried out a finite number of times, and when it can no longer be repeated, the (modified) generating set can contain no more than $\chi$ elements. This completes the proof.\end{proof} 

Before proving Theorem \ref{BKGARTheorem}, we note the following easy consequence of Dilworth's Theorem (\cite[Theorem 1.1]{Dil}):
\begin{Lemma}\label{DilCons} If a partially ordered set $P$ has no chain of cardinality greater than $k$, and no antichain of cardinality greater than $l$, then $P$ cannot have cardinality greater than $kl$.\end{Lemma}

\begin{proof}[Proof of Theorem \ref{BKGARTheorem}] Let $\mathcal{T}$ be a right transversal for $H$ in $G$ with full and partial orders $\leqslant$ and $\preccurlyeq$, as in Lemma \ref{TLemma}. Now define a partial order on the elements of $W$ as follows: First, for each $t\in \mathcal{T}$, choose a full order on the elements of $W$ of height $t$. Now, for $w_1$ and $w_2$ in $W$, say that $w_1< w_2$ if $\tau(w_1)$ is less than $\tau(w_2)$ in $(\mathcal{T},\preccurlyeq)$, or if $\tau(w_1)=\tau(w_2)$ but $w_1$ precedes $w_2$ in the full order chosen for elements of height $\tau(w_1)$.

Then $\tau:W\rightarrow \mathcal{T}$ is a poset homomorphism which takes incomparable elements to incomparable elements, so no antichain of its domain can have cardinality greater than $\ws{(s)}$, by Lemmas \ref{poset1} and \ref{TLemma} Part (ii). Let $X$ be a generating set for $M$ with the properties guaranteed by Proposition \ref{a1Lemma}. Then no chain in $X$ can have more than $\min\{\frac{ad-am}{R_{\mathbb{F}}(K')}+am,\chi\}$ elements. Lemma \ref{DilCons} then implies that
$$|X|\le \min\left\{\frac{ad-am}{R_{\mathbb{F}}(K')}+am,\chi\right\}\ws{(s)}\le \min\left\{\frac{ad-am}{R_{\mathbb{F}}(K')}+am,\chi\right\}\left\lfloor\frac{bs}{\sqrt{\log{s}}}\right\rfloor,$$
where the second inequality follows from Theorem \ref{usefulposet}. If $s=p^{t}$ for $p$ prime, then
$$|X|\le \min\left\{\frac{ad-am}{R_{\mathbb{F}}(K')}+am,\chi\right\}\left\lfloor\frac{bp^{t}}{\sqrt{{t(p-1)}}}\right\rfloor,$$
again by Lemma \ref{DilCons} and Theorem \ref{usefulposet}. This completes the proof. 
\end{proof}

\subsubsection{Induced modules for finite groups: The general case}\label{pmodlemmaSection}
In this section, we prove a weaker form of Theorem \ref{BKGARTheorem} for general finite groups (i.e. those $G$ for which $G^{\Omega}$ does not necessarily contain a soluble transitive subgroup). We retain the notation introduced at the beginning of Section \ref{MainModuleSection}. 

We begin with a definition. Recall the definitions of $\ws(s)$, $s_{p}$, and $\lpp{(s)}$ from Definitions \ref{KDef} and \ref{KDef2}.  
\begin{Definition}\label{EDEF} For a prime $p$, set
$$E(s,p):=\min\left\{{\left\lfloor \frac{bs}{\sqrt{(p-1)\log_{p}{s_{p}}}}\right\rfloor,\frac{s}{\lpp{(s/s_{p})}}}\right\}\text{ and }E_{sol}(s,p):=\min\left\{\ws(s),s_{p}\right\}$$
where we take $\left\lfloor bs/\sqrt{(p-1)\log_{p}{s_{p}}}\right\rfloor$ to be $\infty$ if $s_{p}=1$. \end{Definition}

\begin{Proposition} Let $p$ be prime. Then $E_{sol}(s,p)\le E(s,p)$. \end{Proposition}
\begin{proof} By Theorem \ref{usefulposet} we have $\ws(s)\le \left\lfloor \frac{bs}{\sqrt{(p-1)\log_{p}{s_{p}}}}\right\rfloor$. Also, it is clear that $s_p\le \frac{s}{\lpp{(s/s_{p})}}$. The result follows. \end{proof}

\begin{Remark}\label{JustAbove} For any finite group $G$ and any $G$-module $M$, $d_G(M)$ is bounded above by $\chi(G,M^\ast)$.\end{Remark}

For the remainder of this section, we will make a further assumption: that the field $\mathbb{F}$ has characteristic $p>0$. We are now ready to state and prove the main result of this section.
\begin{Theorem} \label{pmodlemma} For a prime $q\neq p$, let $P_{q}$ be a Sylow $q$-subgroup of $G$. Also, let $P'$ be a maximal $p'$-subgroup of $G$. Let $M$ be a submodule of the induced module $W=V\uparrow^{G}_H$. \begin{enumerate}[(i)]
\item If $G$ is soluble, then $$d_{G}(M)\le \min\left\{\frac{ad-a\chi(P'\cap K,\Omega_1)}{R_{\mathbb{F}}(P'\cap K)}+a\chi(P'\cap K,\Omega_1),\chi(P'\cap K,V^{\ast})\right\}s_{p}.$$
\item Let $N$ be a subgroup of $G$ such that $N^{\Omega}$ is soluble, and let $s_{i}$, $1\le i\le t$, be the sizes of the orbits of $N$ on $\Omega$. Then\begin{enumerate}[(a)]
\item We have \begin{align*}d_{G}(M)\le &\min\left\{\frac{ad-a\chi(N\cap P'\cap K,\Omega_1)}{R_{\mathbb{F}}(N\cap P'\cap K)}+a\chi(N\cap P'\cap K,\Omega_1),\right. \\
&\left. \chi(N\cap P'\cap K,V^{\ast})\vphantom{\frac{1}{2}}\right\}\times\sum_{i=1}^{t}\ws{(s_i)}.\end{align*}
\item If $N$ is soluble, and $P_N'$ is a $p$-complement in $N$, then 
 \begin{align*}d_{G}(M)\le &\min\left\{\frac{ad-a\chi(P_N'\cap K,\Omega_1)}{R_{\mathbb{F}}(P_N'\cap K)}+a\chi(P_N'\cap K,\Omega_1),\right. \\
&\left. \chi(P_N'\cap K,V^{\ast})\vphantom{\frac{1}{2}}\right\}\times\sum_{i=1}^{t}E_{sol}(s_i,p).\end{align*}\end{enumerate}
\item $d_{G}(M)\le \min\left\{\frac{ad-a\chi(P_q\cap K,\Omega_1)}{R_{\mathbb{F}}(P_{q}\cap K)}+a\chi(P_q\cap K,\Omega_1),\chi(P_{q}\cap K,V^\ast)\right\}s/s_{q}$. 
\item Assume that $s_p>1$. Then $$d_{G}(M)\le \min\left\{\frac{ad-am}{R_{\mathbb{F}}(K')}+am,\chi(K,V^{\ast})\right\}\left\lfloor\frac{bs}{\sqrt{\log{s_p}}}\right\rfloor.$$
\end{enumerate} 
\end{Theorem}
\begin{proof} The proof is based on the idea of Lucchini et al. used in the  proof of \cite[Lemma 4]{LucMenMor}. Let $Q$ be a subgroup of $G$, and choose a full set $\left\{{x_{1},x_{2},\hdots,x_{t}}\right\}$ of representatives for the $(H,Q)$-double cosets in $G$. Also, for $1\le i\le t$, put $s_{i}:=|Q:Q\cap H^{x_{i}}|$ (note that, by $H^{x_{i}}$, we mean, as usual, the conjugate subgroup $x_{i}^{-1}Hx_{i}$). By Mackey's Theorem we have \begin{align}\label{GT1} W\downarrow_Q=(V\uparrow^{G}_H)\downarrow_{Q}=\bigoplus_{i=1}^{t} V_{x_{i}} \end{align} where $V_{x_{i}}\cong (V\otimes x_{i})\uparrow^{Q}_{Q\cap H^{x_i}}$. Comparing dimensions of the left and right hand side of (\ref{GT1}) above, we get
$$ads=\dim{W}=\sum_{i=1}^{t}ad|Q:Q\cap H^{x_{i}}|=ad\sum_{i=1}^{t}s_{i}$$ 
so that $\sum_{i=1}^{t}s_{i}=s$. Clearly, the $s_{i}$ represent the sizes of the orbits of $Q$ on the right cosets of $H$ in $G$. 

Next, for $1\le i\le t$, set $V_{i}:=V_{x_{1}}\oplus V_{x_{2}}\oplus \hdots \oplus V_{x_{i}}$. Then, we have a chain $0=V_{0}\le V_{1}\le \hdots \le V_{t}=W$ of $Q$-submodules of $W$. This allows us to define the chain of $Q$-modules $0=M_{0}\le M_{1}\le \hdots \le M_{t}=M$, where $M_{i}:=M\cap V_{i}$. Furthermore, in this case, the quotient $M_{i}/M_{i-1}$ is (isomorphic to) a $Q$-submodule of $V_{x_{i}}$. Hence
\begin{align}\label{GT2} d_{G}(M)\le d_{Q}(M)\le\sum_{i=1}^{t} d_{Q}(M_{i}/M_{i-1}).\end{align}  

Note that $V\otimes x_{i}$ is isomorphic to an induced module $(U\otimes x_{i})\uparrow^{H^{x_{i}}}_{H_{1}^{x_i}}$. Hence, Mackey's Theorem implies that $(V\otimes x_{i})\downarrow_{Q\cap K}$ is isomorphic to a direct sum
\begin{align}\label{GT3} (V\otimes x_{i})\downarrow_{Q\cap K}\cong\bigoplus_{j} U_{x_{i,j}},\end{align}where $U_{x_{i,j}}\cong (U\otimes x_{i,j})\uparrow^{Q\cap K}_{Q\cap K\cap H_1^{x_{i,j}}} $ is an induced module for $Q\cap K$, and\\ $\sum_{j}|Q\cap K:Q\cap K\cap {H_1}^{x_{i,j}}|=|H^{x_{i}}:H_{1}^{x_i}|=d$.

Suppose that $(|Q|,p)=1$. Then each $V_{x_{i}}$ is a semisimple $\mathbb{F}[Q]$-module, so \begin{align*}
d_{Q}(M_{i}/M_{i-1})&\le d_{Q}(V_{x_{i}})\\
&\le d_{Q\cap H^{x_i}}(V\otimes x_{i})\\
&\le d_{Q\cap K}(V\otimes x_{i})\\
&\le \sum_{j}d_{Q\cap K}(U_{x_{i,j}})\\
&\le \sum_{j}\min\left\{\frac{a|Q\cap K:Q\cap K\cap H_1^{x_{i,j}}|-a}{R_{\mathbb{F}}(Q\cap K)}+a,\chi(Q\cap K,\right.\\
&\left.[U_{x_{i,j}}]^\ast\vphantom{\frac{1}{2}})\right\}\\
&\le \min\left\{\sum_j\frac{a|Q\cap K:Q\cap K\cap H_1^{x_{i,j}}|-a}{R_{\mathbb{F}}(Q\cap K)}+a,\sum_j\chi(Q\cap K,\right.\\
&\left. [U_{x_{i,j}})]^\ast\vphantom{\frac{1}{2}})\right\}\\
&= \min\left\{\frac{ad-a\chi(Q\cap K,\Omega_1)}{R_{\mathbb{F}}(Q\cap K)}+a\chi(Q\cap K,\Omega_1),\chi(Q\cap K,{V}^\ast)\right\}
\end{align*}
The fourth inequality above follows from (\ref{GT3}), while the fifth follows from Corollary \ref{BigCor} and Remark \ref{JustAbove}. Thus 
\begin{align}\label{GT4} d_{G}(M)\le \min\left\{\frac{ad-a\chi(Q\cap K,\Omega_1)}{R_{\mathbb{F}}(Q\cap K)}+a\chi(Q\cap K,\Omega_1),\chi(Q\cap K,V^\ast)\right\}t \end{align}
 by (\ref{GT2}).

Write $s_{p}:=p^{\beta}$ and $s_{q}:=q^{\alpha}$. Also, write $s=p^{\beta}q^{\alpha}k$ and $|H|=p^{\delta}q^{\gamma}l$, where $|H|_{p}=p^{\delta}$, $|H|_{q}=q^{\gamma}$. We are now ready to prove the theorem.\begin{enumerate}[(i)]
\item Suppose that $G$ is soluble, and take $Q:=P'$ to be a $p$-complement in $G$. Then $|Q|=q^{\alpha+\gamma}kl$. Hence, $s_{i}=|Q:Q\cap H^{x_{i}}|\ge q^{\alpha}k=s/s_{p}$. Part (i) now follows from (\ref{GT4}), since $s=\sum_{i=1}^{t} s_{i}\ge ts/s_{p}$. 
\item Take $Q:=N$. By Theorem \ref{BKGARTheorem}, we have 
\begin{align*}d_Q(M_i/M_{i-1})\le &\min\left\{\frac{ad-a\chi(Q\cap P'\cap K,\Omega_1)}{R_{\mathbb{F}}(Q\cap P'\cap K)}+a\chi(Q\cap P'\cap K,\Omega_1),\right. \\
&\left. \chi(Q\cap P'\cap K,V^\ast\vphantom{\frac{1}{2}})\right\}\ws(s_i). \end{align*}
Part (a) of (ii) now follows from (\ref{GT2}). Next, assume that $N$ is soluble, with a $p$-complement $P_N'$. Then 
\begin{align*}d_Q(M_i/M_{i-1})\le &\min\left\{\frac{ad-a\chi(Q\cap P'\cap K,\Omega_1)}{R_{\mathbb{F}}(Q\cap P'\cap K)}+a\chi(Q\cap P'\cap K,\Omega_1),\right.\\
&\left.\chi(Q\cap P'\cap K,V^\ast\vphantom{\frac{1}{2}})\right\}(s_i)_p\end{align*}
by Part (i). Also, $P_N'=N\cap P'$ for some maximal $p'$-subgroup $P'$ of $G$, so Part (b) follows from (\ref{GT2}) by combining the above with Part (ii)(a).
\item In the general case, take $Q:=P_{q}$. Then $|Q|=q^{\alpha+\gamma}$, so $s_{i}=|Q:Q\cap H^{x_{i}}|\ge q^{\alpha}$. Also, $s=\sum_{i=1}^{t} s_{i}\ge tq^{\alpha}=ts_{q}$. Part (iii) then follows from (\ref{GT4}). 
\item Here, we have $\beta>0$ since $s_p>0$. Let $P$ be a Sylow $p$-subgroup of $G$, and set $Q=KP$. Then $s_{i}=|Q:Q\cap H^{x_{i}}|=|QH^{x_{i}}|/|H^{x_{i}}|\ge |PH^{x_{i}}|/|H^{x_{i}}|=|P:P\cap H^{x_{i}}|\ge p^{\beta}$, for each $i$. Since $K\le \core_{Q}(Q\cap H^{x_{i}})$, we have $\chi(\core_Q{(Q\cap H^{x_{i}})},(V\otimes x_{i})^\ast)\le \chi(K,V^\ast)=:\chi$ for each $i$. Then (\ref{GT2}) and Theorem \ref{BKGARTheorem} give
\begin{align*}d_{G}(M) &\le \sum_{i=1}^{t} \min\left\{\frac{ad-am}{R_{\mathbb{F}}(K')}+am,\chi\right\}\left\lfloor\frac{bs_{i}}{\sqrt{\log{s_{i}}}}\right\rfloor\\
&\le \sum_{i=1}^{t} \min\left\{\frac{ad-am}{R_{\mathbb{F}}(K')}+am,\chi\right\}\left\lfloor \frac{bs_{i}}{\sqrt{\beta}}\right\rfloor \\
&\le \min\left\{\frac{ad-am}{R_{\mathbb{F}}(K')}+am,\chi\right\}\left\lfloor \sum_{i=1}^{t} \frac{bs_{i}}{\sqrt{\beta}}\right\rfloor \\
&=\min\left\{\frac{ad-am}{R_{\mathbb{F}}(K')}+am,\chi\right\}\left\lfloor\frac{bs}{\sqrt{\beta}}\right\rfloor \end{align*}
This proves (iv).\end{enumerate}
\end{proof}

Since $\frac{ad-f}{e}+f \le ad$ for positive integers $e$ and $f$, the following corollary is immediate.
\begin{Corollary}\label{pmodlemmaMainCor1} Let $M$ be a submodule of $W$. Also, let $q$, $P_q$ and $P'$ be as in Theorem \ref{pmodlemma}. Then\begin{enumerate}[(i)]
\item If $G$ is soluble, then $d_{G}(M)\le \min\left\{ad,\chi(P'\cap K,V^\ast)\right\}s_{p}$.
\item Let $N$ be a subgroup of $G$ such that $N^{\Omega}$ is soluble, and let $s_{i}$, $1\le i\le t$, be the sizes of the orbits of $N$ on $\Omega$. Then \begin{enumerate}[(a)]
\item We have $d_{G}(M)\le \min\left\{ad,\chi(N\cap P'\cap K,V^\ast)\right\}\sum_{i=1}^{t}\ws(s_i).$ 
\item If $N$ is soluble, and $P_N'$ is a $p$-complement in $N$, then $$d_{G}(M)\le \min\left\{ad,\chi(P_N'\cap K,V^\ast)\right\}\sum_{i=1}^{t}E_{sol}(s_i,p).$$\end{enumerate} 
\item $d_{G}(M)\le \min\left\{ad,\chi(P_{q}\cap K,V^\ast)\right\}s/s_{q}$. 
\item $d_{G}(M)\le \min\left\{ad,\chi(K,V^{\ast})\right\}\left\lfloor\frac{bs}{\sqrt{\log{s_p}}}\right\rfloor$.
\end{enumerate} \end{Corollary}

We also record the following, which is an immediate consequence of Corollary \ref{pmodlemmaMainCor1}. Note that Theorem \ref{pmodlemmaMainCorp}
\begin{Corollary}\label{pmodlemmaMainCor} Define $E'$ to be $E_{sol}$ if $G^{\Omega}$ contains a soluble transitive subgroup, and $E':=E$ otherwise. Let $M$ be a submodule of $W$. Then $d_G(M)\le adE'(s,p)$.\end{Corollary}

Note that Theorem \ref{pmodlemmaMainCorp} follows from Corollary \ref{pmodlemmaMainCor}. Using the definition of $E(s,p)$, and Lemma \ref{primecount}, we also deduce the following.
\begin{Corollary}\label{pq} Let $M$ be a submodule of $W$, and fix $0<\alpha<1$. \begin{enumerate}[(i)]
\item If $s_{p}\ge s^{\alpha}$, then $d_{G}(M)\le adE(s,p)\le ad\left\lfloor\frac{bs\sqrt{\frac{1}{\alpha}}}{\sqrt{\log{s}}}\right\rfloor$;
\item If $s_{p}\le s^{\alpha}$, then $d_{G}(M)\le adE(s,p)\le ad\left\lfloor \frac{\frac{1}{1-\alpha}s}{c'\log{s}}\right\rfloor$;
\item We have 
$$d_{G}(M)\le adE(s,p)\le
\begin{cases}
\left\lfloor \frac{2ads}{c'\log{s}}\right\rfloor, & \text{if }2\le s\le 1260,\\
\left\lfloor \frac{adbs\sqrt{2}}{\sqrt{\log{s}}}\right\rfloor, & \text{if }s\ge 1261.
\end{cases}
$$\end{enumerate}
\end{Corollary}
\begin{proof} Part (i) follows immediately from the definition of $E(s,p)$, while Part (ii) follows from the definition and Lemma \ref{primecount}. Finally, set $\alpha:=1/2$. Then $$ \frac{2ads}{c'\log{s}}\le \frac{adbs\sqrt{2}}{\sqrt{\log{s}}}$$
for $s\ge 1261$, so Part (iii) also follows.\end{proof}

The following is also immediate, from Part (ii) of Theorem \ref{pmodlemma}.
\begin{Corollary}\label{pmodlemmacor2} Let $M$ be a submodule of $W$. If $G$ contains a soluble subgroup $N$, acting transitively on $\Omega$, then \begin{align*}d_{G}(M)\le &\min\left\{\frac{ad-a\chi(P_N'\cap K,\Omega_1)}{R_{\mathbb{F}}(P_N'\cap K)}+a\chi(P_N'\cap K,\Omega_1),\chi(P_N'\cap K,V^\ast)\right\}\\
&\times E(s,p)\end{align*}
where $P_N'$ is a $p$-complement in $N$.\end{Corollary}

\subsection{An application to induced modules for bottom heavy groups}
The proofs of the main results of this paper will usually only require the bounds on $d_G(M)$ from Corollary \ref{pmodlemmaMainCor1}. For a specific case of the proof of Theorem \ref{TransOrderTheorem} however, we will need the stronger bounds provided by Theorem \ref{pmodlemma}. This case is the `bottom heavy case', which we will now define. Throughout, we retain the notation introduced at the beginning of Section \ref{MainModuleSection}. In particular, $H$ is a subgroup of $G$ of index of index $s\ge 2$, $H_1$ is a subgroup of $H$ of index $d\ge 1$, $\Omega$ is the set of right cosets of $H$ in $G$, $\Omega_1$ is the set of right cosets of $H_1$ in $H$, and $K:=\Ker_G(\Omega)$. Note that we also continue to assume that the field $\mathbb{F}$ has characteristic $p>0$. 
\begin{Definition}\label{HypA} Assume that $K^{\Omega_1}$, viewed as a subgroup of $\Sym{(d)}$, contains $\Alt(d)$. Then we say that the triple $(G,H,H_{1})$ is \emph{bottom heavy}.\end{Definition}

Before stating the main result of this section, we introduce Vinogradov notation: we will write 
$$A\ll B$$
to mean $A=O(B)$. The main result can now be stated as follows.  
\begin{Proposition}\label{MainModuleTheorem} Assume that $d\ge 5$ and that $(G,H,H_{1})$ is bottom heavy. Let $M$ be a submodule of $W$. Then\begin{enumerate}[(i)]
\item $d_{G}(M)\le 2as$, and;
\item If $s_{p}>1$, then $d_{G}(M)\ll \frac{as}{\sqrt{\log{s_{p}}}}$.\end{enumerate}\end{Proposition}

Before proving Proposition \ref{MainModuleTheorem}, we require the following:
\begin{Proposition}\label{IrrApp}  Assume that $(G,H,H_{1})$ is bottom heavy and that $d\ge 5$. Choose $K'$ to be a subgroup of $K$ minimal with the property that $K'^{\Omega_1}\cong \Alt(d)$. Then a $K'$-composition series for $V\downarrow_{K'}$ has at most $2a$ factors isomorphic to the trivial $K'$-module.\end{Proposition}
\begin{proof} By the minimality of $K'$, we have $C:=\core_{H}(H_{1})\cap K'\le \Phi(K')$, and hence $C$ is soluble. Let $E$ be a subgroup of $K'$ containing $C$ such that $E/C$ is soluble and, viewed as a subgroup of $\Sym{(d)}$, has at most two orbits, such that each orbit is of $p'$-length (such a subgroup exists by Lemma \ref{AltOrbits}). Then $E$ is soluble, so we may choose a $p$-complement $F$ in $E$. Then $F/F\cap C$ also has at most two orbits (and each $F$-orbit has $p'$-length).

Next, consider the $F$-module $X:=V\downarrow_{F}\cong U\uparrow_{H_{1}}^H\downarrow_{F}$. Since $F\le K'$, it suffices to prove that $X$ has at most $2a$ trivial composition factors. To see this, note that since $F$ has at most two orbits on $\Omega_1$ (i.e. the cosets of $H_{1}$ in $H$), represented by $x_{1}$ and $x_{2}$, say, Mackey's Theorem yields
$$X\cong X_{1}\oplus X_{2}\text{ or } X\cong X_{1}$$
where $X_{i}\cong (U\otimes x_{i})\uparrow^{F}_{F\cap H_{1}^{x_{i}}}$. Now, since $F$ has $p'$-order, $X_{i}$ is a semisimple $F$-module. Hence, the number of trivial factors in an $F$-composition series for $X_{i}$ is precisely the number of trivial summands of $X_{i}$, which is
$$\dim{{\Hom}_{\mathbb{F}[F]}(X_{i},1_{F})},$$ 
where $1_{F}$ denotes the trivial $F$-module. By Frobenius Reciprosity, this is equal to $$\dim{{\Hom}_{\mathbb{F}[F\cap H_1^{x_{i}}]}(U\downarrow_{F\cap H_1^{x_{i}}},1_{F\cap H_{1}^{x_i}})}\le \dim{U}=a.$$ The claim follows.\end{proof}

\begin{proof}[Proof of Proposition \ref{MainModuleTheorem}] Choose $K'$ to be a subgroup of $K$ minimal with the property that $K'^{\Omega_1}\cong \Alt(d)$. Then \begin{align}\label{FM} {\core}_H(H_1)\cap K'\le \Phi(K'). \end{align}
 Hence, since $$\Alt(d)\cong K'^\Omega\cong K'/{\core}_H(H_1)\cap K',$$ Proposition \ref{IrrProp} applies: $R_{\mathbb{F}}(K')\ge \ol{R}(\Alt(d))$. Note also that $m\le 2$ by Lemma \ref{AltOrbits}. Since $d\ll \ol{R}(\Alt(d))$ (see \cite[Proposition 5.3.7]{KleidLie}), Part (ii) now follows from Theorem \ref{pmodlemma} Part (iv).

We now prove (i). It follows from Lemma \ref{AltOrbits} that $K'$ has a subgroup $N$ such that $N^{\Omega_1}$ is soluble and has at most $2$ orbits. Furthermore, each orbit has $p'$-length. Also, $N$ is soluble, by (\ref{FM}). 

We now want to apply Corollary \ref{pmodlemmaMainCor1} Part (ii)(b), with $(G,H,H_1,V,\Omega)$ replaced by $(H,H_1,H_1,U,\Omega_1)$ (also, $(a,s,d)$ is replaced by $(a,d,1)$): let $d_i$, for $i\le 2$, denote the lengths of the $N^{\Omega_1}$ orbits. Then 
$$E_{sol}(d_i,p)\le (d_i)_p=1,$$
so $E_{sol}(d_i,p)=1$. Hence for each $H$-submodule $M'$ of the induced module $V={U\uparrow^{H}_{H_{1}}}$, we have
$$d_{H}(M')\le a\sum_{i=1}^{t}E_{sol}(d_i,p)\le 2a.$$
Since $M$ is a submodule of 
$$U\uparrow^G_{H_1}\cong V\uparrow^G_H\cong \sum_{i=1}^{s}V\otimes t_i$$
where each $V\otimes t_i$ is isomorphic, as an $H$-module, to $V$, the result now follows.\end{proof}

\section{Minimal generation of transitive permutation groups}\label{TransChapter}
In this section, we restate and prove the first main result of this paper, which is stated as Theorem \ref{TransTheorempre} in Section 1. The theorem follows in the primitive case from Theorem \ref{derekthm}, so this section deals predominantly with the case when $G\le \Sym(n)$ is imprimitive. In this case, $G$ is a large subgroup of a wreath product $R\wr S$, where $R$ is primitive of degree $r\ge 2$, $S$ is transitive of degree $s\ge 2$, and $n=rs$. Due to the nature of our bounds, the most difficult cases to deal with are when $R=\Sym(2)$ or $R=\Sym(4)$, i.e. when $G$ has a minimal block of cardinality either $2$ or $4$. (Essentially, this is because $\Sym(2)$ and $\Sym(4)$ have large composition lengths relative to their degree.) We deal with the $\Sym(4)$ case in Corollary \ref{S4NewBoundCor}; the idea being that we can use the transitive action of the Sylow $3$-subgroup in $\Sym(4)$ on the non-identity elements of the Klein $4$-group $V\unlhd \Sym(4)$ to reduce the contribution of $V$ to our bounds (this is the primary reason we include the invariant $\chi$ in our bounds in Section \ref{ModuleChapter}). 

However, no such option is available to us when $R\cong \Sym(2)$, since $\Sym(2)$ is abelian. If $G$ has another minimal block, of cardinality larger than $2$, then we can avoid the problem by using this block instead. However, we cannot do this if all minimal blocks for $G$ have cardinality $2$, so assume that this is the case. Then, as we will prove in Section 5.2 below, we have $d(G)\le E(s,2)+d(S)$. Now, since we just need to bound $d(S)$, we apply the same methods to the transitive group $S\le \Sym(s)$. 

Apart from finitely many cases, our methods yield the upper bound we want: the only problems occur when we ``repeatedly get" blocks of cardinality $2$. This is encapsulated in the following non-standard definition.  
\begin{Definition}\label{2BlockDef} Let $G$ be a transitive permutation group, and let $$X:=(R_1,R_2,\hdots,R_t)$$ be a tuple of primitive components for $G$, where each $R_i$ has degree $r_i\ge 2$. Define 
\begin{align*}{\bl}_{X,2}(G) &:=\min\left\{i\text{ : }r_i\neq 2\right\}-1\text{, and }\\
{\bl}_{2}(G) &:=\min\left\{{\bl}_{X,2}(G)\text{ : }X\text{ a tuple of primitive components for }G\right\}.\end{align*}
We call ${\bl}_{2}(G)$ the \emph{$2$-block number} of $G$.\end{Definition} 
Alternatively, the $2$-block number of a transitive permutation group $G$ can be defined inductively as follows: if $G$ is primitive, or if $G$ is imprimitive with a minimal block of cardinality greater than $2$, then set ${\bl}_{2}(G):=0$. Otherwise, $G$ is imprimitive and all minimal blocks for $G$ have cardinality $2$. Let $\Delta$ be such a minimal block, and let $\Gamma:=\{\Delta^g\text{ : }g\in G\}$ be the set of $G$-translates of $\Delta$. Also, let $K:=\Ker_G(\Gamma)$. Then define ${\bl}_{2}(G):=1+\bl_{2}(G/K)$.

For example, a transitive $2$-group $G$ of degree $2^k$ will have $\bl_2(G)=k$. In other words, any tuple of primitive components for $G$ will consist entirely of $\Sym(2)$s. This is because for any prime $p$, any minimal block of any transitive $p$-group has cardinality $p$.

\begin{Remark} If $\bl_{2}(G)\ge 1$, then $G$ has a block of size $2^{\bl_{2}(G)}$, by Remark \ref{WreathBlockRemark}.\end{Remark} 

We can now restate Theorem \ref{TransTheorempre} more precisely as follows.
\begin{Theorem}\label{TransTheorem} Let $G$ be a transitive permutation group of degree $n\ge 2$. Then\begin{enumerate}[(1)]
\item $d(G)\le \left\lfloor \frac{cn}{\sqrt{\log{n}}}\right\rfloor,$where $c:=1512660\sqrt{\log{(2^{19}15)}}/(2^{19}15)=0.920581\hdots$.
\item $d(G)\le \left\lfloor \frac{c_{1}n}{\sqrt{\log{n}}}\right\rfloor,$
where $c_{1}:=\sqrt{3}/2=0.866025\hdots$, unless each of the following conditions hold:\begin{enumerate}[(i)]
\item $n=2^{k}v$, where $v=5$ and $17\le k\le 26$, or $v=15$ and $15\le k\le 35$; 
\item $G$ contains no soluble transitive subgroups; and
\item $\bl_2(G)\ge f$, where $f$ is specified in the middle column of Table A.2 (see Appendix A).\end{enumerate}
In these exceptional cases, the bounds for $d(G)$ in Table A.2 hold.\end{enumerate}\end{Theorem}

Recall that by ``$\log$", we always mean $\log$ to the base $2$. The following is immediate from Theorem \ref{TransTheorem}. Note also that Corollary \ref{TransTheoremCor2} follows immediately from Theorem \ref{TransTheorem}.

As can be seen from the proof of Theorem \ref{TransTheorem}, and the statement of the theorem itself, the cases when $\bl_2(G)$ is large are the most difficult to deal with using our methods. We believe that the finite number of exceptions given in Theorem \ref{TransTheorem} Part (2) are not exceptions at all, that is, we believe that the bound $d(G)\le \lfloor c_{1}n/\sqrt{\log{n}}\rfloor$ should hold for all $n$ and all $G$. 

Note also that, as shown in \cite{KovNew}, the bounds in our results are of the right order. Moreover, the infimum of the set of constants $\overline{c}$ satisfying $d(G)\le \overline{c}n/\sqrt{\log{n}}$, for all soluble transitive permutation groups $G$ of degree $n\ge 2$, is the constant $c_{1}$ in Theorem \ref{TransTheorem}, since $d(G)=4$ when $n=8$ and $G\cong D_{8}\circ D_{8}$. We conjecture that the best ``asymptotic" bound, that is, the best possible upper bound when one is permitted to exclude finitely many cases, is $d(G)\le \lfloor \widetilde{c}n/\sqrt{\log{n}}\rfloor$, where $\widetilde{c}$ is some constant satisfying $b/2\le \widetilde{c}<b=\sqrt{2/\pi}$ (see Example \ref{TransExample} for more details).

In Section \ref{WreathAppSection} we discuss an application of the results of Section \ref{ModuleChapter} to wreath products. We reserve Section \ref{TransProofSection} for the proof of Theorem \ref{TransTheorem}.

\subsection{An application of the results in Section 4 to wreath products}\label{WreathAppSection}
We first make the following easy observation.
\begin{Proposition}\label{ISD} Let $A=T_{1}\times T_{2}\times\hdots\times T_{f}$, where each $T_{i}$ is isomorphic to the nonabelian finite simple group $T$. Suppose that $M\le A$ is a subdirect product of $A$, and suppose that $M'\unlhd M$ is also a subdirect product of $A$. Then $M'=M$.\end{Proposition}
\begin{proof} We prove the claim by induction on $f$, and the case $f=1$ is trivial, so assume that $f>1$. Since $M$ is subdirect, each $M\cap T_{i}$ is normal in $T_{i}$. If $M=A$, then since the only normal subgroups of $A$ are the groups $\prod_{i\in Y} T_{i}$, for $Y\subseteq \{1,\hdots,f\}$, the result is clear. So assume that $M\cap T_{i}=1$ for some $i$. Then $M'\cap T_{i}=1$, and $M'T_{i}/T_{i}$ and $MT_{i}/T_{i}$ are subdirect products of $\prod_{j\neq i} T_{j}$. It follows, using the inductive hypothesis, that $M'T_{i}=MT_{i}$. Hence $M'=M$, since $M\cap T_{i}=1$, and the proof is complete. \end{proof}

We also need the following result of Lucchini and Menegazzo.
\begin{Theorem}[{\bf \cite{Luc2} and \cite{Luc3}}]\label{MinTheorem} Let $L$ be a proper minimal normal subgroup of the finite group $G$. Then $d(G)\le d(G/L)+1$. Furthermore, if $L$ is the unique minimal normal subgroup of $G$, then $d(G)\le \max\left\{2,d(G/L)\right\}$.\end{Theorem}   

We will now fix some notation which will be retained for the remainder of the section.\begin{itemize}
\item Let $R$ be a finite group (we do not exclude the case $R=1$).
\item Let $S$ be a transitive permutation group of degree $s\ge 2$.
\item Let $G$ be a large subgroup of the wreath product $R\wr S$ (see Definition \ref{LargeDef}).
\item Write $B:=R_{(1)}\times R_{(2)}\times\hdots\times R_{(s)}$ for the base group of $R\wr S$.
\item write $\pi:G\rightarrow S$ for the projection homomorphism onto the top group.
\item Let $H:=N_{G}(R_{(1)})=\pi^{-1}(\Stab_S(1))$.
\item Let $\Omega:=H\backslash G$.
\item Let $K:=G\cap B=\core_G(H)=\Ker_G(\Omega)$. \end{itemize}
Recall that for a subgroup $N$ of $R$, $B_N\cong N^s$ denotes the direct product of the distinct $S$-conjugates of $N$. In particular, if $N\unlhd R$, then $B_{N}\unlhd R\wr S$. Throughout, we will view $R$ as a subgroup of $B$ by identifying $R$ with $R_{(1)}$. We also note that 
\begin{itemize}
\item $|G:H|=s$; and
\item $S=G^{\Omega}$.\end{itemize}
In particular, the notation is consistent with the notation introduced at the beginning of Section \ref{MainModuleSection}.
 
\begin{Remark} The results in this section will be obtained by applying the results in Section \ref{ModuleChapter} with $H=H_1$ and $d=1$ (see the notation introduced at the beginning of Section \ref{MainModuleSection}). \end{Remark}
\begin{Remark}\label{NEWREM} If $R$ is a transitive permutation group, acting on a set $\Delta$, then $G$ is an imprimitive permutation group acting on the set $\Delta\times \{1,2,\hdots,s\}$, and $H=\Stab_G((\Delta,1))$. Furthermore $H^{\Delta}=R$, since $G$ is large (see Remark \ref{WreathRemark}).\end{Remark}

Our strategy for proving Theorem \ref{TransTheorem} can now be summarised as follows:\begin{description}
\item[Step 1:] Show that $K$ is ``built" from induced modules for $G$, and non-abelian $G$-chief factors.
\item[Step 2:] Derive bounds on $d(G)$ in terms of the factors from Step 1 and $d(S)$.
\item[Step 3:] Use Theorem \ref{MinTheorem}, together with the results from Section \ref{ModuleChapter}, to bound the contributions from the factors in Step 1 to the bound from Step 2.
\item[Step 4:] Use induction to bound $d(S)$.\end{description}

We begin with Step 1. 
\begin{Lemma}\label{prechief} Suppose that $R>1$ and that $1:=N_0\le N_1\le\hdots\le N_e=R$ is a normal series for $R$, where each factor is either elementary abelian, or a nonabelian chief factor of $R$. Consider the corresponding normal series $1:=G\cap B_{N_0}\le G\cap B_{N_1}\le\hdots\le G\cap B_{N_e}=G$ for $G$. Let $V_i:={N_{i}}/{N_{i-1}}$ and $M_i:=G\cap B_{N_i}/G\cap B_{N_{i-1}}$. \begin{enumerate}[(i)]
\item If $V_i$ is elementary abelian, then $M_i$ is a submodule of the induced module $V_i\uparrow^G_H$.
\item If $V_i$ is a nonabelian chief factor of $R$, then $M_i$ is either trivial, or a nonabelian chief factor of $G$.
\end{enumerate}\end{Lemma}
\begin{proof} Assume first that $V_i$ is elementary abelian, of order $p^a$ say. Then $B_{N_i}/B_{N_{i-1}}$ is a module for $G$ of dimension $as=a|G:H|$ over the finite field of order $p$. Furthermore, $B_{N_i}/B_{N_{i-1}}$ is generated, as a $G$-module, by the $H$-module $V_i$. It now follows from \cite[Corollary 3, Page 56]{Alp} that $B_{N_i}/B_{N_{i-1}}$ is isomorphic to the induced module $V_i\uparrow^G_H$. This proves (i).

Next, suppose that $V_i$ is a nonabelian chief factor of $R$. Write bars to denote reduction modulo $B_{N_{i-1}}$. Then $\ol{G}$ is a large subgroup of the wreath product $\ol{R}\wr S$, and ${\ol{N_i}}$ is a nonabelian minimal normal subgroup of $\ol{R}$. So we just need to prove that $\ol{G}\cap \ol{B_{N_i}}$ is either trivial or a nonabelian minimal normal subgroup of $\ol{G}$. To this end, consider the projection maps 
$$\ol{\rho_{j}}: N_{\ol{G}}(\ol{R_{(j)}})\rightarrow \ol{R_{(j)}}$$
defined in (2.1.1). Suppose that $M$ is a normal subgroup of $\ol{G}$ contained in $\ol{G}\cap \ol{B_{N_i}}$. Then $M\le N_{\ol{G}}(\ol{R_{(1)}})$, and hence $\ol{\rho_1}(M)$ is a normal subgroup of $\ol{\rho_1}(N_{\ol{G}}(\ol{R_{(1)}}))=\ol{R_{(1)}}$ contained in the minimal normal subgroup of $\ol{R_{(1)}}$ corresponding to $\ol{N_i}$. If $\ol{\rho_1}(M)=1$ then $\ol{\rho_j}(M)=1$ for all $j$, since $\pi(\ol{G})=S$ is transitive. Hence, in this case, we have $M=1$. Otherwise, $\ol{\rho_1}(M)\cong \ol{N_i}$, and $M$ is a subdirect product of $s$ copies of $\ol{N_i}$. In this case, since a minimal normal subgroup of a finite group is a direct product of simple groups, we must have $M=\ol{G}\cap \ol{B_{N_i}}$ by Proposition \ref{ISD}. Thus, if $\ol{G}\cap \ol{B_{N_i}}$ is non-trivial, then $\ol{G}\cap \ol{B_{N_i}}$ is a nonabelian minimal normal subgroup of $\ol{G}$, as required.\end{proof}

For the remainder of this section, suppose that $1:=N_0\le N_1\le\hdots\le N_e=R$ is a chief series for $R$, and let $V_i:={N_{i}}/{N_{i-1}}$ and $M_i:=G\cap B_{N_i}/G\cap B_{N_{i-1}}$. If $V_i$ is abelian we will also write $|V_i|=p_{i}^{a_{i}}$, for $p_i$ prime. 

We now have Step 2. 
\begin{Corollary}\label{chiefpreCor} We have
$$d(G)\le \sum_{V_i\text{ abelian}} d_G(M_i)+\nab{(R)}+d(S)$$
\end{Corollary}
\begin{proof} We will prove the corollary by induction on $|R|$. If $|R|=1$ then the bound is trivial, since $G\cong S$ in that case, so assume that $|R|>1$, and note that 
\begin{align}\label{ASTR} G/M_1\text{ is a large subgroup of }(R/V_1)\wr S.\end{align}
Suppose first that $V_1$ is abelian. Then $M_1$ is a $G$-module, so
\begin{align*} d(G)\le d_{G}(M_1)+ d(G/M_1).\end{align*}
Since $\nab{(R)}=\nab{(R/V_1)}$, (\ref{ASTR}) and the inductive hypothesis give the result.

So we may assume that $V_1$ is nonabelian. Then $M_1$ is either trivial or a minimal normal subgroup of $G$, by Lemma \ref{prechief} Part (ii). Hence, $d(G)\le d(G/M_1)+1$ by Theorem \ref{MinTheorem}. The result now follows, again from (\ref{ASTR}) and the inductive hypothesis.\end{proof}     

Before stating our next corollary, we refer the reader to Definition \ref{EDEF} for a reminder of the definitions of the functions $E$ and $E_{sol}$. The next two corollaries deal with Step 3.
\begin{Corollary}\label{chiefCor} Define $E'$ to be $E_{sol}$ if $S$ contains a soluble transitive subgroup, and $E':=E$ otherwise. Then\begin{enumerate}[(i)]
\item $d(G)\le \sum_{V_i\text{ abelian}} a_{i}E'(s,p_{i})+c_{nonab}{(R)}+d(S)$.
\item Suppose that $|R|=2$ and $s=2^{m}q$, where $q$ is odd, and that $S$ has a tuple of primitive components $X=(R_2,\hdots,R_t)$, where $\bl_{X,2}(S)\ge 1$. Let $\Gamma$ be a full set of blocks for $S$ of size $2^{\bl_{X,2}(S)}$, and set $\widetilde{S}:=S^\Gamma$. Then $$d(G)\le \sum_{i=0}^{\bl_{X,2}(S)} E'(2^{m-i}q,2)+d(\widetilde{S}).$$
\item Suppose that $|R|=2$ and $s=2^{m}3$, and that $S$ contains no soluble transitive subgroups. Then by Corollary \ref{MinTransTheoremCor3} there exists a Mersenne prime $p_1=2^{a}-1$ and a triple of integers $(e,t_1,t)$, with $e\ge 1$, and $t\ge t_{1}\ge 0$, such that\begin{enumerate}[(1)]
\item $m=ea+t$, and;
\item There exists a subgroup $N$ of $G$, such that $N^{\Omega}$ is soluble and has $2^{e+t_{1}}$ orbits, with $\binom{e}{k}2^{t_{1}}$ of them of length $3p_1^{k}\times 2^{t-t_{1}}$, for each $0\le k\le e$.\end{enumerate} Here, we have $$d(G)\le \sum_{k=0}^{e}2^{t-t_1}\binom{e}{k}E_{sol}(3p_1^{k}2^{t_1},2)+d(S).$$
\end{enumerate}\end{Corollary}
\begin{proof} By Corollary \ref{chiefpreCor}, we have
$$d(G)\le \sum_{V_i\text{ abelian}} d_G(M_i)+\nab{(R)}+d(S).$$
Now, by Corollary \ref{pmodlemmaMainCor}, $d_G(M_i)\le a_iE'(s,p_i)$. This proves (i). 

To prove (iii) first note that, by Corollary \ref{MinTransTheoremCor3}, and as mentioned in the statement of (iii), there exists a Mersenne prime $p_1:=2^{a}-1$, and a triple $(e,t_1,t)$, with $e\ge 1$, and $t\ge t_{1}\ge 0$, such that \begin{enumerate}[(i)]
\item $m=ea+t$, and;
\item There exists a subgroup $N$ of $G$, such that $N^{\Omega}$ is soluble and has $2^{e+t_{1}}$ orbits, with $\binom{e}{k}2^{t_{1}}$ of them of length $3p_1^{k}\times 2^{t-t_{1}}$, for each $0\le k\le e$.\end{enumerate} Note that, since $|R|=2$, the base group $K\le R^s$ of $G$ is soluble. Hence, since $N^{\Omega}\cong N/N\cap K$ is soluble, it follows that $N$ itself is also soluble. Corollary \ref{pmodlemmaMainCor1} Part (ii)(b) (with $ad=1$) then implies that
$$d_G(M_1)\le \sum_{k=0}^{e}2^{t_1}\binom{e}{k}E_{sol}(3p_1^{k}2^{t-t_1},2)$$
Since $|R|=2$, we have $d(G)\le d_G(M_1)+d(S)$, and the result follows.

Finally, we prove Part (ii). We will show that 
\begin{align}\label{MAD} d(S)\le\sum_{i=1}^{\bl_{X,2}(S)} E(2^{m-i}q,2)+d(\widetilde{S})\end{align}
by induction on $\bl_{X,2}(S)$. The result will then follow, since $d(G)\le E'(2^mq,2)+d(S)$ by Part (i). Now, by hypothesis, $S$ has a tuple of primitive components $X=(R_2,\hdots,R_t)$. Also, $|R_2|=2$ since $\bl_{X,2}(S)\ge 1$. Hence, by Theorem \ref{SupPerm}, $S$ is a large subgroup of a wreath product $R_2\wr S_2$, where either $S_2=1$, or $S_2$ is a transitive permutation group of degree $2^{m-1}q$, with a tuple $Y:=(R_3,\hdots,R_t)$ of primitive components. If $S_2=1$ then the result follows, since $s=4$ and $\widetilde{S}=1$ in that case. So assume that $S_2>1$. By Part (i), we have 
\begin{align}\label{Man} d(S)\le E'(2^{m-1}q,2)+d(S_2)\end{align}
If $\bl_{X,2}(S)=1$ then $S_2=\widetilde{S}$ and (\ref{MAD}) follows from (\ref{Man}). So assume that $\bl_{X,2}(S)>1$. Then $\bl_{Y,2}(S_2)=\bl_{X,2}(S)-1\ge 1$. The inductive hypothesis then yields $d(S_2)\le \sum_{i=1}^{\bl_{Y,2}(S_2)} E(2^{m-1-i}q,2)+d(\widetilde{S})= \sum_{i=2}^{\bl_{X,2}(S)} E(2^{m-i}q,2)+d(\widetilde{S})$. The bound (\ref{MAD}) now follows immediately from (\ref{Man}), which completes the proof.
\end{proof}

The next corollary will be key in our proof of Theorem \ref{TransTheorem} when $G$ is imprimitive with minimal block size $4$. 
\begin{Corollary}\label{S4NewBoundCor} Assume that $R=S_{4}$ or $R=A_{4}$. Define $E'$ to be $E_{sol}$ if $S$ contains a soluble transitive subgroup, and $E':=E$ otherwise. Then 
$$d(G)\le E'(s,2)+\min\left\{\frac{bs}{\sqrt{\log{s_2}}},\frac{s}{s_3}\right\} +E'(s,3)+d(S).$$
\end{Corollary}
\begin{proof} Let $\Delta:=\{1,2,3,4\}$, so that $R$ is transitive on $\Delta$. We have $V_1\cong 2^{2}$, $V_2\cong 3$, and $V_3\cong 2$ if $R\cong S_4$. Since $K^{\Delta}$ is a normal subgroup of $H^{\Delta}=R$ (see Remark \ref{NEWREM}), $K^{\Delta}$ is isomorphic to either $2^{2}$, $A_4$, or $S_4$. In the first two cases $M_{3}$ is trivial, so
$$d(G)\le d_G(M_1)+d_G(M_2)+d(S)\le 2E'(s,2)+E'(s,3)+d(S)$$
by Corollaries \ref{chiefpreCor} and \ref{pmodlemmaMainCor}. So assume that $K^{\Delta}\cong S_4$. Then a Sylow $3$-subgroup $P_3$ of $K^{\Delta}$ acts transitively on the non-identity elements of $V_1$. Thus, $\chi(P_{3}\cap K,V_1^\ast)=1$, so
$$d_G(M_1)\le \min\left\{\frac{bs}{\sqrt{\log{s_2}}},\frac{s}{s_3}\right\}$$
by Corollary \ref{pmodlemmaMainCor1} Parts (iii) and (iv), with $(p,q):=(2,3)$. The result follows.\end{proof}

\subsection{The proof of Theorem \ref{TransTheorem}}\label{TransProofSection}
In this section, we prove Theorem \ref{TransTheorem}. First, we deal with Step 4: the inductive step. As mentioned at the beginning of Section \ref{TransChapter}, the cases where $\bl_2(G)$ is large are the most difficult to deal with using our methods. In these cases, we have $d(G)\le E(s,2)+d(S)$ and usually the bounds on $d(S)$ which come from the inductive hypothesis then suffice to prove the theorem. However in some small cases the inductive hypothesis does not suffice, and we have to work harder. These cases, of which there are finitely many, are the subject of Appendix A, and include both the exceptional cases from Theorem \ref{TransTheorem} (Table A.2), and some additional cases which have a large $2$-part (Table A.1). The purpose of Lemma \ref{48} is to prove that the bounds in Appendix A hold.

Throughout this section, we retain the same notation as introduced immediately following Theorem \ref{MinTheorem}, with one additional assumption: that $R$ is a primitive permutation group of degree $r\ge 2$. Hence, $G$ is a transitive permutation group of degree $n:=rs$, and Remark \ref{NEWREM} applies. Also, set $E'$ to be $E_{sol}$ if $S$ contains a soluble transitive subgroup, and $E':=E$ otherwise. 

Recall also that $p_i^{a_i}$ denote the orders of the abelian chief factors of $R$, for $p_i$ prime.
\begin{Lemma}\label{48} Assume that Theorem \ref{TransTheorem} holds for degrees less than $n$. Then\begin{enumerate}[(i)] 
\item The bounds in Table A.1 (see Appendix A) hold, and;
\item If $n$ and $f$ are as in Table A.2, and either\begin{enumerate}[(a)]
\item $G$ contains a soluble transitive subgroup; or
\item $\bl_2{(G)}<f$,\end{enumerate}
then $d(G)\le \lfloor c_1n/\sqrt{\log{n}}\rfloor$, where $c_1=\frac{\sqrt{3}}{2}$.
\item If $n$ and $f$ are as in Table A.2, and\begin{enumerate}[(a)]
\item $G$ contains no soluble transitive subgroup; and
\item $\bl_2{(G)}\ge f$,\end{enumerate} then, the bounds in Table A.2 (Appendix A) hold.\end{enumerate}\end{Lemma}
\begin{proof} We first recall some bounds which will be used throughout the proof. We have
\begin{align}\label{derekbd} d(G) &\le s\lfloor\log{r}\rfloor+ d(S)\text{, if $r\ge 4$; and}\\
\label{nonderekbd}d(G) &\le \sum_{i}a_{i}E'(s,p_i)+c_{nonab}(R)+d(S). \end{align}
These bounds follow from Corollary \ref{derek} and Corollary \ref{chiefCor} Part (i) respectively.

To bound $d(S)$ above, we use the database of transitive groups of degree up to $32$ in MAGMA (\cite{CanHol}) if $2\le s\le 32$; otherwise, we use either the previous rows of Tables A.1 and A.2; or the bound $d(S)\le \lfloor c_1s/\sqrt{\log{s}}\rfloor$ (from the hypothesis of the lemma) if $s$ is not in Tables A.1 or A.2.

We will first prove (i) and (ii).\begin{description}
\item[(i) and (ii)] The values of $n$ occurring in Table A.1 are $n=2^{m}$ for $6\le m\le 11$; $n=2^{m+1}3$ for $3\le m\le 19$; $n=2^{m}5$ for $3\le m\le 16$; and $n=2^{m}15$ for $2\le m\le 14$. We distinguish a number of cases. Recall that $n=rs$. Throughout, we define $E'':=E_{sol}$ if $s$ is of the form $s=2^m$, and $E'':=E$ otherwise. (Note that a transitive group of prime power degree always contains a soluble transitive subgroup.)\begin{enumerate}
\item $r>16$. Then $d(G) \le s\lfloor\log{r}\rfloor+ d(S)$ by (\ref{derekbd}). Combining this with the bounds on $d(S)$ described above gives the required for each $n$ in Table A.1, and each possible pair $(r,s)$ with $r>16$ and $n=rs$, except when $(n,r,s)=(3145728,24,131072)$. However, each primitive group of degree $24$ is either simple, or has a simple normal subgroup of index $2$ (using the MAGMA \cite{MAGMA} database). Hence, in this case, (\ref{nonderekbd}), together with the hypothesis of the lemma, gives $d(G)\le E(s,2)+1+\lfloor c_1s/\sqrt{\log{s}}\rfloor=52895$. This gives us what we need. 

\item $r=2$. We distinguish two sub-cases.\begin{enumerate}[(a)]
\item $S$ contains a soluble transitive subgroup. Then $d(G)\le E_{sol}(s,2)+d(S)$ by (\ref{nonderekbd}), and this, together with the bounds on $d(S)$ described above gives the bounds in Table A.1 in each of the relevant cases.
\item $S$ contains no soluble transitive subgroups. Then $s$ is not of the form $s=2^m$. We distinguish each of the relevant cases.\begin{enumerate}[i]
\item $s=2^{m}3$, for some $3\le m\le 19$. By using the MAGMA database \cite{MAGMA}, we see that each transitive permutation group of degree $24$ contains a soluble transitive subgroup, so we must have $s=2^m3\ge 48$. In particular, $4\le m\le 19$.  By Corollary \ref{chiefCor} Part (iii) there exists a Mersenne prime $p_1=2^{a}-1$ and a triple of integers $(e,t_1,t)$, with $e\ge 1$, and $t\ge t_{1}\ge 0$, such that $m=ea+t$, and
\begin{align}\label{ARAB} d(G)\le \sum_{k=0}^{e}2^{t-{t_1}}\binom{e}{k}E_{sol}(3p_1^{k}2^{t_1},2)+d(S).\end{align} Since $4\le m\le 19$, the possibilities for $n$ and the triple $(a,e,t)$ are as follows:\\   
\begin{tabular}[t]{|c|p{10.8cm}|}
  \hline
  \multicolumn{2}{ |c| }{Table 5.1}\\ 
  \hline
  $n$ & $(a,e,t)$ \\
  \hline\hline
  $48$ & $(3,1,1)$ \\
  \hline
  $96$ & $(3,1,2)$, $(5,1,0)$ \\
  \hline
  $192$ & $(3,1,3)$, $(3,2,0)$, $(5,1,1)$ \\
  \hline
  $384$ & $(3,1,4)$, $(3,2,1)$, $(5,1,2)$, $(7,1,0)$ \\
  \hline
  $768$ & $(3,1,5)$, $(3,2,2)$, $(5,1,3)$, $(7,1,1)$ \\
  \hline
  $1536$ & $(3,1,6)$, $(3,2,3)$, $(3,3,0)$, $(5,1,4)$, $(7,1,2)$ \\
  \hline 
  $3072$ & $(3,1,7)$, $(3,2,4)$, $(3,3,1)$, $(5,1,5)$, $(7,1,3)$, $(5,2,0)$ \\
  \hline
  $6144$ & $(3,1,8)$, $(3,2,5)$, $(3,3,2)$, $(5,1,6)$, $(7,1,4)$, $(5,2,1)$ \\
  \hline 
  $12288$ & $(3,1,9)$, $(3,2,6)$, $(3,3,3)$, $(3,4,0)$, $(5,1,7)$, $(7,1,5)$, $(5,2,2)$ \\
  \hline
  \end{tabular}\\
  \begin{tabular}[t]{|c|p{4.1cm}|}
  \hline
  \multicolumn{2}{ |c| }{Table 5.1 ctd.}\\ 
  \hline
  $n$ & $(a,e,t)$ \\
  \hline\hline
  $24576$ & $(3,1,10)$, $(3,2,7)$, $(3,3,4)$, $(3,4,1)$, $(5,1,8)$, $(7,1,6)$, $(13,1,0)$, $(5,2,3)$ \\
  \hline
  $49152$ & $(3,1,11)$, $(3,2,8)$, $(3,3,5)$, $(3,4,2)$, $(5,1,9)$, $(7,1,7)$, $(13,1,1)$, $(5,2,4)$, $(7,2,0)$ \\
  \hline 
  $98304$ & $(3,1,12)$, $(3,2,9)$, $(3,3,6)$, $(3,4,3)$, $(3,5,0)$, $(5,1,10)$, $(7,1,8)$, $(13,1,2)$, $(5,2,5)$, $(7,2,1)$, $(5,3,0)$\\
  \hline
  $196608$ & $(3,1,13)$, $(3,2,10)$, $(3,3,7)$, $(3,4,4)$, $(3,5,1)$, $(5,1,11)$, $(7,1,9)$, $(13,1,3)$, $(5,2,6)$, $(7,2,2)$, $(5,3,1)$\\
  \hline
  \end{tabular}
\quad
  \begin{tabular}[t]{|c|p{4.1cm}|}
  \hline
  \multicolumn{2}{ |c| }{Table 5.1 ctd.}\\ 
  \hline
  $n$ & $(a,e,t)$ \\
  \hline\hline
  $393216$ & $(3,1,14)$, $(3,2,11)$, $(3,3,8)$, $(3,4,5)$, $(3,5,2)$, $(5,1,12)$, $(7,1,10)$, $(13,1,4)$, $(17,1,0)$, $(5,2,7)$, $(7,2,3)$, $(5,3,2)$ \\
  \hline
  $786432$ & $(3,1,15)$, $(3,2,12)$, $(3,3,9)$, $(3,4,6)$, $(3,5,3)$, $(3,6,0)$, $(5,1,13)$, $(7,1,11)$, $(13,1,5)$, $(17,1,1)$, $(5,2,8)$, $(7,2,4)$, $(5,3,3)$ \\
  \hline
  $1572864$ & $(3,1,16)$, $(3,2,13)$, $(3,3,10)$, $(3,4,7)$, $(3,5,4)$, $(3,6,1)$, $(5,1,14)$, $(7,1,12)$, $(13,1,6)$, $(17,1,2)$, $(19,1,0)$, $(5,2,9)$, $(7,2,5)$, $(5,3,4)$\\
  \hline
 \end{tabular}
 
 \vspace{5mm}
 
Going through each of the relevant values of $n$ in the first column of Table A.1, each triple $(a,e,t)$ in the last column of Table 5.1, and each possible value of $t_1\le t$, with $n/2=2^{ea+t}3$, the required bound follows from (\ref{ARAB}) each time.
\item $s=2^{m}5$, for some $2\le m\le 15$; or $s=2^m15$ for some $1\le m\le 14$. Then the bound $d(G)\le E(s,2)+d(S)$, together with the bounds on $d(S)$ described above, give the bounds in Table A.1 in each case. \end{enumerate}\end{enumerate}   

\item $r=3$. Here, $d(G)\le E''(s,3)+E''(s,2)+d(S)$, and the bounds from Table A.1 follow in each case from applying the usual upper bounds on $d(S)$.

\item $r=4$. Then
\begin{align}\label{S4bd}d(G)\le E''(s,2)+\min\left\{\frac{bs}{\sqrt{\log{s_2}}},\frac{s}{s_3}\right\}+E''(s,3)+d(S) \end{align}
by Corollary \ref{S4NewBoundCor}. Combining this with the bounds on $d(S)$ described above again gives the bound from the second column of Table A.1 for each of the values of $n$ in the first column, as required.

\item $r=5$. The possible lists of chief factors of the primitive group $R$ of degree $5$ can be obtained from the MAGMA database \cite{MAGMA}. In particular, applying (\ref{nonderekbd}) yields 
$$d(G)\le 2E''(s,2)+E''(s,5)+d(S).$$
Again, combining this with the bounds on $d(S)$ described above yields the required bound from Table A.1 in each case.

\item $r=6$. Again, we take the possible lists of chief factors of the primitive group $R$ of degree $6$ from the MAGMA database \cite{MAGMA}, and apply (\ref{nonderekbd}). We get 
$$d(G)\le E''(s,2)+1+d(S).$$
Combining this with the bounds on $d(S)$ described above yields the required bound from Table A.1 in each of the relevant cases.

\item $r=8$. After obtaining the possible chief factors of $R$ from the MAGMA database, we again apply (\ref{nonderekbd}) and get
$$d(G)\le 3E''(s,2)+E''(s,3)+E''(s,7)+d(S).$$
Using the above with the bounds on $d(S)$ described previously gives the required bound from Table A.1 in each case.

\item $10\le r\le 16$. In each case, we use the same approach as in the previous case, so to avoid being too repetitive we will just check the $r=16$ case. Again we can take the possible lists of chief factors of the primitive groups $R$ of degree $16$ from the MAGMA database, and apply (\ref{nonderekbd}). We get 
$$d(G)\le 7E''(s,2)+E''(s,3)+\max\{E''(s,3),E''(s,5)\}+d(S).$$
As before, combining this with the usual bounds for $d(S)$ gives the bounds in Table A.1 in each case.\end{enumerate}
\item[(iii)] We now consider the bounds in Table A.2., i.e. the exceptional cases from Theorem \ref{TransTheorem}. Thus, either $n=2^{m}5$ and $17\le m\le 26$, or $n=2^{m}15$ and $15\le m\le 35$. Note that $0\le \bl_{2}(G)\le m$. If $\bl_2(G)=0$ then (\ref{derekbd}) for $r>16$, and (\ref{nonderekbd}) for $2<r\le 16$, as in our proofs in {\bf (i) and (ii)} above yields the required bounds in each case. 

So assume that $\bl_2(G)\ge 1$. Then 
\begin{align}\label{fat} d(G)\le \sum_{i=1}^{\bl_{2}(G)} E(2^{m-i}5,2)+d(\widetilde{S})\end{align} 
where $\widetilde{S}$ is transitive of degree $2^{m-\bl_{2}(G)}v$, by Corollary \ref{chiefCor} Part (ii). 

Now, fix a transitive permutation group $G$ of degree $n$ where $n$ is one of the values from the first column of Table A.2. Suppose first that $\bl_2(G)\le f$, where $f$ is the corresponding value to $n$ in the second column of Table A.2. To bound $d(\widetilde{S})$ above, we use the database of transitive permutation groups of degree up to $32$ in MAGMA (see \cite{CanHol}) if $2\le 2^{m-\bl_{2}(G)}v\le 32$; otherwise, we use the previous rows of Tables A.1 and A.2. Combining these bounds for $d(\widetilde{S})$ with (\ref{fat}) yields $d(G)\le \lfloor c_1n/\sqrt{\log{n}}\rfloor$ in each case, as required. 

If $G$ contains a soluble transitive subgroup, then the bound at (\ref{fat}) with $E$ replaced by $E_{sol}$ holds, and yields $d(G)\le \lfloor c_1n/\sqrt{\log{n}}\rfloor$ in each case, as needed. 

So we may assume that $\bl_2(G)>f$, and that $G$ contains so soluble transitive subgroups. In particular, the bound at (\ref{fat}) again holds. If $\widetilde{S}$ is primitive of degree $2^{m-\bl_{2}(G)}v$, then the bound $d(\widetilde{S})\le\lfloor\log{(2^{m-\bl_{2}(G)}v)}\rfloor$ of Theorem \ref{derekthm} gives us the required bound in Table A.2 in each case. So assume that $\widetilde{S}$ is imprimitive, with minimal block size $\widetilde{r}>2$. Also, write $\widetilde{s}:=2^{m-f_{G}}v/\widetilde{r}$. With $(r,s)$ replaced by $(\widetilde{r},\widetilde{s})$, we can now apply (\ref{derekbd}) if $\widetilde{r}>16$, and (\ref{nonderekbd}) for $2<r\le 16$, as in cases (i) and (ii) above. (Note that $d(\widetilde{S})$ is bounded above using the database of transitive permutation groups of degree up to $32$ in MAGMA (see \cite{CanHol}) if $2\le \widetilde{s}\le 32$). This gives us the required bound in Table A.2 in each case. (We perform these calculations for each possible value of $f_{G}$, and each pair $(\widetilde{r},\widetilde{s})$ with $\widetilde{r}>2$ and $2^{m-f_{G}}v=\widetilde{r}\widetilde{s}$.) This completes the proof. \end{description}\end{proof}   

We are now ready to prove Theorem \ref{TransTheorem}.
\begin{proof}[Proof of Theorem \ref{TransTheorem}] The proof is by induction on $n$. Suppose first that $G$ is primitive. The result clearly holds when $n\le 3$. When $n\ge 4$, we have $\log{n}\le c_1n/\sqrt{\log{n}}$, so the result follows immediately from Theorem \ref{derekthm}. This can serve as the initial step.

The inductive step concerns imprimitive $G$. For this, we now use the notation introduced immediately following Theorem \ref{MinTheorem}. Write $V_i$ for the abelian chief factors of $R$, and write $|V_i|=p_i^{a_i}$. Recall that $a(R)$ denotes the composition length of $R$. In particular, $a(R)\ge \sum_{i} a_i+\nab{(R)}$. The inductive hypothesis, together with the bounds obtained in Corollaries \ref{pq} and \ref{derek}, give
\begin{align}\label{1260bd} d(G) &\le \left\lfloor \dfrac{2a(R)s}{c'\log{s}}\right\rfloor+\left\lfloor \dfrac{c_{1}s}{\sqrt{\log{s}}}\right\rfloor &\text{(if $2\le s\le 1260$)}\\
\label{1261bd} d(G) &\le \left\lfloor \dfrac{a(R)b\sqrt{2}s}{\sqrt{\log{s}}}\right\rfloor+\left\lfloor \dfrac{cs}{\sqrt{\log{s}}}\right\rfloor &\text{(if $s\ge 1261$)}\\
\label{shitbd}d(G) &\le \left\lfloor \dfrac{a(R)\frac{2}{c'}s}{\sqrt{\log{s}}}\right\rfloor +\left\lfloor \dfrac{cs}{\sqrt{\log{s}}}\right\rfloor &\text{(for all $s\ge 2$)}\\
\label{derekbd2}d(G) &\le s\lfloor \log{r}\rfloor +\left\lfloor \dfrac{cs}{\sqrt{\log{s}}}\right\rfloor &\text{(for $r\ge 4$, $s\ge 2$)} 
\end{align}respectively. Note that (\ref{1260bd}) and (\ref{1261bd}) follow from Corollaries \ref{pq} and \ref{chiefCor} Part (i), and together imply (\ref{shitbd}), while (\ref{derekbd2}) follows from Corollary \ref{derek}. Recall that we need to prove that $d(G)\le c_1rs/\sqrt{\log{rs}}$ for all cases apart from those listed in Theorem \ref{TransTheorem} Part (2).

Suppose first that $r\ge 481$. Then (\ref{shitbd}), together with Theorem \ref{pyber}, gives $$d(G)\le \frac{([(2+c_{0})\log{r}-(1/3)\log{24}]\frac{2}{c'}+c)s}{\sqrt{\log{s}}}.$$ This is less than $c_{1}rs/\sqrt{\log{rs}}$ for $r\ge 481$ and $s\ge 2$, which gives us what we need. 

So we may assume that $2\le r\le 480$. Suppose first that $10\le r\le 480$, and consider the function 
$$f(e,z,w)=\frac{(eb\sqrt{2}+c)\sqrt{z+w}}{2^{z}\sqrt{w}}$$ defined on triples of positive real numbers. Clearly when the pair $(e,z)$ is fixed, $f$ becomes a decreasing function of $w$. We distinguish two sub-cases:\begin{enumerate}[(a)] 
\item $s\ge 1261$. For each of the cases $10\le r \le 480$, we compute the maximum value $\as(r)$ of the composition lengths of the primitive groups of degree $r$, using MAGMA. Each time, we get $f(\as(r),\log{r},\log{s})\le f(\as(r),\log{r},\log{1261})$\\$<c_{1}$, and the result then follows, in each case, from (\ref{1261bd}).
\item $2\le s\le 1260$. For each fixed $r$, $10\le r\le 480$, and each $s$, $2\le s\le 1260$, we explicitly compute $\min{\left\{{\lfloor 2\as(r)s/(c'\log{s})\rfloor,s\lfloor\log{r}\rfloor}\right\}} +\lfloor c_{1}s/\sqrt{\log{s}}\rfloor$. Each time, except when $r=16$ and $72\le s\le 1260$, this integer is less than or equal to $\lfloor c_{1}rs/\sqrt{\log{rs}}\rfloor$, which, after appealing to the inequalities at (\ref{1260bd}) and (\ref{derekbd2}), gives us what we need. If $r=16$, and $72\le s\le 1260$, we have $d(G)\le 7E(s,2)+2E(s,3)+\lfloor c_{1}s/\sqrt{\log{s}}\rfloor$, by Corollary \ref{chiefCor} Part (i), and this gives the required bound in each case (the chief factors of the primitive groups of degree $16$ are computed using MAGMA - see Table B.2).\end{enumerate}

Finally, we deal with the cases $2\le r\le 9$. In considering each of the relevant cases, we take the possible lists of chief factors of $R$ from the MAGMA database. In each case, we bound $d(S)$ above by using the database of transitive permutation groups of degree up to $32$ in MAGMA (see \cite{CanHol}) if $2\le s\le 32$, Lemma \ref{48} if $s$ is in the left hand column of Table A.1 or Table A.2, or the inductive hypothesis otherwise.\begin{enumerate}[(a)]
\item $r=2$. Corollary \ref{chiefCor} Part (i) gives $d(G)\le E(s,2)+d(S)$. Write $s=2^{m}q$, where $q$ is odd, and assume first that $s<10^{66}$. Assume first that $\lpp(q)\ge 19$. Then $d(G)\le s/19+d(S)$, and the bounds on $d(S)$ described above, yield $d(G)\le 2c_{1}s/\sqrt{\log{2s}}$ for $s<10^{66}$. So assume further that $\lpp(q)\le 17$. Then $q$ is of the form $q=3^{l_{3}}5^{l_{5}}7^{l_{7}}11^{l_{11}}13^{l_{13}}17^{l_{17}}$, where $0\le l_{3}\le 2$, and $0\le l_{i}\le 1$, for $i=5$, $7$, $11$, $13$  and $17$. Fix one such $q$. Then $0\le m\le m(q):=\lfloor\log{(10^{66}/q)}\rfloor$, and $d(G)\le E(2^{m}q,2)+d(S)$. Now, by using the upper bounds on $d(S)$ described above, we get $d(G)\le 2c_{1}s/\sqrt{\log{2s}}$, for each of the $96$ possible values of $q$, and each $0\le m\le m(q)$. This gives us what we need.

Thus, we may assume that $s\ge 10^{66}$. We distinguish two sub-cases.\begin{enumerate}[(i)]
\item $s_{2}\ge s^{858/1000}$. Then $E(s,2)\le bs/\sqrt{\log{s_{2}}}\le bs\sqrt{1000/858}/\sqrt{\log{s}}$. Hence, $d(G)\le bs\sqrt{1000/858}/\sqrt{\log{s}}+c_{1}s/\sqrt{\log{s}}$, and this is less than or equal to $2c_{1}s/\sqrt{\log{2s}}$ for $s\ge 10^{66}$, as required. 
\item $s/s_{2}\ge s^{142/1000}$. Then, by Lemma \ref{primecount}, we have 
$$E(s,2)\le s/(c'\log{(s/s_{2})})\le (1000/142)s/c'\log{s},$$ and hence $d(G)\le (1000/142)s/(c'\log{s})+c_{1}s/\sqrt{\log{s}}$. Again, this is less than or equal to $2c_{1}s/\sqrt{\log{2s}}$, for $s\ge 10^{66}$.\end{enumerate}
\item $r=3$. Here, Corollary \ref{chiefCor} Part (i)  gives $d(G)\le E(s,3)+E(s,2)+d(S)$. Using the bounds for $d(S)$ described above, this gives us what we need whenever $2\le s\le 5577$, and whenever $S$ is one of the exceptional cases listed in Theorem \ref{TransTheorem} Part (2) )in these cases, we take the bounds for $d(S)$ from Table A.2). Otherwise, $s\ge 5578$, and we use Corollary \ref{pq} to distinguish two cases, with $\alpha=1/3$.\begin{enumerate}[(i)]
\item $s_{2}$, $s_{3}\le s^{1/3}$. Then $d(G)\le 3s/(c'\log{s})+c_{1}s/\sqrt{\log{s}}$, and this is less than or equal to $3c_{1}s/\sqrt{\log{3s}}$ for $s\ge 3824$.
\item $s_{2}\ge s^{1/3}$, or $s_{3}\ge s^{1/3}$. Then $\lpp{(s/s_3)}\ge s^{1/3}$ or $\lpp{(s/s_2)}\ge s^{1/3}$, so $d(G)\le b\sqrt{3}s/\sqrt{\log{s}}+{s}^{2/3}+c_{1}s/\sqrt{\log{s}}$, and this is at most $3c_{1}s/\sqrt{\log{3s}}$, for $s\ge 5578$.\end{enumerate}  
\item $r=4$. Here Corollary \ref{S4NewBoundCor} implies that 
\begin{align}\label{aye} d(G)\le E(s,2)+\min\left\{\frac{bs}{\sqrt{\log{s_2}}},\frac{s}{s_3}\right\} +E(s,3)+d(S).\end{align}
Using the bounds on $d(S)$ described above, this yields the required upper bound whenever $S$ is one of the exceptional cases of Theorem \ref{TransTheorem} Part (2), and whenever $7\le s\le 49435925$. When $2\le s\le 6$, $G$ is transitive of degree $4s$, and the result follows by using Table B.1. So assume that $s\ge 115063$, and that $s$ is not one of those cases listed in Theorem \ref{TransTheorem} Part (2). We distinguish three cases.\begin{enumerate}[(i)]
\item $s_{2}$, $s_{3}\le s^{21/50}$. Then $d(G)\le (200/29)s/(c'\log{s})+c_1s/\sqrt{\log{s}}$ by Corollary \ref{pq} (with $alpha=21/50$), and this is less than or equal to $4c_{1}s/\sqrt{\log{4s}}$ for $s\ge 49435925$, as needed.
\item $s_{2}\ge s^{21/50}$. Then $E(s,2)\le\sqrt{50/21}bs/\sqrt{\log{s}}$, and $E(s,3)\le s/\lpp{(s/s_3)}\le s/s_2\le s^{29/50}$. Hence, $d(G)\le 2\sqrt{50/21}bs/\sqrt{\log{s}}+s^{29/50}+c_{1}s/\sqrt{\log{s}}$ by (\ref{aye}). This is at most $4c_{1}s/\sqrt{\log{4s}}$, for $s\ge 28090868$.
\item $s_{3}\ge s^{21/50}$. Then $d(G)\le \sqrt{50/21}bs/\sqrt{\log{s}}+2s^{29/50}+c_{1}s/\sqrt{\log{s}}$ using a similar argument to (ii) above. This is less than or equal to $4c_{1}s/\sqrt{\log{4s}}$, for $s\ge 56$. This completes the proof of the theorem in the case $r=4$.\end{enumerate}
\item $r=5$. Corollary \ref{chiefCor} Part (i) gives $d(G)\le E(s,5)+2E(s,2)+d(S)$. Again, this gives us what we need for each $s$ in the range $3\le s\le 552$, and each exceptional $S$. Also, $s=2$ implies that $G$ is transitive of degree $10$, and the result follows from Table B.1. Thus, we may assume that $s\ge 553$. Applying Corollary \ref{pq}, with $\alpha=2/5$, yields three cases.\begin{enumerate}[(i)]
\item $s_{2}$, $s_{5}\le s^{2/5}$. Then $d(G)\le 5s/(c'\log{s})+c_{1}s/\sqrt{\log{s}}$, which is less than or equal to $5c_{1}s/\sqrt{\log{5s}}$ for $s\ge 553$, as required.
\item $s_{2}\ge s^{2/5}$. Then $d(G)\le 2b\sqrt{5/2}s/\sqrt{\log{s}}+{s}^{3/5}+c_{1}s/\sqrt{\log{s}}$, and this is no greater than $5c_{1}s/\sqrt{\log{5s}}$ when $s\ge 139$.
\item $s_{5}\ge s^{2/5}$. Then $d(G)\le b\sqrt{5/2}s/\sqrt{\log{s}}+2{s}^{3/5}+c_{1}s/\sqrt{\log{s}}$, which is less than or equal to $5c_{1}s/\sqrt{\log{5s}}$ for $s\ge 17$.\end{enumerate}
\item $r=6$. Here, Corollary \ref{chiefCor} Part (i), together with the inductive hypothesis, gives $d(G)\le E(s,2)+1+d(S)$. Using the usual bounds on $d(S)$, this is at most $\lfloor 6cs/\sqrt{\log{6s}}\rfloor$ for $2\le s\le 1260$, and whenever $S$ is one of the exceptional cases. Otherwise, $s\ge 1261$, and $d(S)\le c_{1}s/\sqrt{\log{s}}$. Hence, by Corollary \ref{pq} Part (iii), $d(G)\le b\sqrt{2}s/\sqrt{\log{s}}+1+c_1s/\sqrt{\log{s}}$, which is less than or equal to $6c_{1}s/\sqrt{\log{6s}}$ for $s\ge 2$. This completes the proof of the theorem in the case $r=6$.
\item $r=7$. Here, $d(G)\le E(s,2)+E(s,3)+E(s,7)+d(S)$, again using Corollary \ref{chiefCor} Part (i). Bounding $d(S)$ as described previously, this is at most $\lfloor 7c_{1}s/\sqrt{\log{7s}}\rfloor$ for each $s$ in the range $2\le s\le 1260$, and each exceptional $S$. Otherwise, $s\ge 1261$, and by Corollary \ref{pq} Part (iii) $d(G)\le 3b\sqrt{2}s/\sqrt{\log{s}}+ c_{1}s/\sqrt{\log{s}}$. This is less than $7c_{1}s/\sqrt{\log{7s}}$ for $s\ge 7$, and, again, we have what we need.
\item $r=8$. Using Corollary \ref{chiefCor} Part (i), $d(G)\le 3E(s,2)+E(s,3)+E(s,7)+d(S)$. In each of the cases $2\le s\le 272$, and each exceptional case, this bound, together with the bounds on $d(S)$ described above, give us what we need. Thus, we may assume that $s\ge 273$. Then the inductive hypothesis gives $d(S)\le c_{1}s/\sqrt{\log{s}}$, and applying Corollary \ref{pq}, with $\alpha=37/100$, yields three cases.\begin{enumerate}[(i)]
\item $\max\left\{s_{2},s_{3},s_{7}\right\}\le s^{37/100}$. Then $d(G)\le (500/63)s/(c'\log{s})+c_{1}s/\sqrt{\log{s}}$, which is less than or equal to $8c_{1}s/\sqrt{\log{8s}}$ for $s\ge 273$, as required.
\item $s_{2}\ge s^{37/100}$. Then $d(G)\le 3b\sqrt{100/37}s/\sqrt{\log{s}}+2{s}^{63/100}+c_{1}s/\sqrt{\log{s}}$, and this is no greater than $8c_{1}s/\sqrt{\log{8s}}$ when $s\ge 98$.
\item $\max\left\{s_{3},s_{7}\right\}\ge s^{37/100}$. Then $d(G)\le 2b\sqrt{100/37}s/\sqrt{\log{s}}+3{s}^{63/100}+c_{1}s/\sqrt{\log{s}}$, which is less than or equal to $8c_{1}s/\sqrt{\log{8s}}$ for $s\ge 27$.\end{enumerate}  
\item $r=9$. By Corollary \ref{chiefCor} Part (i), $d(G)\le 4E(s,2)+3E(s,3)+d(S)$. When $3\le s\le 2335$, and when $S$ is one of the exceptional cases, this bound, together with the usual bounds on $d(S)$, give us what we need. If $s=2$, then $G$ is transitive of degree $18$, and the result follows from Table A.1. Otherwise, $s\ge 2336$, and $d(S)\le c_{1}s/\sqrt{\log{s}}$, using the inductive hypothesis. We now use Corollary \ref{pq} to distinguish three cases, with $\alpha=37/100$.\begin{enumerate}[(i)]  
\item $s_{2}$, $s_{3}\le {s}^{37/100}$. Then $d(G)\le (700/63)s/(c'\log{s})+c_{1}s/\sqrt{\log{s}}$, and this is less than or equal to $9c_{1}s/\sqrt{\log{9s}}$ for $s\ge 2336$, as needed.
\item $s_{2}\ge {s}^{37/100}$. Then $d(G)\le 4b\sqrt{100/37}s/\sqrt{\log{s}}+3{s}^{63/100}+c_{1}s/\sqrt{\log{s}}$, which is no larger than $9cs/\sqrt{\log{9s}}$, whenever $s\ge 1197$.
\item $s_{3}\ge {s}^{37/100}$. Here, $d(G)\le 3b\sqrt{100/37}s/\sqrt{\log{s}}+4{s}^{63/100}+c_{1}s/\sqrt{\log{s}}$, and this is less than or equal to $9c_{1}s/\sqrt{\log{9s}}$ for $s\ge 148$.\end{enumerate}\end{enumerate}
This completes the proof of Theorem \ref{TransTheorem}.\end{proof}

\section{The proof of Theorem \ref{TransOrderTheorem}}\label{TransOrderChapter}
In proving Theorem \ref{TransOrderTheorem}, we will omit reference to the constant $C$, and just use the Vinogradov notation defined immediately after Definition \ref{HypA}. We will now restate some results from Sections 2, 3 and 4 in this language for the convenience of the reader. 

We begin with Theorems \ref{pyber} and \ref{TransTheorempre}.
\begin{Theorem}\label{pyber1} Let $R$ be a primitive permutation group of degree $r$. Then $a{(R)}\ll \log{r}$.\end{Theorem}

\begin{Theorem}\label{TransTheorem1} Let $S$ be a transitive permutation group of degree $s\ge 2$. Then $d(S)\ll  s/\sqrt{\log{s}}$.\end{Theorem}

We also note the following useful consequence of Corollaries \ref{chiefpreCor} and \ref{pq}, and Theorem \ref{TransTheorem1}. 
\begin{Corollary}\label{UsefulCor} Let $R$ be a finite group, let $S$ be a transitive permutation group of degree $s\ge 2$, and let $G$ be a large subgroup of the wreath product $R\wr S$. Then
$$d(G)\ll \frac{a(R)s}{\sqrt{\log{s}}}.$$\end{Corollary}

Theorem \ref{derekthm} reads as follows in Vinogradov notation.
\begin{Theorem}[{\bf\cite{derek}, Theorem 1.1}]\label{derekthm1} Let $H$ be a subnormal subgroup of a primitive permutation group of degree $r$. Then $d(H)\ll \log{r}$.\end{Theorem}

Finally, we will need the following theorem of Cameron, Solomon and Turull; note that we only give a simplified version of their result here. 
\begin{Theorem}[{\bf\cite{Cam}, Theorem 1}]\label{Cam} Let $G$ be a permutation group of degree $n\ge 2$. Then $a{(G)}\ll n$.\end{Theorem}

\subsection{Orders of transitive permutation groups}
We now turn to bounds on the order of a transitive permutation group $G$, of degree $n$. First, we fix some notation which will be retained for the remainder of the section. Let $G$ be a transitive permutation group of degree $n$, and let $(R_{1},\hdots,R_t)$ be a tuple of primitive components for $G$, where each $R_i$ is primitive of degree $r_i$, and $\prod_i r_i=n$. Furthermore, we will write $\pi_{1}$ for the identity map $G\rightarrow G$, and for $i\ge 2$, we will write $\pi_{i}$ to denote the projection $\pi_{i}:G\pi_{i-1}\le R_{i-1}\wr (R_{i}\wr R_{i+1}\wr\hdots\wr R_{t})\rightarrow R_{i}\wr R_{i+1}\wr \hdots\wr R_{t}$.

The following is a simplified version of a theorem of C. Praeger and J. Saxl \cite{PraeSax} (which was later improved by A. Mar\'{o}ti in \cite{AMaroti}).
\begin{Theorem}[{\bf\cite{PraeSax}, Main Theorem}]\label{Maroti} Let $G$ be a primitive permutation group of degree $r$, not containing $\Alt(r)$. Then $\log{|G|}\ll r$.\end{Theorem}

Since the symmetric and alternating groups are $2$-generated, the next corollary follows immediately from Theorems \ref{derekthm1} and \ref{Maroti}.
\begin{Corollary}\label{PrimOrderBound} Let $G$ be a subnormal subgroup of a primitive permutation group of degree $r$. Then $d(G)\log{|G|}\ll r\log{r}$.\end{Corollary}

\subsection{The proof of Theorem \ref{TransOrderTheorem}}
Before proceeding to the proof of Theorem \ref{TransOrderTheorem}, we require an application of the results in Section \ref{WreathAppSection}. First, we need a preliminary lemma.
\begin{Lemma}\label{Chiefit}Let $R$ and $S$ be transitive permutation groups of degree $r\ge 2$ and $s\ge 1$ respectively, let $D$ be a subgroup of $\Sym(d)$ containing $\Alt(d)$, let $P$ be a large subgroup of the wreath product $D\wr S$, and let $G$ be a large subgroup of $R\wr P$. Also, write $U_{i}$ for the abelian chief factors of $R$. Suppose that $d\ge 5$. Then\begin{enumerate}[(i)]
\item There exists a large subgroup $Q$ of the wreath product $R\wr D$, and an embedding $\theta:G\rightarrow Q\wr S$, such that $G\theta$ is a large subgroup of $Q\wr S$.
\item Let $H:=N_{Q}(R_{(1)})$. Then $Q$ has a normal series
$$1=N_{0}\le N_{1}\le \hdots<N_{t}<N_{t+1}\le N_{t+2}=Q,$$ where for each abelian $U_i$ with $i\le t$, $N_{i}/N_{i-1}$ is contained in the $Q$-module $U_{i}\uparrow^Q_H$; and for each non-abelian $U_i$ with $i\le t$, $N_i/N_{i-1}$ is either trivial or a non-abelian chief factor of $Q$. Also, $N_{t+1}/N_{t}\cong \Alt(d)$, and $|N_{t+2}/N_{t+1}|\le 2$. \end{enumerate}
\end{Lemma}
\begin{proof} Note first that $G$ is an imprimitive permutation group of degree $rds$, with a block $\Delta_1$ of size $r$, by Remark \ref{WreathBlockRemark}. Now, by Remark \ref{WreathRemark}, $G$ is also a subgroup of the wreath product $X:=(R\wr D)\wr S$. Hence, $G$ also has a block of size $rd$, again using Remark \ref{WreathBlockRemark}. Let $\Delta$ be a block of size $rd$ containing $\Delta_1$. Let $H_1:=\Stab_G(\Delta_1)$ and $H:=\Stab_G(\Delta)=N_Q(R_{(1)})$. Then $H_1\le H$, and $\Delta_1$ is a block for $H^{\Delta}$ of size $r$, with block stabiliser $H_{1}^{\Delta}$. Let $\Gamma_1$ be the set of $H$-translates of $\Delta_1$, and let $\Gamma$ be the set of $G$-translates of $\Delta$. Then $G$ is a large subgroup of $H^\Delta\wr G^\Gamma$, while $H^{\Delta}$ is a large subgroup of $H_1^{\Delta_1}\wr H^{\Gamma_1}$, by Theorem \ref{SupPerm}. By Definition \ref{LargeDef}, $H_1^{\Delta_1}\cong R$. Thus, to complete the proof of Part (i) we just need to show that $H^{\Gamma_1}\cong D$ and $G^{\Gamma}\cong S$ (we then take $Q=H^{\Delta}$).

First, let $\pi:G\le R\wr P\rightarrow P$ denote projection over the top group. Note that $H\pi\le P$ is a permutation group of degree $ds$, stabilising a block of size $d$. Furthermore, since $\Ker(\pi)=\core_G(H_1)\le H_1\le H$, we have $s=|G:H|=|G\pi:H\pi|$. Thus, $H\pi$ is the full (set-wise) stabiliser of a block for $P$ of size $d$. It follows that $H^{\Gamma_1}\cong D$, since $P$ is large in $D\wr S$. 

Since $\Ker(\pi)=\Ker_G(\Delta_1^G)\le \Ker_G(\Gamma)$, we have $G^{\Gamma}\cong \pi(G)^{\Gamma}=P^{\Gamma}=S$, as needed. Finally, since $Q$ is a large subgroup of $R\wr D$, and $D\cong \Alt(d)$ or $D\cong\Sym(d)$, Part (ii) follows from Lemma \ref{prechief}.
\end{proof}      

The mentioned application can now be given as follows.
\begin{Proposition}\label{WreathApp}Let $R$ be a finite group, let $S$ be a transitive permutation group of degree $s\ge 2$, let $D$ be a subgroup of $\Sym(d)$ containing $\Alt(d)$, let $P$ be a large subgroup of the wreath product $D\wr S$, and let $G$ be a large subgroup of $R\wr P$. Also, let $K_1$ be the kernel of the action of $P\le D\wr S$ on a set of blocks of size $d$, and let $A$ be the induced action of $K_1$ on a fixed block $\Delta$ for $P$. Assume that $A\neq 1$, that $d\ge 5$, and set $g(d,s):=\max\{1,\frac{d}{\sqrt{\log{s}}}\}$. Then\begin{enumerate}[(i)]
\item $d(G)\ll a(R)s$; and
\item $d(G)\ll\frac{a(R)g(d,s)s}{\sqrt{\log{s}}}$.\end{enumerate}
\end{Proposition}
\begin{proof} Let $U_{i}$, for $1\le i\le t$ say, denote the chief factors of $R$. Also, if $U_{i}$ is abelian, write $|U_{i}|=p_{i}^{a_{i}}$, for $p_{i}$ prime. By Lemma \ref{Chiefit} Part (i), $G$ is a large subgroup of $Q\wr S$, where $Q$ is a large subgroup of $R\wr D$. Let $H_1:=N_Q(R_{(1)})$. By Lemma \ref{Chiefit} Part (ii), $Q$ has a normal series
$$1=N_{0}\le N_{1}\le \hdots\le N_{t}<N_{t+1}\le N_{t+2}=Q,$$ where each abelian factor $N_{i}/N_{i-1}$, for $i\le t$, is contained in the $Q$-module $U_{i}\uparrow^Q_{H_1}$, and each nonabelian factor is a chief factor of $Q$. Also, $N_{t+1}/N_{t}\cong \Alt(d)$, and $|N_{t+2}/N_{t+1}|\le 2$. In particular, \begin{align}\label{First} \nab{(Q)}\le \nab{(R)}+1.\end{align}
 Denote by $B$ the base group of $Q\wr S$, and consider the corresponding normal series
\begin{align}\label{fo} 1 &=G\cap B_{N_{0}}\le G\cap B_{N_{1}}\le G\cap B_{N_{2}} \le \hdots\le G\cap B_{N_{t}}\\
 & < G\cap B_{N_{t+1}}\le G\cap B_{N_{t+2}}=G\cap B\end{align}
for $G\cap B$. Let $M_{i}$ be the abelian factors in (\ref{fo}). Then
\begin{align}\label{Second}
d(G)\ll \sum_{U_{i}\text{ abelian}} d_{G}(M_{i})+\nab{(R)}+\frac{s}{\sqrt{\log{s}}}\end{align}
by Corollary \ref{chiefpreCor} and Theorem \ref{TransTheorem1}. Viewing $G$ as a subgroup of $Q\wr S$, let $H:=N_{G}(Q_{(1)})$. Also, let $\pi:R\wr P\rightarrow P$ denote projection over the top group. Since $H\pi\le P$ stabilises a block of size $d$, we may assume, without loss of generality, that 
$$H\pi={\Stab}_P(\Delta)$$
(recall that $\Delta$ is a block of size $d$ for $P\le D\wr S$). Note also that $M_i$ is a submodule of the induced module $U_i\uparrow^H_{H_1}\uparrow^G_H\cong U_i\uparrow^G_{H_1}$, by Lemmas \ref{prechief} and \ref{First}. 

Fix $i$ in the range $1\le i\le t$ such that $U_i$ is abelian. Suppose first that $s_{p_{i}}\le \sqrt{s}$. Then Corollary \ref{pq} Part (ii), with $\alpha:=1/2$, gives
\begin{align}\label{Third} d_{G}(M_{i})\ll \frac{a_{i}ds}{\log{s}}\le\frac{a_{i}g(d,s)s}{\sqrt{\log{s}}}\end{align}
Assume next that $s_{p_{i}}>\sqrt{s}$ for some fixed $i$. Let $K:=\core_{G}(H)$. Note that $K\pi=K_1\le P$, since $H\pi=\Stab_P(\Delta)$ is a block stabiliser. Then $$1<A=(K\pi)^\Delta\unlhd (H\pi)^\Delta=D,$$ so $(K\pi)^{\Delta}\ge \Alt(d)$. Hence, Proposition \ref{MainModuleTheorem} Part (ii) implies that
\begin{align}\label{Fourth} d_{G}(M_{i})\ll \frac{a_{i}s}{\sqrt{\log{s_{p_{i}}}}}\le \frac{\sqrt{2}a_{i}s}{\sqrt{\log{s}}}\ll\frac{a_{i}g(d,s)s}{\sqrt{\log{s}}}.\end{align}

Thus, (\ref{Second}), (\ref{Third}) and (\ref{Fourth}) yield:\begin{align*}
d(G) &\ll \sum_{U_{i}\text{ abelian}} \frac{a_{i}g(d,s)s}{\sqrt{\log{s}}}+\nab{(R)}+\frac{s}{\sqrt{\log{s}}}\\
&\ll \frac{a(R)g(d,s)s}{\sqrt{\log{s}}}+\frac{s}{\sqrt{\log{s}}}\\
&\ll \frac{a(R)g(d,s)s}{\sqrt{\log{s}}}+\frac{g(d,s)s}{\sqrt{\log{s}}}\ll \frac{a(R)g(d,s)s}{\sqrt{\log{s}}}\end{align*}
and this proves Part (ii).

Finally, \ref{Second} and Proposition \ref{MainModuleTheorem} Part (i) give 
\begin{align*}
d(G) &\ll \sum_{U_{i}\text{ abelian}} a_{i}s+\nab{(R)}+\frac{s}{\sqrt{\log{s}}}\\
&\ll {a(R)s}+\frac{s}{\sqrt{\log{s}}}\ll a(R)s\end{align*}
and this completes the proof.\end{proof}

We are now ready to prove Theorem \ref{TransOrderTheorem}.       
\begin{proof}[Proof of Theorem \ref{TransOrderTheorem}] Let $f(G)=d(G)\log{|G|}\sqrt{\log{n}}/n^{2}$. We will prove, by induction on $n$, that $f(G)\ll 1$. If $G$ is primitive, then $f(G)\ll (\log{n})^{3/2}/n$ by Corollary \ref{PrimOrderBound}, and the claim follows.

For the inductive step, assume that $G$ is imprimitive. Fix a tuple \\$(R_{1}, R_{2},\hdots,R_{t})$ of primitive components for $G$, where each $R_{i}$ is primitive of degree $r_{i}$, say. Also, for $1\le i\le t-1$, let $\Delta_{i}$ be a block of size $r_{i}$ for $\pi_{i}(G)\le$ $R_{i}\wr \pi_{i+1}{(R_{i})}$, and denote by $A_{i}$ the induced action of $\Ker_{\pi_{i}(G)}( \{{\Delta_i}^{g}\text{ : }g\in \pi_{i}(G)\})$ on $\Delta_i$ (in particular, note that $A_{i}\unlhd R_{i}$). Finally, set $A_{t}:=\pi_{t}(G)$. Then
\begin{align}\label{OrderBound}|G|\le \prod_{i=1}^{t}|A_{i}|^{\frac{n}{r_{1}\hdots r_{i}}}\end{align}

Next, for $1\le i\le t$, we define the functions $f_{i}$ as follows
\begin{align}\label{MainBoundfi} f_{i}(G):=\frac{d(G)n\log{|A_{i}|}\sqrt{\log{n}}}{r_{1}r_{2}\hdots r_{i}n^{2}}=\frac{d(G)\log{|A_{i}|}\sqrt{\log{n}}}{r_{1}r_{2}\hdots r_{i}n}\end{align}
The inequality at \ref{OrderBound} then yields $f(G)\le \sum_{i=1}^{t} f_{i}(G)$. We claim that $f_{i}(G)\ll\frac{(i-1)}{2^{i-1}}$ for $2\le i\le t$, and that $f_{1}(G)\ll 1$ (the implied constants here are independent on $i$). The result will then follow. Indeed, in this case, $f(G)\ll \sum_{i=1}^{\infty}\frac{{i-1}}{2^{i-1}}\ll 1$.

To this end, first fix $i$ in the range $2\le i\le t$. Clearly we may assume that $A_{i}$ is non-trivial. Let $D=R_{i}$, $S:=\pi_{i}(G)$, and note that $G$ is a large subgroup of a wreath product $R\wr P$, where $R$ is transitive of degree $r:=r_{1}r_{2}\hdots r_{i-1}$, and $P$ is a large subgroup of $D\wr S$. Set $d:=r_{i}$, $s:=r_{i+1}\hdots r_{t}$, and $m:=\max\left\{r,d,s\right\}$. Suppose first that $d\ge 5$ and that $D$ contains the alternating group $\Alt(d)$. (In particular, we are in the ``bottom heavy" situation of Proposition \ref{WreathApp}.) Then $A_i$, being a nontrivial normal subgroup of $D$, also contains $\Alt(d)$. Note that $|A_{i}|\le d^d$. We distinguish two cases. Note throughout that $\log{n}\le \log{m^3}\ll \log{m}$. \begin{enumerate}
\item $s\le 2^{(\log{d})^{2}}$. Then $n=rds\le m_{1}^{2}2^{(\log{m_{1}})^{2}}$, where $m_{1}:=\max\left\{r,d\right\}$. Thus, $\log{n}\le 2\log{m_{1}}+(\log{m_{1}})^{2}\ll (\log{m_{1}})^{2}$. Since $a(R)\ll r$ by Theorem \ref{Cam}, Proposition \ref{WreathApp} Part (i) then implies that $d(G)\ll rs$. Hence, from \ref{MainBoundfi} we deduce 
\begin{align*}f_{i}(G)&\ll \frac{rsd\log{d}\log{m_{1}}}{r^{2}d^{2}s}
=\frac{\log{d}\log{m_{1}}}{rd}
\ll \frac{\log{r}}{r}\le \frac{(i-1)}{2^{i-1}}\end{align*}
since $r\ge 2^{i-1}$, and this gives us what we need.

\item $s> 2^{(\log{d})^{2}}$. Note that $m\in \{r,s\}$ in this case. Set $g(d,s):=\max\left\{1,\frac{d}{\sqrt{\log{s}}}\right\}$. Then 
\begin{align}\label{ASTON}g(d,s)\log{d}&\le d \end{align}
since $\sqrt{\log{s}}>\log{d}$. Now, Theorem \ref{Cam} gives $a(R) \ll r$. Hence, Proposition \ref{WreathApp} Part (ii) gives $d(G)\ll \frac{rg(d,s)s}{\sqrt{\log{s}}}$. Hence, since $n\le m^{3}$, we have 
\begin{align*}f_{i}(G)&\ll \frac{rg(d,s) sd\log{d}\sqrt{\log{m}}}{r^{2}d^{2}s\sqrt{\log{s}}}\\
&= \frac{ g(d,s)\log{d}\sqrt{\log{m}}}{rd\sqrt{\log{s}}}\\
&\le \frac{ d\sqrt{\log{m}}}{rd\sqrt{\log{s}}}&\text{by }(\ref{ASTON}),\\
&\le \frac{\sqrt{\log{r}}}{r}\le \frac{\sqrt{i-1}}{2^{i-1}}&\text{since }m\in\{r,s\}.\end{align*}
This gives us what we need.\end{enumerate}

Next, suppose that either $d\le 4$, or that $D$ does not contain $\Alt(d)$. Then $\log{|A_{i}|}\ll {d}$ by Theorem \ref{Maroti}. Now, $G$ is a large subgroup of $R\wr P$, where $P$ is transitive of degree $ds$. Also, $a(R)\ll r$ by Theorem \ref{Cam}. Then, by Corollary \ref{UsefulCor} we have
$$d(G)\ll \frac{rds}{\sqrt{\log{ds}}}.$$Thus
\begin{align*}f_i(G)\ll \frac{rdsd\sqrt{\log{m}}}{r^{2}d^{2}s\sqrt{\log{ds}}}
= \frac{\sqrt{\log{m}}}{r\sqrt{\log{ds}}}
\le \frac{\sqrt{\log{r}}}{r}\le \frac{\sqrt{i-1}}{2^{i-1}}\end{align*}
and again this gives us what we need.
  
Finally, we deal with the case $i=1$. Here, set $r:=r_{1}$, $s:=r_{2}r_{3}\hdots r_{t}$, and $m=\max\left\{r,s\right\}$. Then $|A_{i}|\le r^{r}$ and $\log{n}\ll \log{m}$. Also, $G$ is a large subgroup of a wreath product $R\wr S$, where $R$ is primitive of degree $r$, and $S$ is transitive of degree $s$. Thus, $a(R)\ll \log{r}$ by Theorem \ref{pyber1}. Thus, Corollary \ref{UsefulCor} implies that $d(G)\le s\log{r}/\sqrt{\log{s}}$, and hence 
$$f_{i}(G)\ll \frac{(\log{r})sr\log{r}\sqrt{\log{m}}}{r^{2}s\sqrt{\log{s}}}=\frac{(\log{r})^{2}\sqrt{\log{m}}}{r\sqrt{\log{s}}}\le \frac{(\log{r})^{5/2}}{r}\ll 1.$$
This completes the proof.\end{proof}

We conclude with an example which shows that the bounds of Theorems \ref{TransTheorem} and \ref{TransOrderTheorem} are asymptotically best possible.
\begin{Example}\label{TransExample} Let $A$ be an elementary abelian group of order $2^{2k-1}$, and write $R$ for the radical of the group algebra $\mathbb{F}_{2}[A]$. Consider the $2$-group $G:=R^{k-1}\rtimes A$.

The largest trivial submodule of $\mathbb{F}_{2}[A]$ is $1$-dimensional, while $\dim{(R^{k-1})}>1$, by \cite[3.2]{KovNew}. Hence, the centraliser $C_{A}(R^{k-1})$ of $R^{k-1}$ in $A$ is a proper characteristic subgroup of $A$; since $A$ is characteristically simple, it follows that $C_{A}(R^{k-1})=1$. Thus, $C_{G}(R^{k-1})=R^{k-1}$, so $Z:=Z(G)=C_{R^{k-1}}(A)$. Again, since the largest trivial submodule of $\mathbb{F}_{p}[A]$ is $1$-dimensional, and $Z$ is nontrivial, it follows that $Z$ has order $2$, and hence $Z$ is the unique minimal normal subgroup of $G$. Let $H$ be a subspace complement to $Z$ in $R^{k-1}$. Then $H$ has codimension $1$ in $R^{k-1}$, and hence has index $2^{2k}$ in $G$. It is also clear that $H$ is core-free in $G$, so $G$ is a transitive permutation group of degree $2^{2k}$.

Next, note that $$\sqrt{2k}\binom{2k}{k}\frac{1}{4^{k}}=\left[\frac{1}{2}\left(\frac{3}{2}\frac{3}{4}\right)\left(\frac{5}{4}\frac{5}{6}\right)\hdots\left(\frac{2k-1}{2k-2}\frac{2k-1}{2k}\right)\right]^{1/2}=\left[\frac{1}{2}\prod_{j=2}^{k}\left(1+\frac{1}{4j(j-1)}\right)\right]^{1/2}.$$ As in the proof of Theorem \ref{usefulposet}, the expression in the middle converges to $b=\sqrt{2/\pi}$, by Wallis' formula. Hence, since the expression on the right is increasing, we conclude that for all $\epsilon>0$, there exists a positive integer $k$ such that $\sqrt{2k}\binom{2k}{k}\frac{1}{4^{k}}\ge b-\epsilon$, that is, $\binom{2k}{k}\ge (b-\epsilon)4^{k}/\sqrt{2k}$.

Now, the derived subgroup $G'$ of $G$ is $R^{k}$, and $G/G'\cong (R^{k-1}/R^{k})\times A$ is elementary abelian of rank $\binom{2k-1}{k-1}+2k-1$, again using \cite[3.2]{KovNew}. Thus, for large enough $k$ we have
\begin{align*}d(G)= \binom{2k-1}{k-1}+2k-1 =\frac{1}{2}\binom{2k}{k}+2k-1\ge \frac{(b-\epsilon )2^{2k}}{2\sqrt{2k}}+2k-1.\end{align*}
Furthermore, $|R^{k-1}|=2^{\sum_{i=k-1}^{2k-1}\binom{2k-1}{i}}=2^{2^{2k-1}-2^{k-2}}\sim 2^{n/2}$. Hence, $|G|\sim 2^{n-1}$, which shows that $d(G)\log{|G|}$ is at least a constant times $n^{2}/\sqrt{\log{n}}$. 
\end{Example}

\newpage \begin{appendices}
\section{Upper bounds for $d(G)$ for some transitive groups of small degree}
The groups $G$ in the right hand column of Table A.1 below are transitive permutation groups of degree $d$, where $d$ is as specified in the left hand column. In Table A.2, the groups are transitive permutation groups of degree $d$ which have at least $f$ $2$-blocks (see Section 1).
  
  \begin{tabular}[t]{|c|p{1.5cm}|}
  \hline
  \multicolumn{2}{ |c| }{Table A.1}\\ 
  \hline
  $d$ & $d(G)\le$ \\
  \hline\hline
  $48$ &  $16$ \\
  \hline
  $64$ &  $20$ \\
  \hline
  $96$ &  $31$ \\
  \hline
  $128$ &  $40$ \\
  \hline
  $192$ &  $57$ \\
  \hline
  $256$ &  $75$ \\
  \hline
  $384$ &  $109$ \\
  \hline
  $512$ &  $145$ \\
  \hline 
  $2^{8}3$ &  $203$ \\
  \hline
  $2^{10}$ &  $271$ \\
  \hline
  $2^{9}3$ &  $392$ \\
  \hline
  $2^{11}$ &  $523$ \\
  \hline
  $2^{10}3$ &  $738$ \\
  \hline
   \end{tabular} \quad
  \begin{tabular}[t]{|c|p{1.5cm}|}
  \hline
  \multicolumn{2}{ |c| }{Table A.1 ctd}\\ 
  \hline
  $d$ & $d(G)\le$ \\
  \hline\hline
  $2^{11}3$ & $1431$ \\
  \hline
  $2^{12}3$ & $2718$ \\
  \hline
  $2^{13}3$ & $5292$ \\
  \hline
  $2^{14}3$ & $10118$ \\
  \hline
  $2{15}3$ & $19770$ \\
  \hline
  $2^{16}3$ & $38002$ \\
  \hline
  $2^{17}3$ & $74467$ \\
  \hline
  $2^{18}3$ & $143750$ \\
  \hline
  $2^{19}3$ & $282317$ \\
  \hline
  $2^{20}3$ & $546854$ \\
  \hline 
  $2^{3}5$ & $9$  \\
  \hline
   $2^{4}5$ & $18$  \\
  \hline
   $2^{5}5$ & $34$ \\
  \hline \end{tabular}\quad
 \begin{tabular}[t]{|c|p{1.5cm}|}
  \hline
  \multicolumn{2}{ |c| }{Table A.1 ctd}\\ 
  \hline
  $d$  & $d(G)\le$ \\
  \hline\hline
   $2^{6}5$  & $66$ \\
  \hline
   $2^{7}5$  & $130$ \\
  \hline
   $2^{8}5$  & $258$ \\
  \hline
   $2^{9}5$ & $514$ \\
  \hline
   $2^{10}5$  & $1026$ \\
  \hline
   $2^{11}5$  & $2050$ \\
  \hline
   $2^{12}5$  & $4098$ \\
  \hline
   $2^{13}5$  & $8194$ \\
  \hline
   $2^{14}5$  & $16386$ \\ \hline
   $2^{15}5$  & $32770$ \\
  \hline
   $2^{16}5$  & $65538$ \\
  \hline
   $2^{2}15$  & $15$  \\
  \hline   \end{tabular} \quad
   \begin{tabular}[t]{|c|p{1.5cm}|}
  \hline
  \multicolumn{2}{ |c| }{Table A.1 ctd}\\ 
  \hline
  $d$ & $d(G)\le $ \\
  \hline\hline
  $2^{3}15$ & $27$  \\
  \hline
   $2^{4}15$ & $52$  \\
  \hline
   $2^{5}15$ & $100$ \\
  \hline
   $2^{6}15$ & $196$ \\
  \hline
   $2^{7}15$ & $388$ \\
  \hline
   $2^{8}15$ & $772$ \\
  \hline
   $2^{9}15$ & $1540$ \\
  \hline
   $2^{10}15$ & $3076$ \\
  \hline
   $2^{11}15$ & $6148$ \\
  \hline
   $2^{12}15$ & $12292$ \\
  \hline
   $2^{13}15$ & $24580$ \\
  \hline
   $2^{14}15$ & $49156$ \\
  \hline \end{tabular}
   
\vspace{5mm}
   \begin{tabular}[t]{|c|p{0.5cm}|p{1.5cm}|}
  \hline
  \multicolumn{3}{ |c| }{Table A.2}\\ 
  \hline
  $d$ & $f$ & $d(G)\le $ \\
  \hline\hline
   $2^{17}5$ & $5$ & $130900$ \\
  \hline
   $2^{18}5$ & $4$ & $257722$ \\
  \hline
   $2^{19}5$ & $4$ & $504220$ \\
  \hline
   $2^{20}5$ & $4$ & $984067$ \\
  \hline
   $2^{21}5$ & $4$ & $1919461$ \\
  \hline
   $2^{22}5$ & $4$ & $3745164$ \\
  \hline
   $2^{23}5$ & $5$ & $7312620$ \\
  \hline
   $2^{24}5$ & $5$ & $14290701$ \\
  \hline
   $2^{25}5$ & $6$ & $27953017$ \\
  \hline
   $2^{26}5$ & $7$ & $54725580$ \\
  \hline
   $2^{15}15$ & $6$& $98308$ \\
  \hline \end{tabular}\quad
  \begin{tabular}[t]{|c|p{0.5cm}|p{1.5cm}|}
  \hline
  \multicolumn{3}{ |c| }{Table A.2 ctd}\\ 
  \hline
  $d$ & $f$ & $d(G)\le$ \\
  \hline\hline
   $2^{16}15$ & $4$& $196612$ \\
  \hline
   $2^{17}15$ & $3$ & $392700$ \\
  \hline
   $2^{18}15$ & $3$ & $773166$ \\
  \hline
   $2^{19}15$ & $3$ & $1512660$ \\
  \hline
   $2^{20}15$ & $3$ & $2952202$ \\
  \hline
   $2^{21}15$ & $3$ & $5758386$ \\
  \hline
   $2^{22}15$ & $3$ & $11235497$ \\
  \hline
   $2^{23}15$ & $3$ & $21937865$ \\
  \hline
   $2^{24}15$ & $3$ & $42872110$ \\
  \hline
   $2^{25}15$ & $3$ & $83859059$ \\
  \hline
    \end{tabular}\quad
  \begin{tabular}[t]{|c|p{0.5cm}|p{2.0cm}|}
  \hline
  \multicolumn{3}{ |c| }{Table A.2 ctd}\\ 
  \hline
  $d$ & $f$ & $d(G)\le$ \\
  \hline\hline
  $2^{26}15$ & $4$ & $164176748$ \\
  \hline
  $2^{27}15$ & $4$ & $321692696$ \\
  \hline
   $2^{28}15$ & $4$ & $630835627$ \\
  \hline
   $2^{29}15$ & $4$ & $1237980292$ \\
  \hline
   $2^{30}15$ & $5$ & $2431149936$ \\
  \hline
   $2^{31}15$ & $5$ & $4777379825$ \\
  \hline
   $2^{32}15$ & $5$ & $9393534359$ \\
  \hline
   $2^{33}15$ & $6$ & $18480443646$ \\
  \hline
   $2^{34}15$ & $7$ & $36376783048$ \\
  \hline
   $2^{35}15$ & $8$ & $71639170628$ \\
  \hline
\end{tabular}
 \end{appendices}

\end{document}